\documentclass[nonblindrev]{informs3}

 \RequirePackage{tgtermes}
 \RequirePackage{newtxtext}
 \RequirePackage{newtxmath}
 \RequirePackage{bm}
 \RequirePackage{endnotes}
\OneAndAHalfSpacedXII %

\usepackage{algorithm}
\usepackage[noend]{algpseudocode}
\usepackage{tikz}
\usepackage{mathtools}
\usepackage{enumerate}

\usepackage{natbib}
 \bibpunct[, ]{(}{)}{,}{a}{}{,}%
 \def\bibfont{\small}%

\EquationsNumberedThrough %

\TheoremsNumberedThrough %

\allowdisplaybreaks

\begin{document}

\RUNAUTHOR{Morales}

\RUNTITLE{The Polyhedral Structure of Energy Storage Operations}

\TITLE{The Polyhedral Structure of Energy Storage Operations: Convex Hull and Valid Inequalities}

\ARTICLEAUTHORS{%
\AUTHOR{Juan M. Morales}
\AFF{OASYS research group, University of Málaga, Spain, \EMAIL{juan.morales@uma.es}}

} %

\ABSTRACT{%
Energy storage models rely on complementarity constraints to prevent simultaneous charging and discharging, which turn them into disjunctive, nonconvex programs. Their linear relaxations, widely used in power system operation and planning, often violate these constraints. This paper studies the polyhedral structure of the multi-period operational set of a storage unit with conversion losses. We characterize the extreme points of its convex hull and show that, over any sub-horizon, the maximum energy that can be charged, or discharged, during a subset of periods is a submodular function of that subset, computable by a greedy rule. This structure yields an exponential family of linear valid inequalities, which we strengthen with the continuous power variables, and for which we establish necessary and sufficient conditions to define facets of the convex hull. The family contains inequalities that strictly dominate the storage capacity inequalities available in the literature. We further develop a separation algorithm that computes all the coefficients of a cut in linear time and can be made exact over contiguous sets of periods, together with a lifting procedure for the cuts that fail a facet condition. In 2000 instances of energy arbitrage with negative prices, the linear relaxation strengthened with the separated cuts returns an optimal schedule, free of simultaneous charging and discharging, in every instance, mostly after a single round of separation, whereas the strongest linear relaxation in the literature charges and discharges simultaneously in 43\% to 77\% of them.
}%


\maketitle

\section{Introduction} \label{sec:intro}
Energy storage systems (ESSs) have become a key component of modern power systems, driven by the rapid growth of variable renewable generation and the need to balance supply and demand in real time. Consequently, models of the operation of storage are increasingly embedded in the optimization tools used for power system operation \citep{pozo2014unit,wen2015enhanced, li2015sufficient, lorca2016multistage, cobos2018robust, castillo2013profit, duan2016improved,garifi2020convex, chen2022battery} and planning \citep{yang2017joint,zhao2018using,zhang2018coordinated,piansky2025optimizing}.

{ Linear programming (LP) formulations are commonly preferred for modeling ESSs, as they do not increase the computational complexity of the underlying optimization problem. However, pure LP models may produce operationally unrealistic schedules, in which an individual battery unit charges and discharges within the same time interval \citep{arroyo2020use}. Although large storage plants composed of independent modules can operate different modules in opposite modes, a single unit cannot charge and discharge at the same time when its power set-points are constant within each period: such schedules are physically unrealizable and let the model dissipate energy through fictitious round-trip losses. To rule out this behavior at the device level, LP formulations must be supplemented with complementarity constraints, expressed either as bilinear equalities \citep{li2015sufficient,elsaadany2023battery} or with binary variables \citep{chen2013mpc, jabr2014robust, parisio2014model,POZO2022108565,erdinc2014smart, cobos2018robust}; either way, the problem becomes nonconvex.}

{ Several strategies have been explored to retain linearity without complementarity constraints. \citet{elsaadany2023battery} and \citet{nazir2021guaranteeing} develop convex inner approximations of the original feasible set. Every point of such an approximation is physically feasible, but the approximation may exclude feasible points that are optimal for the original problem. Outer relaxations and convex hull formulations, conversely, preserve all feasible operating points and therefore never cut off the optimum. Their drawback is that an optimal solution of the relaxation may lie in the relaxed region rather than at one of the nonconvex extreme points, and thus still charge and discharge simultaneously, particularly when the objective function does not single out an extreme point.}

Other authors have established sufficient conditions under which the LP relaxation of the storage model, without complementarity constraints, produces physically realizable schedules. These conditions depend on the optimization context (e.g., economic dispatch, unit commitment, or their variants). In some cases, they can be verified a priori~\citep{garifi2018control, wu2016distributed, lin2023relaxing, liang2024deriving, wang2024exact, prat2025sufficient}; in others, they depend on the solution of the relaxed problem, which limits their practical use~\citep{castillo2013profit, li2015sufficient, duan2016improved, li2018extended, garifi2020convex, chen2022battery}.

Recently, \citet{al2025lossy} have shown that a concave, piecewise-affine bijection maps the feasible power profiles onto the feasible state-of-charge (SoC) profiles, which yields a convex feasible set in the SoC variables alone. The nonconvexity, however, remains embedded in the mapping, and the conditions under which the objective remains convex are restrictive; notably, they do not hold for the arbitrage problem with negative prices considered in our experiments.

{ From a methodological standpoint, the polyhedral study of the multi-period operational set of a storage unit shares conceptual similarities with that of lot-sizing and inventory problems \citep{atamturk2005lot, pochet2006production}: both involve an intertemporal state variable subject to capacity limits. The operation of a battery, however, is genuinely different. In lot-sizing, the inventory is controlled in one direction through replenishment decisions, whereas its depletion is dictated by exogenous demand. In an ESS, both charging and discharging are endogenous decisions. This bidirectional control, together with lower and upper SoC limits, separate charging and discharging power limits, and distinct conversion efficiencies, changes the geometry of the nonconvex set and calls for a dedicated polyhedral analysis.}

{ Such an analysis was started by \citet{pozo2023convex}, who proposed linear valid inequalities that strengthen the linear relaxation of the storage set and reduce the incidence of simultaneous charging and discharging. These inequalities describe the convex hull of the feasible set restricted to a single period, and \citet{elgersma2024tight} extended them to reserve provision and investment decisions. Multi-period valid inequalities, which capture how the SoC links consecutive periods, have not been characterized so far.}

{ The single-unit problem with a linear objective is, in fact, polynomially solvable. \citet{bakhshi2023polynomial} show that, between two periods in which the storage is full or empty, an extreme point has at most one period at partial power, and they derive from this property a shortest-path algorithm and an extended formulation of the convex hull in a lifted space \citep[see also][]{bakhshi2024battery}; through the normalization of Section~\ref{sec:Formulation}, their model also covers conversion losses. As in lot sizing, where polynomial algorithms and compact extended formulations \citep{vanhoesel1996algorithm, pochet2006production} coexist with the $(l,S)$ inequalities of~\citet{barany1984uncapacitated}, which strengthen multi-item models \citep{barany1984strong}, an explicit description in the original space remains valuable when the unit is a component of larger models, all the more so because, as we discuss in Section~\ref{sec:extreme_points}, the size of their extended formulation depends on the arithmetic of the data and can grow exponentially with period-dependent limits or self-discharge.}

{ This paper builds upon that line of research and studies the polyhedral structure of the multi-period operational set of a storage unit with conversion losses. Its contributions are the following.
\begin{enumerate}[(i)]
    \item We characterize the extreme points of the convex hull of the operational set by a necessary and sufficient condition, which completes the necessary condition of~\citet{bakhshi2023polynomial}: between two periods in which the storage is full or empty, at most one period operates at partial power.
    \item We show that, over any sub-horizon and for any initial SoC, the maximum energy that can be charged (respectively, discharged) during a subset of periods is a submodular function of that subset, which a greedy rule computes.
    \item We turn this structure into an exponential family of linear valid inequalities for the mixed-integer formulation, strengthen them with the continuous power variables, and establish necessary and sufficient conditions under which they define facets of the convex hull. The family contains, for the sets of initial periods, inequalities that dominate the storage capacity inequalities of~\citet{pozo2023convex}, its members with gaps between periods have no counterpart in the literature, and its validity extends to period-dependent limits and to self-discharge.
    \item We develop a separation algorithm that computes all the coefficients of a cut in linear time, can be made exact over contiguous sets of periods, and handles sets with gaps through moves that never decrease the violation. We also derive in closed form the lifting coefficients that strengthen the cuts violating a facet condition.
\end{enumerate}
Although the family does not describe the convex hull completely, it is strong enough in practice: in our experiments, the linear relaxation strengthened with the separated cuts never charges and discharges simultaneously at termination, and therefore returns an optimal schedule of the mixed-integer problem.}

We evaluate the cuts on the problem of scheduling a battery for energy arbitrage in the presence of negative prices, which frequently yields simultaneous charging and discharging under a linear relaxation~\citep{prat2025sufficient}. In 2000 instances with horizons of 24 and 72 hours, built from Danish market prices and the 100 battery configurations of~\citet{pozo2023convex}, the strongest relaxation in the literature charges and discharges simultaneously in 43\% to 77\% of the instances, whereas the relaxation strengthened with our cuts eliminates simultaneous charging and discharging in all of them, mostly after a single round of separation. For a single battery the problem is easy: a modern mixed-integer solver solves these instances at the root node, and the shortest-path algorithm of~\citet{bakhshi2023polynomial} is polynomial. The experiments therefore measure the quality of the linear description, which is what matters when the storage is one component of a larger model or when the model must remain linear, as in bilevel, equilibrium, and pricing problems. { Convexifications that exploit a specific objective function, such as the tracking error in set-point tracking problems, are of a different nature and fall outside the scope of this paper.}

The rest of this paper is organized as follows. Section~\ref{sec:Formulation} introduces the operational set of an ESS, its mixed-integer formulation, and its natural linear relaxation, and characterizes the extreme points of its convex hull. Section~\ref{sec:LVI} establishes the submodular structure of the storage capacity, derives the valid inequalities and their facet-defining conditions, compares them with the capacity inequalities of the literature, and presents the separation and strengthening procedures. Section~\ref{sec:case_study} reports the numerical experiments, and Section~\ref{sec:conclusions} concludes. The appendices are organized by the sections of the paper that it complements; it contains the proof of the facet characterization, the remaining proofs of Section~\ref{sec:LVI}, illustrative examples, and the expression of the inequalities in physical units.

\section{Operational Feasibility Set} \label{sec:Formulation}
We begin by introducing the standard nonconvex operational feasibility set~$\mathcal{S}$ of an energy storage system (ESS), its mixed-integer formulation~$\mathcal{S}^{0\text{-}1}$, and the natural linear relaxation~$\mathcal{S}^{\mathcal{R}}$ of the latter. We then project out the state-of-charge variables and characterize the extreme points of the resulting convex hull.

\paragraph{Notation.} For integers $l \leq k$, we write $[l,k] := \{l, l+1, \ldots, k\}$, with $[l,k] := \emptyset$ if $l > k$, and $[T] := [1,T]$. For $\boldsymbol{x} \in \mathbb{R}^T$ and $\Omega \subseteq [T]$, we let $x(\Omega) := \sum_{t \in \Omega} x_t$, and $x^+ := \max\{x, 0\}$, $x^- := \max\{-x, 0\}$ for $x \in \mathbb{R}$. The vector of all ones is denoted by $\boldsymbol{1}$, and $\boldsymbol{e}_t$ denotes the $t$-th unit vector of $\mathbb{R}^T$ ($\boldsymbol{e}_{c_t}$, $\boldsymbol{e}_{d_t}$, $\boldsymbol{e}_{u_t}$ for the coordinates of $(\boldsymbol{d}, \boldsymbol{c}, \boldsymbol{u})$). Inequalities between vectors are understood componentwise.

\subsection{Formulations}

Consider a horizon of $T$ periods of length $\Delta > 0$. For each $t \in [T]$, let $d_t \geq 0$ and $c_t \geq 0$ denote the discharging and charging power of the ESS during period~$t$, and let $s_t$ denote its state of charge (SoC) at the end of period~$t$. The operational feasibility set is
\begin{subequations}\label{eq:Sset}
  \begin{align}
      \mathcal{S} = \Big\{&(\boldsymbol{d}, \boldsymbol{c}) \in \mathbb{R}^{2T}_{\geq 0},\ \boldsymbol{s} \in \mathbb{R}^{T}: \notag\\
      & \quad \boldsymbol{d} \leq \overline{d}\,\boldsymbol{1}, \label{eq:S_max_dis}\\
      & \quad \boldsymbol{c} \leq \overline{c}\,\boldsymbol{1}, \label{eq:S_max_ch}\\
      & \quad \underline{s}\,\boldsymbol{1} \leq \boldsymbol{s} \leq \overline{s}\,\boldsymbol{1}, \label{eq:S_soc_limits}\\
      & \quad s_t = s_{t-1} + \Delta\big(\eta_c\, c_t - d_t/\eta_d\big), \quad t \in [T], \label{eq:S_soc_dynamics}\\
      & \quad d_t\, c_t = 0, \quad t \in [T] \Big\}. \label{eq:S_complementarity}
  \end{align}
\end{subequations}
Here, $\overline{d} > 0$ and $\overline{c} > 0$ are the maximum discharging and charging power, respectively, and $0 \leq \underline{s} < \overline{s}$ are the minimum and maximum energy that can be stored. Equation~\eqref{eq:S_soc_dynamics} describes the evolution of the SoC, where $\eta_d, \eta_c \in (0,1]$ are the discharging and charging efficiencies, and the initial SoC $s_0 \in [\underline{s}, \overline{s}]$ is given. The complementarity constraints~\eqref{eq:S_complementarity} prevent the ESS from charging and discharging simultaneously within the same period. They make $\mathcal{S}$ a disjunctive set, whose projection onto the $(\boldsymbol{d}, \boldsymbol{c}, \boldsymbol{s})$-space coincides with that of the mixed-integer set~\citep{pozo2023convex}
\begin{subequations}\label{eq:Sset_MIP}
  \begin{align}
      \mathcal{S}^{0\text{-}1} = \Big\{&(\boldsymbol{d}, \boldsymbol{c}) \in \mathbb{R}^{2T}_{\geq 0},\ \boldsymbol{s} \in \mathbb{R}^{T},\ \boldsymbol{u} \in \{0,1\}^{T}: \notag\\
      & \quad \boldsymbol{d} \leq \overline{d}\,(\boldsymbol{1} - \boldsymbol{u}), \label{eq:S01_max_dis}\\
      & \quad \boldsymbol{c} \leq \overline{c}\,\boldsymbol{u}, \label{eq:S01_max_ch}\\
      & \quad \underline{s}\,\boldsymbol{1} \leq \boldsymbol{s} \leq \overline{s}\,\boldsymbol{1}, \label{eq:S01_soc_limits}\\
      & \quad s_t = s_{t-1} + \Delta\big(\eta_c\, c_t - d_t/\eta_d\big), \quad t \in [T] \Big\}. \label{eq:S01_soc_dynamics}
  \end{align}
\end{subequations}
The binary variable $u_t$ selects the operating mode in period~$t$: $u_t = 1$ allows charging and $u_t = 0$ allows discharging. If $d_t = c_t = 0$, both values of $u_t$ are feasible.

\paragraph{Normalization.} The change of variables and parameters
\[
c_t \leftarrow \Delta \eta_c\, c_t, \qquad d_t \leftarrow \frac{\Delta}{\eta_d}\, d_t, \qquad \overline{c} \leftarrow \Delta \eta_c\, \overline{c}, \qquad \overline{d} \leftarrow \frac{\Delta}{\eta_d}\, \overline{d}
\]
turns~\eqref{eq:S_soc_dynamics} and~\eqref{eq:S01_soc_dynamics} into $s_t = s_{t-1} + c_t - d_t$ and leaves every other constraint of $\mathcal{S}$ and $\mathcal{S}^{0\text{-}1}$, including~\eqref{eq:S_complementarity}, unchanged in form. Being a positive diagonal scaling of the coordinates, it is a linear bijection that maps convex hulls onto convex hulls and preserves extreme points, faces, and facets. In particular, an inequality $(\boldsymbol{\theta}^d)^\top \boldsymbol{d} + (\boldsymbol{\theta}^c)^\top \boldsymbol{c} + (\boldsymbol{\theta}^u)^\top \boldsymbol{u} + (\boldsymbol{\theta}^s)^\top \boldsymbol{s} \leq \theta_0$ is valid (respectively, facet-defining) for the scaled set if and only if $\frac{\Delta}{\eta_d}(\boldsymbol{\theta}^d)^\top \boldsymbol{d} + \Delta\eta_c\,(\boldsymbol{\theta}^c)^\top \boldsymbol{c} + (\boldsymbol{\theta}^u)^\top \boldsymbol{u} + (\boldsymbol{\theta}^s)^\top \boldsymbol{s} \leq \theta_0$ is valid (respectively, facet-defining) for the original one. Therefore, we assume without loss of generality that $\Delta = \eta_c = \eta_d = 1$ hereinafter. Under this normalization, $c_t$ and $d_t$ measure the energy injected into and withdrawn from the storage during period~$t$; physical power values are recovered by applying the inverse scaling.

\paragraph{Effective power limits.} Define the \emph{effective maximum discharge and charge rates}
\[
\overline{d}_e := \min\{\overline{d},\, \overline{s} - \underline{s}\} \leq \overline{d}, \qquad \overline{c}_e := \min\{\overline{c},\, \overline{s} - \underline{s}\} \leq \overline{c}.
\]
In any point of $\mathcal{S}^{0\text{-}1}$, if $u_t = 0$, then $c_t = 0$ and $d_t = s_{t-1} - s_t \leq \overline{s} - \underline{s}$, whereas if $u_t = 1$, then $d_t = 0$. Hence, $d_t \leq \overline{d}_e (1 - u_t)$ holds for all $t \in [T]$ and, symmetrically, $c_t \leq \overline{c}_e\, u_t$. The same argument applies to $\mathcal{S}$ by virtue of~\eqref{eq:S_complementarity}. Replacing $\overline{d}$ and $\overline{c}$ with $\overline{d}_e$ and $\overline{c}_e$ thus leaves $\mathcal{S}$ and $\mathcal{S}^{0\text{-}1}$ unaltered. In contrast, it may tighten the linear relaxation introduced below, because a fractional $u_t$ allows $c_t$ and $d_t$ to be simultaneously positive and the bound $\overline{s} - \underline{s}$ is then no longer implied. We assume hereinafter that $\overline{d} = \overline{d}_e$ and $\overline{c} = \overline{c}_e$.

\paragraph{Linear relaxation and projection.} The natural linear relaxation of $\mathcal{S}^{0\text{-}1}$ is
\begin{equation}\label{eq:Prelax}
\mathcal{S}^{\mathcal{R}} = \Big\{(\boldsymbol{d}, \boldsymbol{c}) \in \mathbb{R}^{2T}_{\geq 0},\ \boldsymbol{s} \in \mathbb{R}^{T},\ \boldsymbol{u} \in [0,1]^{T}: \eqref{eq:S01_max_dis}\text{--}\eqref{eq:S01_soc_dynamics}\Big\}.
\end{equation}
We define the \emph{net-charging} vector
\begin{equation}\label{eq:x_def}
\boldsymbol{x} := \boldsymbol{c} - \boldsymbol{d}.
\end{equation}
For every point of $\mathcal{S}^{0\text{-}1}$, $x_t \in [-\overline{d}, \overline{c}]$, $c_t = x_t^+$, and $d_t = x_t^-$ for all $t \in [T]$. Moreover, the SoC dynamics read $s_t = s_0 + x([1,t])$, that is, $\boldsymbol{s} = s_0\,\boldsymbol{1} + \boldsymbol{H}\boldsymbol{x}$, where $\boldsymbol{H} \in \mathbb{R}^{T \times T}$ is the lower triangular matrix of ones. Eliminating $\boldsymbol{s}$ from~\eqref{eq:Sset_MIP} yields the set
\begin{subequations}\label{eq:Pset_MIP}
  \begin{align}
      \mathcal{P}^{0\text{-}1} = \Big\{&(\boldsymbol{d}, \boldsymbol{c}) \in \mathbb{R}^{2T}_{\geq 0},\ \boldsymbol{u} \in \{0,1\}^{T}: \notag\\
      & \quad \boldsymbol{d} \leq \overline{d}\,(\boldsymbol{1} - \boldsymbol{u}), \label{eq:P01_max_dis}\\
      & \quad \boldsymbol{c} \leq \overline{c}\,\boldsymbol{u}, \label{eq:P01_max_ch}\\
      & \quad \underline{s}\,\boldsymbol{1} \leq s_0\,\boldsymbol{1} + \boldsymbol{H}(\boldsymbol{c} - \boldsymbol{d}) \leq \overline{s}\,\boldsymbol{1} \Big\}. \label{eq:P01_soc_limits}
  \end{align}
\end{subequations}
We denote the points of $\mathcal{P}^{0\text{-}1}$ by $\boldsymbol{y} = (\boldsymbol{d}, \boldsymbol{c}, \boldsymbol{u}) \in \mathbb{R}^{3T}$, and by $\boldsymbol{s}(\boldsymbol{y}) := s_0\,\boldsymbol{1} + \boldsymbol{H}(\boldsymbol{c} - \boldsymbol{d})$ the associated SoC trajectory.

\begin{remark}\label{rem:lifting}
The affine map $\Lambda : (\boldsymbol{d}, \boldsymbol{c}, \boldsymbol{u}) \mapsto \big(\boldsymbol{d}, \boldsymbol{c}, s_0\,\boldsymbol{1} + \boldsymbol{H}(\boldsymbol{c} - \boldsymbol{d}), \boldsymbol{u}\big)$ is a bijection from $\mathcal{P}^{0\text{-}1}$ onto $\mathcal{S}^{0\text{-}1}$, whose inverse drops $\boldsymbol{s}$. Since affine maps commute with convex hulls,
\[
\operatorname{conv}(\mathcal{S}^{0\text{-}1}) = \Lambda\big(\operatorname{conv}(\mathcal{P}^{0\text{-}1})\big) = \Big\{(\boldsymbol{d}, \boldsymbol{c}, \boldsymbol{s}, \boldsymbol{u}) : (\boldsymbol{d}, \boldsymbol{c}, \boldsymbol{u}) \in \operatorname{conv}(\mathcal{P}^{0\text{-}1}),\ \boldsymbol{s} = s_0\,\boldsymbol{1} + \boldsymbol{H}(\boldsymbol{c} - \boldsymbol{d})\Big\},
\]
and $\Lambda$ maps the faces, facets, and extreme points of $\operatorname{conv}(\mathcal{P}^{0\text{-}1})$ one-to-one onto those of $\operatorname{conv}(\mathcal{S}^{0\text{-}1})$. Hence an inequality in $(\boldsymbol{d}, \boldsymbol{c}, \boldsymbol{u})$ is valid (respectively, facet-defining) for $\operatorname{conv}(\mathcal{P}^{0\text{-}1})$ if and only if it is so for $\operatorname{conv}(\mathcal{S}^{0\text{-}1})$, and substituting $\boldsymbol{s} = s_0\,\boldsymbol{1} + \boldsymbol{H}(\boldsymbol{c} - \boldsymbol{d})$ turns any inequality valid (respectively, facet-defining) for $\operatorname{conv}(\mathcal{S}^{0\text{-}1})$ into one valid (respectively, facet-defining) for $\operatorname{conv}(\mathcal{P}^{0\text{-}1})$; no generality is lost by studying $\mathcal{P}^{0\text{-}1}$. Since every point of $\operatorname{conv}(\mathcal{S}^{0\text{-}1})$ satisfies the $T$ SoC equations, this polytope is not full-dimensional, and adding to an inequality any linear combination of these equations changes its coefficients but not its face; for instance, $s_1 \leq \overline{s}$ and $c_1 - d_1 \leq \overline{s} - s_0$ define the same face.
\end{remark}

\subsection{Extreme Point Characterization}
\label{sec:extreme_points}

For each $\bar{\boldsymbol{u}} \in \{0,1\}^T$, let $\mathcal{P}_{\bar{\boldsymbol{u}}} := \{(\boldsymbol{d}, \boldsymbol{c}, \boldsymbol{u}) \in \mathcal{P}^{0\text{-}1} : \boldsymbol{u} = \bar{\boldsymbol{u}}\}$ be the polytope obtained by fixing the operating modes, so that $\mathcal{P}^{0\text{-}1} = \bigcup_{\bar{\boldsymbol{u}} \in \{0,1\}^T} \mathcal{P}_{\bar{\boldsymbol{u}}}$. The following lemma shows that these polytopes are exactly the faces of $\operatorname{conv}(\mathcal{P}^{0\text{-}1})$ on which $\boldsymbol{u}$ is fixed.

\begin{lemma}\label{lem:binary_faces}
For every $\bar{\boldsymbol{u}} \in \{0,1\}^T$, the set $\operatorname{conv}(\mathcal{P}^{0\text{-}1}) \cap \{\boldsymbol{u} = \bar{\boldsymbol{u}}\}$ is a face of $\operatorname{conv}(\mathcal{P}^{0\text{-}1})$ and equals $\mathcal{P}_{\bar{\boldsymbol{u}}}$. Consequently, $\operatorname{conv}(\mathcal{P}^{0\text{-}1})$ is a polytope and
\[
\operatorname{ext}\big(\operatorname{conv}(\mathcal{P}^{0\text{-}1})\big) = \bigcup_{\bar{\boldsymbol{u}} \in \{0,1\}^T} \operatorname{ext}(\mathcal{P}_{\bar{\boldsymbol{u}}}).
\]
\end{lemma}

The proof, given in Appendix~\ref{ec:proof_binary_faces}, observes that $\{\boldsymbol{u} = \bar{\boldsymbol{u}}\}$ is the set where the valid inequalities $0 \leq u_t \leq 1$ that are tight at $\bar{\boldsymbol{u}}$ hold with equality, and that a convex combination of points of $\mathcal{P}^{0\text{-}1}$ can only have $u_t \in \{0,1\}$ if all of them have the same $u_t$.

Lemma~\ref{lem:binary_faces} reduces the study of the extreme points of $\operatorname{conv}(\mathcal{P}^{0\text{-}1})$ to the polytopes $\mathcal{P}_{\bar{\boldsymbol{u}}}$. Once $\bar{\boldsymbol{u}}$ is fixed, each period operates in a single mode: $x_t$ ranges over $[\underline{x}_t, \overline{x}_t]$, with $\underline{x}_t := -\overline{d}\,(1 - \bar{u}_t)$ and $\overline{x}_t := \overline{c}\,\bar{u}_t$. Since this interval lies on one side of zero, the relation between $x_t$ and the power variables is linear: $(d_t, c_t) = (0, x_t)$ if $\bar{u}_t = 1$, and $(d_t, c_t) = (-x_t, 0)$ if $\bar{u}_t = 0$. Therefore, $\mathcal{P}_{\bar{\boldsymbol{u}}}$ is the image of the polytope
\begin{equation}\label{eq:X_u}
\mathcal{X}_{\bar{\boldsymbol{u}}} := \Big\{\boldsymbol{x} \in \mathbb{R}^T : \underline{x}_t \leq x_t \leq \overline{x}_t, \enskip \underline{s} \leq s_0 + x([1,t]) \leq \overline{s}, \enskip t \in [T]\Big\}
\end{equation}
under an injective linear map, and the extreme points of both polytopes are in one-to-one correspondence.

To state the characterization, fix $\boldsymbol{y} \in \mathcal{P}^{0\text{-}1}$ with SoC trajectory $\boldsymbol{s} = \boldsymbol{s}(\boldsymbol{y})$. We call a period $t \in [T]$ \emph{tight} if $s_t \in \{\underline{s}, \overline{s}\}$, and a \emph{partial-power period} if $0 < c_t < \overline{c}$ or $0 < d_t < \overline{d}$, which is equivalent to $\underline{x}_t < x_t < \overline{x}_t$. Let $t_1 < \cdots < t_m$ be the tight periods of $\boldsymbol{y}$ and set $t_0 := 0$. These periods split the horizon into the \emph{segments} $I_r := [t_{r-1}+1, t_r]$, $r \in [m]$, and the (possibly empty) \emph{tail} $I_{m+1} := [t_m + 1, T]$.

\begin{proposition}[Extreme Point Structure]\label{prop:extreme_points}
A point $\boldsymbol{y} = (\boldsymbol{d}, \boldsymbol{c}, \boldsymbol{u}) \in \mathcal{P}^{0\text{-}1}$ is an extreme point of $\operatorname{conv}(\mathcal{P}^{0\text{-}1})$ if and only if
\begin{enumerate}[(a)]
    \item every segment $I_r$, $r \in [m]$, contains at most one partial-power period, and
    \item the tail $I_{m+1}$ contains no partial-power period.
\end{enumerate}
\end{proposition}

\begin{proof}{Proof}
By Lemma~\ref{lem:binary_faces} and the correspondence above, $\boldsymbol{y}$ is extreme in $\operatorname{conv}(\mathcal{P}^{0\text{-}1})$ if and only if its net-charging vector $\boldsymbol{x}$ is extreme in $\mathcal{X}_{\boldsymbol{u}}$.

\emph{Necessity.} Suppose that a segment $I_r$, $r \in [m]$, contains two partial-power periods $i < j$, and let $\boldsymbol{\delta} := \boldsymbol{e}_i - \boldsymbol{e}_j$. For $\varepsilon > 0$, the points $\boldsymbol{x} \pm \varepsilon \boldsymbol{\delta}$ have the same SoC as $\boldsymbol{x}$ except in periods $t \in [i, j-1] \subseteq [t_{r-1}+1, t_r - 1]$, which are not tight, where the SoC changes by $\pm\varepsilon$. Since $x_i$ and $x_j$ are interior to their bounds, both points belong to $\mathcal{X}_{\boldsymbol{u}}$ for $\varepsilon$ small enough, and $\boldsymbol{x}$ is their midpoint. Similarly, if the tail contains a partial-power period $i$, then $\boldsymbol{x} \pm \varepsilon \boldsymbol{e}_i$ only changes the SoC in periods $t \geq i$, none of which is tight. In both cases, $\boldsymbol{x}$ is not extreme.

\emph{Sufficiency.} Assume (a) and (b), and let $\boldsymbol{x} = \frac{1}{2}(\boldsymbol{x}' + \boldsymbol{x}'')$ with $\boldsymbol{x}', \boldsymbol{x}'' \in \mathcal{X}_{\boldsymbol{u}}$; we show that $\boldsymbol{x}' = \boldsymbol{x}''$. The argument rests on an elementary fact: if two numbers do not exceed a bound and their average attains it, both attain it. First, if period $t$ is not partial-power, then $x_t \in \{\underline{x}_t, \overline{x}_t\}$; since $x'_t, x''_t \in [\underline{x}_t, \overline{x}_t]$ average to $x_t$, we get $x'_t = x''_t = x_t$. Second, the SoC is affine in $\boldsymbol{x}$, so the SoC trajectory of $\boldsymbol{x}$ is the average of those of $\boldsymbol{x}'$ and $\boldsymbol{x}''$, i.e., $s_t = \frac{1}{2}(s'_t + s''_t)$ for all $t$. At a tight period $t_r$, $r \in [m]$, we have $s_{t_r} \in \{\underline{s}, \overline{s}\}$ while $s'_{t_r}, s''_{t_r} \in [\underline{s}, \overline{s}]$, hence $s'_{t_r} = s''_{t_r} = s_{t_r}$; at $t_0 = 0$, all three trajectories start from the same $s_0$. Consequently, for every segment $r \in [m]$,
\[
x'(I_r) = x''(I_r) = x(I_r) = s_{t_r} - s_{t_{r-1}}.
\]
By (a), all coordinates in $I_r$ except at most one are already known to coincide, and this equation forces the remaining one to coincide as well. By (b), every period in the tail is not partial-power, so its coordinate is also fixed. Thus, $\boldsymbol{x}' = \boldsymbol{x}'' = \boldsymbol{x}$, and $\boldsymbol{x}$ is extreme.
\hfill$\square$
\end{proof}

\begin{corollary}\label{cor:extreme_points}
Let $\boldsymbol{y}$ be an extreme point of $\operatorname{conv}(\mathcal{P}^{0\text{-}1})$ with $m$ tight periods. Then:
\begin{enumerate}[(i)]
    \item $\boldsymbol{y}$ has at most $m$ partial-power periods. In particular, if the SoC never reaches its limits, then $c_t \in \{0, \overline{c}\}$ and $d_t \in \{0, \overline{d}\}$ for all $t \in [T]$.
    \item If $t^\star \in I_r$ is a partial-power period, then $x_{t^\star} = s_{t_r} - s_{t_{r-1}} - x(I_r \setminus \{t^\star\})$. Hence, every coordinate of $(\boldsymbol{d}, \boldsymbol{c})$ is an integer combination of $s_0$, $\underline{s}$, $\overline{s}$, $\overline{c}$, and $\overline{d}$.
\end{enumerate}
\end{corollary}

\begin{proof}{Proof}
Part (i) follows from Proposition~\ref{prop:extreme_points}, since there are $m$ segments and the tail contains no partial-power period. For part (ii), the identity follows from $x(I_r) = s_{t_r} - s_{t_{r-1}}$, where $s_{t_0} = s_0$ and $s_{t_r} \in \{\underline{s}, \overline{s}\}$ for $r \in [m]$; all other coordinates of $\boldsymbol{x}$ lie in $\{-\overline{d}, 0, \overline{c}\}$, and $(\boldsymbol{d}, \boldsymbol{c}) = (\boldsymbol{x}^-, \boldsymbol{x}^+)$.
\hfill$\square$
\end{proof}

\begin{remark}\label{rem:extreme_points_role}
Proposition~\ref{prop:extreme_points} and its proof hold verbatim on any sub-horizon $[l,k]$ with a given initial SoC $s_{l-1}$, replacing $s_0$ and $[T]$ with $s_{l-1}$ and $[l,k]$. Two consequences are used below. First, the greedy net-charging sequences of Lemma~\ref{lem:greedy_polymatroid}, which attain the set functions underlying our valid inequalities, place every period at a power bound or at a tight SoC, so they are extreme points. Second, this structure is the one exploited by~\citet{bakhshi2023polynomial}, who were the first to use it to solve the problem with a linear objective as a shortest-path problem and to derive from it an extended formulation of $\operatorname{conv}(\mathcal{P}^{0\text{-}1})$. By the proof of Corollary~\ref{cor:extreme_points}, the SoC of an extreme point only visits the levels $a + n_c\,\overline{c} - n_d\,\overline{d}$, $a \in \{s_0, \underline{s}, \overline{s}\}$, before the partial-power period of a segment, and $a - n_c\,\overline{c} + n_d\,\overline{d}$, $a \in \{\underline{s}, \overline{s}\}$, from that period on, with integers $n_c, n_d \geq 0$ and $n_c + n_d \leq T$. Their extended formulation is a network over these levels, whose size depends on the arithmetic of the data and can grow exponentially with period-dependent limits or self-discharge. The inequalities of Section~\ref{sec:LVI} include facets of the same polytope, written in the original space, and do not rely on this lattice.
\end{remark}

\section{Linear Valid Inequalities} \label{sec:LVI}
In this section, we derive a family of linear valid inequalities for $\mathcal{P}^{0\text{-}1}$. They are formulated in the $(\boldsymbol{d}, \boldsymbol{c}, \boldsymbol{u})$-space and, by Remark~\ref{rem:lifting}, are equally valid for $\mathcal{S}^{0\text{-}1}$. Their validity rests on a submodular structure: over any sub-horizon, the maximum energy that can be charged (respectively, discharged) during a subset of periods is a submodular function of that subset.

Section~\ref{sec:submodular_structure} establishes this structure, and Section~\ref{sec:submodular_cuts} derives the resulting cuts. Section~\ref{sec:cut_strengthening} strengthens them with the continuous power variables, and Section~\ref{sec:facet_defining_conditions} characterizes when they define facets of $\operatorname{conv}(\mathcal{P}^{0\text{-}1})$. Section~\ref{sec:dominance} shows that they contain, for the sets $\Omega = [1,t]$, inequalities that dominate the storage capacity inequalities of~\citet{pozo2023convex}, and Section~\ref{sec:separation_algorithm} presents a separation algorithm.

\paragraph{Preliminaries on set functions.} We recall some basic notions; see, e.g.,~\citet{fujishige2005submodular}.
\begin{definition}\label{def:submodularity}
Let $N$ be a finite ground set and $\varphi : 2^N \to \mathbb{R}$. For $A \subseteq N$ and $j \in N \setminus A$, the \emph{marginal value} of $j$ in context $A$ is
\[
\varphi(j \mid A) := \varphi(A \cup \{j\}) - \varphi(A).
\]
The function $\varphi$ is \emph{submodular} if marginal values do not increase as the context grows, i.e., $\varphi(j \mid A) \geq \varphi(j \mid B)$ for all $A \subseteq B \subseteq N$ and $j \in N \setminus B$. The following two conditions are each equivalent to this:
\begin{enumerate}[(i)]
    \item $\varphi(j \mid A) \geq \varphi(j \mid A \cup \{j'\})$ for all $A \subseteq N$ and all distinct $j, j' \in N \setminus A$;
    \item $\varphi(A) + \varphi(B) \geq \varphi(A \cup B) + \varphi(A \cap B)$ for all $A, B \subseteq N$.
\end{enumerate}
The function $\varphi$ is \emph{supermodular} if $-\varphi$ is submodular.
\end{definition}

The following bound, a special case of a classical inequality of~\citet{nemhauser1978analysis}, is the key step in turning submodularity into valid inequalities.
\begin{lemma}\label{lem:submodular_bound}
Let $\varphi : 2^N \to \mathbb{R}$ be submodular. Then, for all $U \subseteq \Omega \subseteq N$,
\begin{equation}\label{eq:submodular_bound}
\varphi(U) \leq \varphi(\Omega) - \sum_{j \in \Omega \setminus U} \varphi(j \mid \Omega \setminus \{j\}).
\end{equation}
\end{lemma}

\begin{proof}{Proof}
Write $\Omega \setminus U = \{j_1, \ldots, j_q\}$ in any order, and let $A_i := U \cup \{j_1, \ldots, j_{i-1}\}$ for $i \in [q+1]$, so that $A_1 = U$ and $A_{q+1} = \Omega$. Telescoping,
\[
\varphi(\Omega) - \varphi(U) = \sum_{i=1}^{q} \varphi(j_i \mid A_i) \geq \sum_{i=1}^{q} \varphi(j_i \mid \Omega \setminus \{j_i\}),
\]
where the inequality follows from submodularity, because $A_i \subseteq \Omega \setminus \{j_i\}$ and $j_i \notin \Omega \setminus \{j_i\}$. Rearranging gives~\eqref{eq:submodular_bound}.
\hfill$\square$
\end{proof}

\subsection{Submodular Structure of Storage Capacity}\label{sec:submodular_structure}

Throughout this subsection, we fix a window $[l,k] \subseteq [T]$ and an initial SoC $\sigma \in [\underline{s}, \overline{s}]$ at the end of period $l-1$. The window is omitted from the notation whenever no confusion arises; otherwise, it is indicated by a subscript, as in $f_{[l,k]}$ or $\mathcal{X}_{[l,k]}$. The \emph{sub-horizon set} is
\begin{multline}\label{eq:Pset_window}
\mathcal{P}_{[l,k]}(\sigma) := \Big\{(\boldsymbol{d}, \boldsymbol{c}) \in \mathbb{R}^{[l,k]}_{\geq 0} \times \mathbb{R}^{[l,k]}_{\geq 0},\ \boldsymbol{u} \in \{0,1\}^{[l,k]} : \boldsymbol{d} \leq \overline{d}\,(\boldsymbol{1} - \boldsymbol{u}),\ \boldsymbol{c} \leq \overline{c}\,\boldsymbol{u}, \\
\underline{s} \leq \sigma + x([l,t]) \leq \overline{s},\ t \in [l,k] \Big\},
\end{multline}
where $\boldsymbol{x} = \boldsymbol{c} - \boldsymbol{d}$, as in~\eqref{eq:x_def}. Note that $\mathcal{P}_{[1,T]}(s_0) = \mathcal{P}^{0\text{-}1}$, and that the restriction to $[l,k]$ of any $\boldsymbol{y} \in \mathcal{P}^{0\text{-}1}$ belongs to $\mathcal{P}_{[l,k]}(s_{l-1}(\boldsymbol{y}))$, where $s_{l-1}(\boldsymbol{y})$ denotes the SoC of $\boldsymbol{y}$ at the end of period $l-1$ (equal to $s_0$ if $l = 1$). For $\Omega \subseteq [l,k]$, let $\boldsymbol{1}_\Omega \in \{0,1\}^{[l,k]}$ be its indicator vector. We define the \emph{charging} and \emph{discharging capacity functions}
\begin{align}
f(\Omega; \sigma) &:= \max\big\{ c(\Omega) : (\boldsymbol{d}, \boldsymbol{c}, \boldsymbol{u}) \in \mathcal{P}_{[l,k]}(\sigma),\ \boldsymbol{u} = \boldsymbol{1}_\Omega \big\}, \label{eq:LVIopt}\\
h(\Omega; \sigma) &:= \max\big\{ d(\Omega) : (\boldsymbol{d}, \boldsymbol{c}, \boldsymbol{u}) \in \mathcal{P}_{[l,k]}(\sigma),\ \boldsymbol{u} = \boldsymbol{1}_{[l,k] \setminus \Omega} \big\}. \label{eq:LVIopt_dis}
\end{align}
Thus, $f(\Omega; \sigma)$ is the maximum energy that can be charged during the periods in $\Omega$ when the window starts at SoC $\sigma$ and charging is allowed exactly in $\Omega$; $h(\Omega; \sigma)$ is its discharging counterpart. Both maxima are attained, because the feasible sets are polytopes containing $\boldsymbol{x} = \boldsymbol{0}$. Constraints outside the window are ignored, so these functions provide upper bounds on what any point of $\mathcal{P}^{0\text{-}1}$ can achieve; Proposition~\ref{prop:submodularity} makes this precise. We now show that $f(\cdot\,; \sigma)$ and $h(\cdot\,; \sigma)$ are submodular.

\paragraph{Generalized polymatroids.} Let $N$ be a finite ground set. A pair of set functions $p, b : 2^N \to \mathbb{R}$ is \emph{paramodular} if $p(\emptyset) = b(\emptyset) = 0$, $p$ is supermodular, $b$ is submodular, and the cross-inequality $b(A) - p(B) \geq b(A \setminus B) - p(B \setminus A)$ holds for all $A, B \subseteq N$. The polyhedron
\[
Q(p,b) := \big\{ \boldsymbol{x} \in \mathbb{R}^N : p(A) \leq x(A) \leq b(A),\ A \subseteq N \big\}
\]
is then called a \emph{generalized polymatroid} (g-polymatroid). We use the following classical result; see~\citet[Chapter~14]{frank2011connections}.
\begin{lemma}\label{lem:gpolymatroid}
A nonempty g-polymatroid $Q$ determines its paramodular pair uniquely, namely $b(A) = \max\{x(A) : \boldsymbol{x} \in Q\}$ and $p(A) = \min\{x(A) : \boldsymbol{x} \in Q\}$ for all $A \subseteq N$. In particular, $A \mapsto \max\{x(A) : \boldsymbol{x} \in Q\}$ is submodular.
\end{lemma}

We call $b$ the \emph{rank function} of $Q(p,b)$; it satisfies $b(A) = \max\{x(A) : \boldsymbol{x} \in Q(p,b)\}$ for every $A \subseteq N$.

Dropping the operating modes from $\mathcal{P}_{[l,k]}(\sigma)$ leaves the set of net-charging sequences
\begin{equation}\label{eq:Xset}
\mathcal{X}(\sigma) := \Big\{ \boldsymbol{x} \in \mathbb{R}^{[l,k]} : -\overline{d} \leq x_t \leq \overline{c},\ \ \underline{s} \leq \sigma + x([l,t]) \leq \overline{s},\ \ t \in [l,k] \Big\}.
\end{equation}
After relabeling the periods of $[l,k]$ as $1, \ldots, k-l+1$ and translating the SoC bounds by $-\sigma$, the set $\mathcal{X}(\sigma)$ coincides with the feasible set of a lossless storage unit with zero initial SoC, power bounds $[-\overline{d}, \overline{c}]$, and SoC bounds $[\underline{s} - \sigma, \overline{s} - \sigma] \ni 0$. Hence, by~\citet[Theorem~2]{11369438}, which even allows period-dependent bounds, $\mathcal{X}(\sigma)$ is a g-polymatroid, and by Lemma~\ref{lem:gpolymatroid} the function
\begin{equation}\label{eq:b_parameterized}
b(\Omega; \sigma) := \max\big\{ x(\Omega) : \boldsymbol{x} \in \mathcal{X}(\sigma) \big\}, \qquad \Omega \subseteq [l,k],
\end{equation}
is submodular. The next lemma shows that $b$ is attained by a greedy rule and describes its dependence on $\sigma$.

\begin{lemma}\label{lem:greedy_polymatroid}
For $\Omega \subseteq [l,k]$, let $s^{(\Omega)}_{l-1} := \sigma$ and, for $t = l, \ldots, k$,
\begin{equation}\label{eq:greedy_rule}
x_t^{(\Omega)} :=
\begin{cases}
\min\big\{\overline{c},\ \overline{s} - s^{(\Omega)}_{t-1}\big\} & \text{if } t \in \Omega,\\[2pt]
\max\big\{-\overline{d},\ \underline{s} - s^{(\Omega)}_{t-1}\big\} & \text{if } t \notin \Omega,
\end{cases}
\qquad s^{(\Omega)}_t := s^{(\Omega)}_{t-1} + x_t^{(\Omega)}.
\end{equation}
That is, the rule charges as much as possible in $\Omega$ and discharges as much as possible outside $\Omega$. Then:
\begin{enumerate}[(i)]
    \item $\boldsymbol{x}^{(\Omega)}$ is optimal in~\eqref{eq:b_parameterized}, so $b(\Omega; \sigma) = x^{(\Omega)}(\Omega)$, which is computed in $\mathcal{O}(k-l+1)$ time;
    \item $\sigma \mapsto b(\Omega; \sigma)$ is nonincreasing and $1$-Lipschitz on $[\underline{s}, \overline{s}]$, i.e., $0 \leq b(\Omega; \sigma) - b(\Omega; \sigma') \leq \sigma' - \sigma$ whenever $\sigma \leq \sigma'$.
\end{enumerate}
\end{lemma}

The proof, given in Appendix~\ref{ec:proof_greedy}, is by induction on the window length. Optimizing first over the periods after $l$, the value $b(\Omega; \sigma)$ is the maximum, over the SoC $\sigma_1$ reached at the end of period $l$, of the charge of period $l$ (if $l \in \Omega$) plus the value $\beta'(\sigma_1)$ of the shorter window $[l+1,k]$. The induction hypothesis states that $\beta'$ is nonincreasing, because a fuller storage can only reduce the future charge, and that $\sigma_1 + \beta'(\sigma_1)$ is nondecreasing, because one more unit in storage costs at most one unit of future charge. The first property makes the largest discharge optimal if $l \notin \Omega$, and the second makes the largest charge optimal if $l \in \Omega$; both properties are inherited by $b(\Omega; \cdot\,)$, which closes the induction. The maps $\phi^+(\sigma) := \min\{\sigma + \overline{c}, \overline{s}\}$ and $\phi^-(\sigma) := \max\{\sigma - \overline{d}, \underline{s}\}$, which give the SoC reached at the end of a period that starts at $\sigma$ and charges, respectively discharges, as much as possible, are used throughout the paper.

By construction, every period of $\boldsymbol{x}^{(\Omega)}$ either operates at a power bound or ends at a tight SoC. Hence, by Proposition~\ref{prop:extreme_points}, applied to the window $[l,k]$ (see Remark~\ref{rem:extreme_points_role}), the greedy sequences correspond to extreme points. We can now establish the main result of this subsection.

\begin{proposition}\label{prop:submodularity}
For every $\sigma \in [\underline{s}, \overline{s}]$ and $\Omega \subseteq [l,k]$, the following hold:
\begin{enumerate}[(i)]
    \item $f(\Omega; \sigma) = b(\Omega; \sigma) = x^{(\Omega)}(\Omega)$;
    \item the set function $f(\cdot\,; \sigma)$ is submodular on $2^{[l,k]}$;
    \item $\sigma \mapsto f(\Omega; \sigma)$ is nonincreasing and $1$-Lipschitz on $[\underline{s}, \overline{s}]$;
    \item for every $\boldsymbol{y} = (\boldsymbol{d}, \boldsymbol{c}, \boldsymbol{u}) \in \mathcal{P}^{0\text{-}1}$ and every $\sigma \in [\underline{s}, s_{l-1}(\boldsymbol{y})]$,
    \[
    c(\Omega) \leq f(\Omega \cap U_{\boldsymbol{y}}; \sigma), \qquad \text{where } U_{\boldsymbol{y}} := \{t \in [l,k] : u_t = 1\}.
    \]
\end{enumerate}
\end{proposition}

\begin{proof}{Proof}
(i) If $\boldsymbol{u} = \boldsymbol{1}_\Omega$, then $d_t = 0$ for $t \in \Omega$ and $c_t = 0$ for $t \notin \Omega$. Hence, by the correspondence of Section~\ref{sec:extreme_points}, problem~\eqref{eq:LVIopt} is equivalent to
\[
f(\Omega; \sigma) = \max\big\{ x(\Omega) : \boldsymbol{x} \in \mathcal{X}(\sigma),\ x_t \geq 0 \ (t \in \Omega),\ x_t \leq 0 \ (t \notin \Omega) \big\}.
\]
Its feasible set is contained in $\mathcal{X}(\sigma)$, so $f(\Omega; \sigma) \leq b(\Omega; \sigma)$. Conversely, the greedy sequence $\boldsymbol{x}^{(\Omega)}$ keeps the SoC in $[\underline{s}, \overline{s}]$. Thus, for $t \in \Omega$, $x^{(\Omega)}_t = \min\{\overline{c}, \overline{s} - s^{(\Omega)}_{t-1}\} \geq 0$, and for $t \notin \Omega$, $x^{(\Omega)}_t = \max\{-\overline{d}, \underline{s} - s^{(\Omega)}_{t-1}\} \leq 0$. Hence, $\boldsymbol{x}^{(\Omega)}$ is feasible for the problem above, and Lemma~\ref{lem:greedy_polymatroid}(i) yields $f(\Omega; \sigma) \geq x^{(\Omega)}(\Omega) = b(\Omega; \sigma)$.

(ii) This follows from (i) and the submodularity of $b(\cdot\,; \sigma)$.

(iii) This follows from (i) and Lemma~\ref{lem:greedy_polymatroid}(ii).

(iv) Let $\sigma_{\boldsymbol{y}} := s_{l-1}(\boldsymbol{y})$ and let $\boldsymbol{x}$ be the net-charging vector of $\boldsymbol{y}$ restricted to $[l,k]$, so that $\boldsymbol{x} \in \mathcal{X}(\sigma_{\boldsymbol{y}})$. For $t \in \Omega \cap U_{\boldsymbol{y}}$, we have $d_t = 0$ and thus $c_t = x_t$; for $t \in \Omega \setminus U_{\boldsymbol{y}}$, we have $c_t = 0$. Therefore,
\[
c(\Omega) = x(\Omega \cap U_{\boldsymbol{y}}) \leq b(\Omega \cap U_{\boldsymbol{y}}; \sigma_{\boldsymbol{y}}) = f(\Omega \cap U_{\boldsymbol{y}}; \sigma_{\boldsymbol{y}}) \leq f(\Omega \cap U_{\boldsymbol{y}}; \sigma),
\]
where the equality follows from (i) and the last inequality from (iii), since $\sigma \leq \sigma_{\boldsymbol{y}}$.
\hfill$\square$
\end{proof}

Part~(i) shows that fixing the operating modes does not reduce the charging capacity of $\Omega$: the best schedule charges at full rate in $\Omega$ and discharges at full rate elsewhere, which is compatible with $\boldsymbol{u} = \boldsymbol{1}_\Omega$. Part~(iv) is the bridge to valid inequalities. It bounds the energy charged during $\Omega$ by any feasible point in terms of the periods in which that point actually charges, regardless of its operation outside $\Omega$ and of its exact SoC at the start of the window.

The discharging function $h$ reduces to a charging problem for a \emph{mirrored} storage unit. This unit has the same window and SoC limits $[\underline{s}, \overline{s}]$ as the original one, but exchanged power limits: its maximum charging power is $\overline{d}$ and its maximum discharging power is $\overline{c}$. It reads every charge of the original unit as a discharge and vice versa, and it measures the stored energy from the opposite end of the SoC interval, so that its SoC is $s'_t = \underline{s} + \overline{s} - s_t$ and its initial SoC is $\sigma' = \underline{s} + \overline{s} - \sigma$. We denote by $f^{\leftrightarrow}(\Omega; \sigma')$ the charging capacity function of the mirrored unit, that is, problem~\eqref{eq:LVIopt} with $\overline{c}$ and $\overline{d}$ exchanged and initial SoC $\sigma'$.

\begin{corollary}\label{cor:submodular_disch}
For every $\sigma \in [\underline{s}, \overline{s}]$ and $\Omega \subseteq [l,k]$,
\begin{equation}\label{eq:h_mirror}
h(\Omega; \sigma) = f^{\leftrightarrow}(\Omega; \underline{s} + \overline{s} - \sigma).
\end{equation}
Consequently:
\begin{enumerate}[(i)]
    \item $h(\Omega; \sigma) = d^{[\Omega]}(\Omega)$, where the greedy rule starts at $s^{[\Omega]}_{l-1} := \sigma$ and, for $t = l, \ldots, k$, discharges as much as possible if $t \in \Omega$ and charges as much as possible otherwise:
    \[
    \begin{cases}
    d^{[\Omega]}_t := \min\big\{\overline{d},\ s^{[\Omega]}_{t-1} - \underline{s}\big\}, \quad s^{[\Omega]}_t := s^{[\Omega]}_{t-1} - d^{[\Omega]}_t & \text{if } t \in \Omega,\\[2pt]
    c^{[\Omega]}_t := \min\big\{\overline{c},\ \overline{s} - s^{[\Omega]}_{t-1}\big\}, \quad s^{[\Omega]}_t := s^{[\Omega]}_{t-1} + c^{[\Omega]}_t & \text{if } t \notin \Omega;
    \end{cases}
    \]
    \item the set function $h(\cdot\,; \sigma)$ is submodular on $2^{[l,k]}$;
    \item $\sigma \mapsto h(\Omega; \sigma)$ is nondecreasing and $1$-Lipschitz on $[\underline{s}, \overline{s}]$;
    \item for every $\boldsymbol{y} = (\boldsymbol{d}, \boldsymbol{c}, \boldsymbol{u}) \in \mathcal{P}^{0\text{-}1}$ and every $\sigma \in [s_{l-1}(\boldsymbol{y}), \overline{s}]$, we have $d(\Omega) \leq h(\Omega \setminus U_{\boldsymbol{y}}; \sigma)$.
\end{enumerate}
\end{corollary}

\begin{proof}{Proof}
The reflection $(\boldsymbol{d}, \boldsymbol{c}, \boldsymbol{u}) \mapsto (\boldsymbol{c}, \boldsymbol{d}, \boldsymbol{1} - \boldsymbol{u})$ maps the feasible set of~\eqref{eq:LVIopt_dis} bijectively onto that of the charging problem of the mirrored unit, and it preserves the objective value. This yields~\eqref{eq:h_mirror}. Parts (i)--(iv) then follow from Lemma~\ref{lem:greedy_polymatroid} and Proposition~\ref{prop:submodularity} applied to the mirrored unit. A detailed proof is given in Appendix~\ref{app:reflection}.
\hfill$\square$
\end{proof}

Next, we exploit Proposition~\ref{prop:submodularity} and Corollary~\ref{cor:submodular_disch} to derive linear valid inequalities for $\mathcal{P}^{0\text{-}1}$.

\subsection{Submodular Cuts}\label{sec:submodular_cuts}

We now apply the results of Section~\ref{sec:submodular_structure} to the whole horizon, that is, to the window $[T]$ with initial SoC $s_0$. For $\Omega \subseteq [T]$, let
\begin{equation}\label{eq:f_h_horizon}
f(\Omega) := f_{[T]}(\Omega; s_0) \qquad \text{and} \qquad h(\Omega) := h_{[T]}(\Omega; s_0),
\end{equation}
so that $f(\Omega)$ is the maximum energy that can be charged during $\Omega$, starting from $s_0$, when charging is allowed exactly in $\Omega$, and $h(\Omega)$ is its discharging counterpart. By Proposition~\ref{prop:submodularity} and Corollary~\ref{cor:submodular_disch}, both are submodular set functions on $2^{[T]}$. Their marginal values $f(t \mid A)$ and $h(t \mid A)$ are as in Definition~\ref{def:submodularity}.

\begin{theorem}\label{th:LVI}
For every $\Omega \subseteq [T]$, the inequalities
\begin{align}
\sum_{t \in \Omega} c_t &\leq f(\Omega) - \sum_{t \in \Omega} f(t \mid \Omega \setminus \{t\})\,(1 - u_t), \label{eq:VIc}\\
\sum_{t \in \Omega} d_t &\leq h(\Omega) - \sum_{t \in \Omega} h(t \mid \Omega \setminus \{t\})\, u_t \label{eq:VId}
\end{align}
are valid for $\mathcal{P}^{0\text{-}1}$ and, by Remark~\ref{rem:lifting}, for $\mathcal{S}^{0\text{-}1}$.
\end{theorem}

\begin{proof}{Proof}
Let $\boldsymbol{y} = (\boldsymbol{d}, \boldsymbol{c}, \boldsymbol{u}) \in \mathcal{P}^{0\text{-}1}$ and let $U_{\boldsymbol{y}} := \{t \in [T] : u_t = 1\}$ be its set of charging periods. For $t \in \Omega$, we have $t \notin U_{\boldsymbol{y}}$ if and only if $1 - u_t = 1$.

For~\eqref{eq:VIc}, Proposition~\ref{prop:submodularity}(iv), applied with window $[T]$ and $\sigma = s_0$, gives $c(\Omega) \leq f(\Omega \cap U_{\boldsymbol{y}})$. Lemma~\ref{lem:submodular_bound}, applied with $U = \Omega \cap U_{\boldsymbol{y}} \subseteq \Omega$, then gives
\[
c(\Omega) \leq f(\Omega) - \sum_{t \in \Omega \setminus U_{\boldsymbol{y}}} f(t \mid \Omega \setminus \{t\}) = f(\Omega) - \sum_{t \in \Omega} f(t \mid \Omega \setminus \{t\})\,(1 - u_t).
\]

For~\eqref{eq:VId}, Corollary~\ref{cor:submodular_disch}(iv) gives $d(\Omega) \leq h(\Omega \setminus U_{\boldsymbol{y}})$. Lemma~\ref{lem:submodular_bound}, applied with $U = \Omega \setminus U_{\boldsymbol{y}}$, then gives
\[
d(\Omega) \leq h(\Omega) - \sum_{t \in \Omega \cap U_{\boldsymbol{y}}} h(t \mid \Omega \setminus \{t\}) = h(\Omega) - \sum_{t \in \Omega} h(t \mid \Omega \setminus \{t\})\, u_t.
\]
\hfill$\square$
\end{proof}

\paragraph{Interpretation.} If $\boldsymbol{y}$ charges in every period of $\Omega$, inequality~\eqref{eq:VIc} reduces to the capacity bound $c(\Omega) \leq f(\Omega)$. Each period $t \in \Omega$ in which $\boldsymbol{y}$ does not charge modifies the right-hand side by its marginal value $f(t \mid \Omega \setminus \{t\})$. The marginal value is positive when charging in $t$ would increase the chargeable energy of $\Omega$, and then the bound tightens. It can also be negative: if the storage is nearly full, discharging in $t$ frees capacity for the remaining periods of $\Omega$, and the bound is relaxed accordingly.

For instance, let $T = 2$, $[\underline{s}, \overline{s}] = [0,2]$, $s_0 = 1.5$, and $\overline{c} = \overline{d} = 1$. By Lemma~\ref{lem:greedy_polymatroid}, $f(\{1,2\}) = f(\{1\}) = 0.5$ and $f(\{2\}) = 1$. Hence, $f(1 \mid \{2\}) = -0.5$ and $f(2 \mid \{1\}) = 0$, and~\eqref{eq:VIc} for $\Omega = \{1,2\}$ reads $c_1 + c_2 \leq 0.5 + 0.5\,(1 - u_1)$. The inequality allows $c_2 = 1$ only if the storage discharges in period~1; Section~\ref{sec:cut_strengthening} replaces the term $1 - u_t$ with the discharging power $d_t$.

\paragraph{Computation of the coefficients.} The coefficients of~\eqref{eq:VIc} and~\eqref{eq:VId} only depend on the periods spanned by $\Omega$. For $t \in [0, T]$, let
\begin{equation}\label{eq:sigma_bounds}
\underline{\sigma}_t := \max\{s_0 - t\,\overline{d},\ \underline{s}\} \qquad \text{and} \qquad \overline{\sigma}_t := \min\{s_0 + t\,\overline{c},\ \overline{s}\}
\end{equation}
be the minimum and maximum SoC reachable at the end of period~$t$. These are attained by discharging, respectively charging, as much as possible from the start; that they are bounds follows by induction from the SoC dynamics and the power limits. Let $\Omega \neq \emptyset$, $l := \min \Omega$, and $k := \max \Omega$, and let $A \subseteq \Omega$. The greedy rule~\eqref{eq:greedy_rule} applied to $A$ on $[T]$ discharges as much as possible in periods $1, \ldots, l-1$, which are not in $A$, and thus reaches $\underline{\sigma}_{l-1}$ at the end of period $l-1$; periods after $k$ do not contribute to $x(A)$. Hence, by Lemma~\ref{lem:greedy_polymatroid},
\[
f(A) = f_{[l,k]}(A; \underline{\sigma}_{l-1}) \qquad \text{and, symmetrically,} \qquad h(A) = h_{[l,k]}(A; \overline{\sigma}_{l-1}).
\]
Moreover, the greedy trajectories for $\Omega$ and $\Omega \setminus \{t\}$ coincide up to period $t-1$. Writing $V_{t+1}(\sigma) := f_{[t+1,k]}(\Omega \cap [t+1,k]; \sigma)$, Lemma~\ref{lem:greedy_polymatroid} yields
\begin{equation}\label{eq:gain_formula}
f(t \mid \Omega \setminus \{t\}) = x^{(\Omega)}_t + V_{t+1}\big(\min\{s^{(\Omega)}_{t-1} + \overline{c}, \overline{s}\}\big) - V_{t+1}\big(\max\{s^{(\Omega)}_{t-1} - \overline{d}, \underline{s}\}\big),
\end{equation}
so all coefficients of~\eqref{eq:VIc} for a given $\Omega$ are obtained from one reference pass over $[l,k]$ and two passes over $[t+1,k]$ for each $t \in \Omega$, i.e., in $\mathcal{O}\big(\sum_{t \in \Omega} (k - t + 1)\big)$ time, because the term of $t = l$ already bounds the reference pass; Section~\ref{sec:two_passes} reduces this cost to $\mathcal{O}(k - l + 1)$. The coefficients of~\eqref{eq:VId} are computed analogously.

\paragraph{Singleton cuts.} For a singleton $\Omega = \{t\}$, we have $f(t \mid \emptyset) = f(\{t\})$ and $h(t \mid \emptyset) = h(\{t\})$, so~\eqref{eq:VIc} and~\eqref{eq:VId} become
\begin{equation}\label{eq:singleton_cuts}
c_t \leq \overline{c}_t\, u_t \qquad \text{and} \qquad d_t \leq \overline{d}_t\, (1 - u_t),
\end{equation}
with
\[
\overline{c}_t := f(\{t\}) = \min\{\overline{c},\ \overline{s} - \underline{\sigma}_{t-1}\} \qquad \text{and} \qquad \overline{d}_t := h(\{t\}) = \min\{\overline{d},\ \overline{\sigma}_{t-1} - \underline{s}\},
\]
where the closed forms follow from the computation above with $l = k = t$. These \emph{singleton cuts} tighten the power limits~\eqref{eq:P01_max_dis}--\eqref{eq:P01_max_ch} by accounting for the SoC that can be reached at the beginning of period~$t$; note that $\overline{c}_t \leq \overline{c}$ and $\overline{d}_t \leq \overline{d}$. They are a period-dependent refinement of the effective limits of Section~\ref{sec:Formulation}, and their number is only $2T$. Whenever the proposed inequalities are used, we therefore include all singleton cuts, and we denote by $\mathcal{S}^{\mathcal{R}}_1$ the relaxation $\mathcal{S}^{\mathcal{R}}$ strengthened with~\eqref{eq:singleton_cuts}. The relaxation $\mathcal{S}^{\mathcal{R}}$ itself, which serves as the benchmark in Section~\ref{sec:case_study}, does not contain them.

\subsection{Strengthening with Continuous Power Variables}\label{sec:cut_strengthening}

Fix $\Omega \subseteq [T]$ and define, for every $\tau \in \Omega$,
\[
\pi_\tau := -f(\tau \mid \Omega \setminus \{\tau\}) \qquad \text{and} \qquad \rho_\tau := -h(\tau \mid \Omega \setminus \{\tau\}).
\]
In~\eqref{eq:VIc}, a period $\tau \in \Omega$ in which the storage does not charge contributes the term $\pi_\tau (1 - u_\tau)$ to the right-hand side. When $\pi_\tau > 0$, this term credits the full amount $\pi_\tau$ of capacity freed by switching to discharging mode in $\tau$, regardless of how much is actually discharged. The discharging power $d_\tau$ provides a finer measure of that capacity. The next result shows that the binary terms of any subset of periods can be replaced simultaneously by their continuous counterparts.

\begin{proposition}\label{prop:cut_strengthening}
For every $\Omega \subseteq [T]$ and every $D \subseteq \Omega$, the inequalities
\begin{align}
\sum_{t \in \Omega} c_t &\leq f(\Omega) + \sum_{t \in D} d_t + \sum_{\tau \in \Omega \setminus D} \pi_\tau\,(1 - u_\tau), \label{eq:VIc_D}\\
\sum_{t \in \Omega} d_t &\leq h(\Omega) + \sum_{t \in D} c_t + \sum_{\tau \in \Omega \setminus D} \rho_\tau\, u_\tau \label{eq:VId_D}
\end{align}
are valid for $\mathcal{P}^{0\text{-}1}$. Moreover, for every $\tau \in \Omega$,
\begin{equation}\label{eq:pi_bound}
\pi_\tau \leq \min\big\{\overline{d},\ s^{(\Omega)}_{\tau-1} - \underline{s}\big\} \leq \overline{d}_\tau \qquad \text{and} \qquad \rho_\tau \leq \min\big\{\overline{c},\ \overline{s} - s^{[\Omega]}_{\tau-1}\big\} \leq \overline{c}_\tau,
\end{equation}
where $s^{(\Omega)}$ and $s^{[\Omega]}$ are the SoC trajectories produced by the greedy rules of Lemma~\ref{lem:greedy_polymatroid} and Corollary~\ref{cor:submodular_disch} for the set $\Omega$, applied to the whole horizon $[T]$ with initial SoC $s_0$.
\end{proposition}

For $D = \emptyset$, inequalities~\eqref{eq:VIc_D} and~\eqref{eq:VId_D} reduce to~\eqref{eq:VIc} and~\eqref{eq:VId}. At the other extreme, $D = \Omega$ turns~\eqref{eq:VIc_D} into $x(\Omega) \leq f(\Omega)$, which is already implied by the linear relaxation, because the net-charging vector of every point of $\mathcal{S}^{\mathcal{R}}$ belongs to $\mathcal{X}_{[T]}(s_0)$. The family thus interpolates between the submodular cuts and the linear relaxation. We call $\pi_\tau(1 - u_\tau)$, for $\tau \in \Omega \setminus D$, and $d_\tau$, for $\tau \in D$, the \emph{credit} that the cut grants for period~$\tau$: the amount by which $\tau$ raises the right-hand side above $f(\Omega)$ when it does not charge.

\begin{proof}{Proof}
\emph{Validity of~\eqref{eq:VIc_D}.} Let $\boldsymbol{y} = (\boldsymbol{d}, \boldsymbol{c}, \boldsymbol{u}) \in \mathcal{P}^{0\text{-}1}$ with net-charging vector $\boldsymbol{x}$ and charging periods $U := U_{\boldsymbol{y}}$, and let $B := (\Omega \cap U) \cup D \subseteq \Omega$. For $t \in \Omega \cap U$, we have $x_t = c_t$ and $d_t = 0$; for $t \in D \setminus U$, we have $x_t = -d_t$ and $c_t = 0$. Hence,
\[
x(B) = c(\Omega \cap U) - d(D \setminus U) = c(\Omega) - d(D).
\]
Since $\boldsymbol{x} \in \mathcal{X}_{[T]}(s_0)$, Proposition~\ref{prop:submodularity}(i) gives $x(B) \leq b(B; s_0) = f(B)$. Lemma~\ref{lem:submodular_bound}, applied to $B \subseteq \Omega$, then gives $f(B) \leq f(\Omega) + \sum_{\tau \in \Omega \setminus B} \pi_\tau$. Since $\Omega \setminus B = \{\tau \in \Omega \setminus D : u_\tau = 0\}$, the last sum equals $\sum_{\tau \in \Omega \setminus D} \pi_\tau (1 - u_\tau)$. Combining these relations yields~\eqref{eq:VIc_D}.

\emph{Validity of~\eqref{eq:VId_D}.} Apply~\eqref{eq:VIc_D} to the mirrored storage unit of Appendix~\ref{app:reflection} on the window $[T]$, with initial SoC $\underline{s} + \overline{s} - s_0$, and to the reflected point $\Theta(\boldsymbol{y})$. By Lemma~\ref{lem:reflection}, $\Theta(\boldsymbol{y})$ is feasible for the mirrored unit. Its charging and discharging powers are $\boldsymbol{d}$ and $\boldsymbol{c}$, and its operating modes are $\boldsymbol{1} - \boldsymbol{u}$. By~\eqref{eq:h_mirror}, the charging capacity function of the mirrored unit is $h$. The resulting inequality is exactly~\eqref{eq:VId_D}.

\emph{Bounds~\eqref{eq:pi_bound}.} Let $\tau \in \Omega$, $k := \max \Omega$, and $s := s^{(\Omega)}_{\tau-1}$. By Lemma~\ref{lem:greedy_polymatroid}(ii) and Proposition~\ref{prop:submodularity}(i), the function $V_{\tau+1}$ in~\eqref{eq:gain_formula} is nonincreasing and $1$-Lipschitz; for $\tau = k$, it is identically zero. Since $\phi^+(s) := \min\{s + \overline{c}, \overline{s}\} \geq \phi^-(s) := \max\{s - \overline{d}, \underline{s}\}$ and $x^{(\Omega)}_\tau = \phi^+(s) - s$, identity~\eqref{eq:gain_formula} yields
\[
f(\tau \mid \Omega \setminus \{\tau\}) \geq \big(\phi^+(s) - s\big) - \big(\phi^+(s) - \phi^-(s)\big) = \phi^-(s) - s = -\min\{\overline{d},\ s - \underline{s}\}.
\]
This is the first bound on $\pi_\tau$. The second one follows from~\eqref{eq:singleton_cuts}, because $s^{(\Omega)}_{\tau-1} \leq \overline{\sigma}_{\tau-1}$, the maximum SoC reachable at the end of period $\tau - 1$. The bounds on $\rho_\tau$ follow by applying these arguments to the mirrored unit.
\hfill$\square$
\end{proof}

Bound~\eqref{eq:pi_bound} has a direct interpretation: switching period $\tau$ to discharging cannot free more charging capacity in $\Omega$ than the energy that can actually be discharged in $\tau$, which is limited by the power $\overline{d}$ and by the stored energy. The limit $\overline{d}_\tau$ determines when the continuous term should be preferred to the binary one.

\begin{proposition}\label{prop:term_comparison}
Let $\Omega \subseteq [T]$, $\tau \in \Omega$, and $D \subseteq \Omega \setminus \{\tau\}$. Inequality~\eqref{eq:VIc_D} for $D \cup \{\tau\}$ differs from~\eqref{eq:VIc_D} for $D$ only in that the term $\pi_\tau (1 - u_\tau)$ is replaced by $d_\tau$. On the set $\mathcal{Q}_\tau := \{(u_\tau, d_\tau) : 0 \leq u_\tau \leq 1,\ 0 \leq d_\tau \leq \overline{d}_\tau\,(1 - u_\tau)\}$, which contains the corresponding projections of $\mathcal{P}^{0\text{-}1}$ and of $\mathcal{S}^{\mathcal{R}}_1$, the following hold:
\begin{enumerate}[(a)]
    \item If $\pi_\tau \leq 0$, then $\pi_\tau (1 - u_\tau) \leq d_\tau$, so the inequality for $D$ implies the one for $D \cup \{\tau\}$.
    \item If $\pi_\tau = \overline{d}_\tau$, then $d_\tau \leq \pi_\tau (1 - u_\tau)$, so the inequality for $D \cup \{\tau\}$ implies the one for $D$.
    \item If $0 < \pi_\tau < \overline{d}_\tau$, neither term is pointwise smaller: $d_\tau < \pi_\tau (1 - u_\tau)$ at $(u_\tau, d_\tau) = (0,0)$, whereas $d_\tau > \pi_\tau (1 - u_\tau)$ at $(0, \overline{d}_\tau)$.
\end{enumerate}
Here, ``implies'' holds at every point whose $(u_\tau, d_\tau)$ lies in $\mathcal{Q}_\tau$, in particular over $\mathcal{S}^{\mathcal{R}}_1$. The analogous statements hold for~\eqref{eq:VId_D}, with $\rho_\tau$, $c_\tau$, $u_\tau$, and $\overline{c}_\tau$ in place of $\pi_\tau$, $d_\tau$, $1 - u_\tau$, and $\overline{d}_\tau$.
\end{proposition}

\begin{proof}{Proof}
Parts (a) and (c) follow by direct evaluation, using $1 - u_\tau \geq 0$ and $d_\tau \geq 0$ in (a). Part (b) follows from $d_\tau \leq \overline{d}_\tau\,(1 - u_\tau)$. By~\eqref{eq:pi_bound}, $\pi_\tau \leq \overline{d}_\tau$, so the three cases are exhaustive. Since the two inequalities share all other terms, a pointwise comparison of the differing terms yields the stated implications.
\hfill$\square$
\end{proof}

In case~(c), the two inequalities are the two linear pieces of the concave bound
\[
c(\Omega) \leq f(\Omega) + \sum_{t \in D} d_t + \sum_{\tau' \in \Omega \setminus (D \cup \{\tau\})} \pi_{\tau'}\,(1 - u_{\tau'}) + \min\big\{d_\tau,\ \pi_\tau\,(1 - u_\tau)\big\},
\]
which no single linear term can express, so both are useful. Note also that case~(b) depends on the singleton cuts: over $\mathcal{S}^{\mathcal{R}}$ alone, where only $d_\tau \leq \overline{d}\,(1 - u_\tau)$ is available, the same argument requires $\pi_\tau = \overline{d}$, which fails whenever the stored energy, rather than the power limit, restricts discharging in~$\tau$. Appendix~\ref{ec:example_term_comparison} gives an instance with $T = 2$ in which the cuts for $D = \emptyset$ and $D = \{1\}$ are incomparable over $\mathcal{S}^{\mathcal{R}}$ but not over $\mathcal{S}^{\mathcal{R}}_1$.

Accordingly, we adopt the following convention.

\begin{remark}\label{rem:dominating_cuts_convention}
Hereinafter, every period $\tau \in \Omega$ with $\pi_\tau = \overline{d}_\tau$ is included in $D$, and every period with $\pi_\tau \leq 0$ is excluded from it. Periods with $0 < \pi_\tau < \overline{d}_\tau$ may be placed on either side; the separation algorithm of Section~\ref{sec:separation_algorithm} decides which. The same convention applies to~\eqref{eq:VId_D} with $\rho_\tau$ and $\overline{c}_\tau$.
\end{remark}

All results in this section are stated for the normalized model of Section~\ref{sec:Formulation}. Appendix~\ref{app:physical_units} derives step by step their equivalent form in terms of the physical charging and discharging power and the original data, and interprets the efficiency factors that appear in it.

\begin{remark}[Period-dependent limits and self-discharge]\label{rem:extensions}
The proofs of Theorem~\ref{th:LVI} and Proposition~\ref{prop:cut_strengthening} use only two properties of the capacity functions: they are attained by greedy rules that respect the operating modes, and they are submodular. Both survive two extensions under which the lattice of Remark~\ref{rem:extreme_points_role} breaks down. The first is period-dependent power limits $c^{\max}_t$ and $d^{\max}_t$, for which the proof of Lemma~\ref{lem:greedy_polymatroid} holds verbatim with period-dependent maps $\phi^\pm_t$, and submodularity follows from~\citet[Theorem~2]{11369438}, which allows period-dependent power and energy bounds. The second is self-discharge, $s_t = \eta\, s_{t-1} + c_t - d_t$ with $\eta \in (0,1]$ and $\underline{s} = 0$. The rescaling $\tilde{s}_t := \eta^{-t} s_t$ turns the unit into a lossless one with period-dependent bounds, and $f(\Omega)$ and $h(\Omega)$ become nonnegative combinations of the values of the rank functions of the rescaled unit at the sets $\Omega \cap [1,j]$, $j \in [T]$, hence submodular. In both cases, the cuts~\eqref{eq:VIc_D} and~\eqref{eq:VId_D}, with coefficients computed by the corresponding greedy rules, remain valid (Appendix~\ref{ec:extensions}). The facet characterization of Section~\ref{sec:facet_defining_conditions} and the closed forms of Section~\ref{sec:separation_algorithm} are derived for constant limits without self-discharge, and we do not claim them beyond that case.
\end{remark}

\subsection{Strength of the Cuts and Facet-Defining Conditions}
\label{sec:facet_defining_conditions}

We now study when the cuts of Section~\ref{sec:cut_strengthening} define facets of $\operatorname{conv}(\mathcal{P}^{0\text{-}1})$. By Remark~\ref{rem:lifting}, the conclusions transfer to $\operatorname{conv}(\mathcal{S}^{0\text{-}1})$. We focus on the charging inequalities~\eqref{eq:VIc_D}; the discharging inequalities~\eqref{eq:VId_D} are covered by the reflection of Appendix~\ref{app:reflection}. We begin with the dimension of $\operatorname{conv}(\mathcal{P}^{0\text{-}1})$.

\begin{assumption}\label{ass:interior_s0}
The initial SoC is strictly interior, i.e., $\underline{s} < s_0 < \overline{s}$.
\end{assumption}

\begin{lemma}\label{lem:full_dimensional}
Under Assumption~\ref{ass:interior_s0}, the polytope $\operatorname{conv}(\mathcal{P}^{0\text{-}1})$ is full-dimensional, that is, its dimension is $3T$.
\end{lemma}

The proof, given in Appendix~\ref{ec:auxiliary}, exhibits $3T + 1$ affinely independent points of $\mathcal{P}^{0\text{-}1}$: the idle point and, for each period, the three points that differ from it only in that period, by idling in charging mode, charging a small amount $\varepsilon$, or discharging $\varepsilon$.

If Assumption~\ref{ass:interior_s0} fails, $\operatorname{conv}(\mathcal{P}^{0\text{-}1})$ is not full-dimensional: $s_0 = \overline{s}$ forces $c_1 = 0$, and $s_0 = \underline{s}$ forces $d_1 = 0$. We assume Assumption~\ref{ass:interior_s0} hereinafter. Note that it also implies $\overline{c}_t > 0$ and $\overline{d}_t > 0$ in~\eqref{eq:singleton_cuts} for all $t \in [T]$.

\paragraph{The cuts under study.} Fix $\Omega \subseteq [T]$ with $|\Omega| \geq 2$, and let $l := \min \Omega$ and $k := \max \Omega$. Singletons are excluded because they yield the singleton cuts~\eqref{eq:singleton_cuts}. Following Remark~\ref{rem:dominating_cuts_convention}, we consider the cuts~\eqref{eq:VIc_D} whose set $D$ satisfies
\begin{equation}\label{eq:D_Omega}
D_\Omega \subseteq D \subseteq D^+_\Omega, \qquad \text{with } D_\Omega := \{\tau \in \Omega : \pi_\tau = \overline{d}_\tau\} \text{ and } D^+_\Omega := \{\tau \in \Omega : \pi_\tau > 0\},
\end{equation}
that is, $D$ contains every period whose continuous term dominates the binary one and no period whose binary term dominates the continuous one; no other choice of $D$ yields a facet (Theorem~\ref{th:facets}). We fix such a set $D$. The face of $\operatorname{conv}(\mathcal{P}^{0\text{-}1})$ induced by the cut is
\begin{equation}\label{eq:face_Omega}
F_\Omega := \Big\{ \boldsymbol{y} \in \operatorname{conv}(\mathcal{P}^{0\text{-}1}) : \ c(\Omega) = f(\Omega) + d(D) + \sum_{\tau \in \Omega \setminus D} \pi_\tau\,(1 - u_\tau) \Big\}.
\end{equation}

\paragraph{Structure of the face.} For a point $\boldsymbol{y} \in \mathcal{P}^{0\text{-}1}$, let
\[
R_{\boldsymbol{y}} := \{\tau \in \Omega \setminus D : u_\tau = 0\}
\]
be the set of periods of $\Omega$ that $\boldsymbol{y}$ \emph{relaxes}, i.e., in which it does not charge although their binary term is in the cut, and let $B_{\boldsymbol{y}} := \Omega \setminus R_{\boldsymbol{y}}$ be the remaining periods of $\Omega$. The proof of Proposition~\ref{prop:cut_strengthening}, applied with $B = B_{\boldsymbol{y}}$, establishes the chain
\begin{equation}\label{eq:cut_chain}
c(\Omega) - d(D) = x(B_{\boldsymbol{y}}) \overset{\text{(a)}}{\leq} f(B_{\boldsymbol{y}}) \overset{\text{(b)}}{\leq} f(\Omega) + \sum_{\tau \in R_{\boldsymbol{y}}} \pi_\tau = f(\Omega) + \sum_{\tau \in \Omega \setminus D} \pi_\tau\,(1 - u_\tau).
\end{equation}
Inequality (a) is the capacity bound $x(B) \leq b(B) = f(B)$ of Proposition~\ref{prop:submodularity}(i): the energy charged by $\boldsymbol{y}$ in $B_{\boldsymbol{y}}$ cannot exceed the maximum chargeable energy. Inequality (b) is the submodular bound of Lemma~\ref{lem:submodular_bound} for $B_{\boldsymbol{y}} \subseteq \Omega$. The leftmost and rightmost terms of~\eqref{eq:cut_chain} are the two sides of~\eqref{eq:VIc_D} after moving $d(D)$ to the left. Hence, $\boldsymbol{y}$ satisfies~\eqref{eq:VIc_D} with equality if and only if both (a) and (b) hold with equality. Inequality (b) holds with equality exactly for the relaxation sets in the family
\[
\mathcal{R}_\Omega := \Big\{ R \subseteq \Omega \setminus D : f(\Omega \setminus R) = f(\Omega) + \sum_{\tau \in R} \pi_\tau \Big\}.
\]
By definition of $\pi_\tau$, the family $\mathcal{R}_\Omega$ contains $\emptyset$ and every singleton $\{\tau\}$ with $\tau \in \Omega \setminus D$. It contains larger sets only when the gains of their elements do not interact, for instance when a period is isolated from the rest of $\Omega$ by a preceding full discharge. For $R \in \mathcal{R}_\Omega$, the \emph{attainment set}
\[
\mathcal{A}_R := \big\{ \boldsymbol{y} \in \mathcal{P}^{0\text{-}1} : u_\tau = 0 \ (\tau \in R),\ u_\tau = 1 \ (\tau \in \Omega \setminus (D \cup R)),\ x(\Omega \setminus R) = f(\Omega \setminus R) \big\}
\]
contains the points that relax exactly $R$ and charge the maximum energy $f(\Omega \setminus R)$ in the remaining periods. The modes of the periods in $D$ and outside $\Omega$ are free.

\begin{lemma}\label{lem:face_structure}
A point $\boldsymbol{y} \in \mathcal{P}^{0\text{-}1}$ satisfies~\eqref{eq:VIc_D} with equality if and only if $R_{\boldsymbol{y}} \in \mathcal{R}_\Omega$ and $\boldsymbol{y} \in \mathcal{A}_{R_{\boldsymbol{y}}}$. Consequently,
\[
F_\Omega = \operatorname{conv}\Big( \bigcup_{R \in \mathcal{R}_\Omega} \mathcal{A}_R \Big).
\]
\end{lemma}

\begin{proof}{Proof}
By~\eqref{eq:cut_chain}, a point $\boldsymbol{y} \in \mathcal{P}^{0\text{-}1}$ satisfies~\eqref{eq:VIc_D} with equality if and only if (a) and (b) hold with equality. Equality in (a) means $x(B_{\boldsymbol{y}}) = f(B_{\boldsymbol{y}})$; since $\boldsymbol{y}$ has, by definition of $R_{\boldsymbol{y}}$, the modes required in $\mathcal{A}_{R_{\boldsymbol{y}}}$, this is equivalent to $\boldsymbol{y} \in \mathcal{A}_{R_{\boldsymbol{y}}}$. Equality in (b) means $R_{\boldsymbol{y}} \in \mathcal{R}_\Omega$. The second claim follows because $F_\Omega$ is a face of the polytope $\operatorname{conv}(\mathcal{P}^{0\text{-}1})$, so it is the convex hull of the extreme points of $\operatorname{conv}(\mathcal{P}^{0\text{-}1})$ that it contains, and these belong to $\mathcal{P}^{0\text{-}1}$.
\hfill$\square$
\end{proof}

Lemma~\ref{lem:face_structure} describes $F_\Omega$ through the attainment sets, which always include $\mathcal{A}_\emptyset$, containing the greedy point of Lemma~\ref{lem:greedy_polymatroid}, and $\mathcal{A}_{\{\tau\}}$ for each $\tau \in \Omega \setminus D$. Since $\operatorname{conv}(\mathcal{P}^{0\text{-}1})$ is full-dimensional, $F_\Omega$ is a facet if and only if every linear equation satisfied by all the points of $\bigcup_{R \in \mathcal{R}_\Omega} \mathcal{A}_R$ is a multiple of the equation defining $F_\Omega$ in~\eqref{eq:face_Omega}. Each attainment set satisfies further equations of its own, such as $u_\tau = 0$ for $\tau \in R$ in $\mathcal{A}_R$, which fail on $\mathcal{A}_\emptyset$; the conditions below exclude the equations that hold on all attainment sets at once.

\paragraph{Thresholds.} The facet conditions, and the separation algorithm of Section~\ref{sec:separation_algorithm}, rely on a simple description of how much a set of charging periods can still charge once the storage holds a given amount of energy. For $B \subseteq [T]$ and $t \in [0,T]$, let
\[
V^B_{t+1}(\sigma) := f_{[t+1,T]}(B \cap [t+1,T];\, \sigma), \qquad \sigma \in [\underline{s}, \overline{s}],
\]
be the maximum energy that the periods of $B$ after $t$ can charge when the SoC at the end of period~$t$ is $\sigma$; for $B = \Omega$ and $t < k$, it is the function $V_{t+1}$ of~\eqref{eq:gain_formula}. A fuller storage can only reduce it, and the next lemma, proved in Appendix~\ref{ec:thresholds}, shows that it does so in the simplest possible way.

\begin{lemma}\label{lem:thresholds}
Let $B \subseteq [T]$. Define $\kappa^B_T := \overline{s}$ and, for $t = T, \ldots, 1$,
\[
\kappa^B_{t-1} := \max\{\kappa^B_t - \overline{c},\ \underline{s}\} \ \text{ if } t \in B, \qquad \kappa^B_{t-1} := \min\{\kappa^B_t + \overline{d},\ \overline{s}\} \ \text{ if } t \notin B.
\]
Then $\kappa^B_t \in [\underline{s}, \overline{s}]$ and $V^B_{t+1}(\sigma) = V^B_{t+1}(\underline{s}) - (\sigma - \kappa^B_t)^+$ for all $\sigma \in [\underline{s}, \overline{s}]$ and $t \in [0,T]$.
\end{lemma}

We call $\kappa^B_t$ the \emph{threshold} of $B$ at the end of period~$t$: below it, extra energy in storage is harmless for the charging periods of $B$ after $t$, because the other periods can discharge it in time; above it, every extra unit fills one unit of room that a later charging period of $B$ would have used. Read backwards in time, the recursion says that a charging period needs $\overline{c}$ units of room before it and that a period outside $B$ can create $\overline{d}$ units of room. We write $\kappa_t := \kappa^\Omega_t$ and, for $R \in \mathcal{R}_\Omega$, $\kappa^R_t := \kappa^{\Omega \setminus R}_t$ and $s^R := s^{(\Omega \setminus R)}$, the SoC trajectory of the greedy rule for $\Omega \setminus R$ on $[T]$ from $s_0$. Note that $\kappa^B_t$ only depends on $B \cap [t+1,T]$; in particular, $\kappa_t = \overline{s}$ for $t \geq k$.

\paragraph{Facet conditions.} Recall that $s^{(\Omega)}$ is the SoC trajectory produced by the greedy rule~\eqref{eq:greedy_rule} for the set $\Omega$, applied to the whole horizon $[T]$ with initial SoC $s_0$ (i.e., Lemma~\ref{lem:greedy_polymatroid} with window $[T]$ and $\sigma = s_0$). That is, $s^{(\Omega)}$ charges as much as possible in the periods of $\Omega$ and discharges as much as possible in all other periods. In particular, since no period before $l = \min \Omega$ belongs to $\Omega$, it discharges as much as possible in periods $1, \ldots, l-1$, so $s^{(\Omega)}_{l-1} = \underline{\sigma}_{l-1}$. Recall also the singleton limits $\overline{c}_t$ of~\eqref{eq:singleton_cuts}. Finally, for $t \in [l-1, k-1]$, let $\Omega_{\leq t} := \Omega \cap [1,t]$ and define the linear function
\begin{equation}\label{eq:L_t}
L_t(\boldsymbol{y}) := s_t(\boldsymbol{y}) - \Big[ c(\Omega_{\leq t}) - d(D \cap \Omega_{\leq t}) - \sum_{\tau \in \Omega_{\leq t} \setminus D} \pi_\tau\,(1 - u_\tau) \Big],
\end{equation}
i.e., the SoC at the end of period $t$ minus the terms of the periods of $\Omega$ up to $t$ in the cut~\eqref{eq:VIc_D}, written as $c(\Omega) - d(D) - \sum_{\tau \in \Omega \setminus D} \pi_\tau (1 - u_\tau) \leq f(\Omega)$. For a point that charges in every period of $\Omega_{\leq t}$, with $D \cap \Omega_{\leq t} = \emptyset$, $L_t(\boldsymbol{y}) = s_t(\boldsymbol{y}) - c(\Omega_{\leq t})$ is the SoC that would remain at $t$ had those charges not been made.

\begin{definition}\label{def:facet_conditions}
We say that the pair $(\Omega, D)$ satisfies:
\begin{enumerate}[({C}1)]
    \item \emph{(coupling)} if $\overline{s} - s^{(\Omega)}_t < \sum_{q \in \Omega,\, q > t} \overline{c}_q$ for $t = l-1$ and for every $t \in [l, k-1] \setminus \Omega$;
    \item \emph{(no intermediate depletion)} if $s^{(\Omega)}_t > \underline{s}$ for every $t \in [l, k-1] \setminus \Omega$;
    \item \emph{(chargeability outside $\Omega$)} if for every $t \in [T] \setminus \Omega$ there exists $R \in \mathcal{R}_\Omega$ with $s^R_{t-1} < \kappa^R_t$;
    \item \emph{(room at the bottlenecks)} if, for every $\tau \in \Omega \setminus D$ with $\kappa_\tau = \underline{s}$, there exists $R \in \mathcal{R}_\Omega$ with $R \subseteq (\tau, k]$ and $\kappa^R_\tau > s^{(\Omega)}_\tau$.
\end{enumerate}
\end{definition}

Each condition excludes one way in which $F_\Omega$ can lose dimension:
\begin{itemize}
    \item If (C1) fails at some $t$, the storage has room at $t$ for all the remaining charge of $\Omega$, even at full power. The periods of $\Omega$ after $t$ then interact neither with each other nor with the earlier ones, and the cut is the sum of the cut for $\Omega \cap [l,t]$ and the singleton cuts $c_q \leq \overline{c}_q u_q$ for $q \in \Omega$, $q > t$.
    \item If (C2) fails at some $t$, the greedy point empties the storage at $t$, which decouples the two parts of $\Omega$: the cut is the sum of the cuts for $\Omega \cap [l,t]$ and $\Omega \cap [t+1,k]$. In both cases, the cut is the sum of other valid inequalities, so it is implied by them and cannot define a facet.
    \item If (C3) fails at $t$, every attainment set reaches period $t$ at or above the threshold that follows it, so no point of $F_\Omega$ can charge in $t$ (Lemma~\ref{lem:C3_equiv}): $F_\Omega \subseteq \{c_t = 0\}$, and the cut can be strengthened by lifting $c_t$ with a positive coefficient.
    \item A period $\tau \in \Omega \setminus D$ with $\kappa_\tau = \underline{s}$ is a \emph{bottleneck}: from an empty storage at the end of $\tau$, the periods of $\Omega$ after $\tau$ fill it completely, so every unit of energy left in storage at the end of $\tau$ blocks one unit of their charge. (C4) asks for some relaxation of later periods that frees room at $\tau$. If it fails, every way of attaining the cut leaves the same baseline at $\tau$: every point of $F_\Omega$ satisfies the additional equation $L_\tau(\boldsymbol{y}) = L_\tau(\boldsymbol{z})$, where $\boldsymbol{z}$ is the greedy point of $\Omega$, as illustrated next.
\end{itemize}

For instance, let $T = 3$, $[\underline{s}, \overline{s}] = [0,1]$, $s_0 = 0.05$, $\overline{c} = 1$, $\overline{d} = 0.3$, and $\Omega = \{2,3\}$. Then $f(\Omega) = f(\{2\}) = f(\{3\}) = 1$, both gains are zero, $D = \emptyset$, and the cut reads $c_2 + c_3 \leq 1$. Conditions (C1)--(C3) hold. The thresholds are $\kappa_3 = \overline{s} = 1$ and $\kappa_2 = \max\{1 - \overline{c}, 0\} = 0 = \underline{s}$, so period~2 is a bottleneck, and the only relaxation of a later period, $R = \{3\}$, gives $\kappa^{\{3\}}_2 = \overline{s} = s^{(\Omega)}_2$: it creates no room at~2, and (C4) fails. Indeed, charging the full capacity in $\{2,3\}$ requires an empty storage right before charging starts, so every point of $F_\Omega$ satisfies $L_2(\boldsymbol{y}) = s_0 + c_1 - d_1 - d_2 = 0$, an equation that holds on all attainment sets at once and is not a multiple of the equation of the cut, and $\dim F_\Omega = 7 < 3T - 1$. By contrast, let $T = 4$, $[\underline{s}, \overline{s}] = [0,1]$, $s_0 = 0.1$, $\overline{c} = \overline{d} = 0.5$, and $\Omega = \{2,3,4\}$. The greedy point empties the storage in period~1 and fills it in periods~2 and~3, so $f(\Omega) = 1$, all gains are zero, $D = \emptyset$, and the cut reads $c_2 + c_3 + c_4 \leq 1$. Period~2 is again a bottleneck, since $\kappa_3 = 0.5$ and $\kappa_2 = 0$, but relaxing period~3 creates room at~2: $\kappa^{\{3\}}_2 = \min\{0.5 + \overline{d}, 1\} = 1 > s^{(\Omega)}_2 = 0.5$, because the energy left in storage at the end of period~2 can be discharged in period~3 without reducing the charge in period~4. Hence (C4) holds; so do (C1)--(C3), and the cut defines a facet ($\dim F_\Omega = 11 = 3T - 1$).

\begin{theorem}\label{th:facets}
Let Assumption~\ref{ass:interior_s0} hold, and let $D \subseteq \Omega \subseteq [T]$ with $|\Omega| \geq 2$. Then inequality~\eqref{eq:VIc_D} defines a facet of $\operatorname{conv}(\mathcal{P}^{0\text{-}1})$ if and only if $D_\Omega \subseteq D \subseteq D^+_\Omega$ and $(\Omega, D)$ satisfies (C1)--(C4). Moreover, inequality~\eqref{eq:VId_D} defines a facet of $\operatorname{conv}(\mathcal{P}^{0\text{-}1})$ if and only if the same conditions hold for the mirrored storage unit of Appendix~\ref{app:reflection}.
\end{theorem}

\begin{proof}{Proof}
The proof is given in Appendix~\ref{ec:proof_facets}. Necessity (Proposition~\ref{prop:necessity}) shows that a set $D$ outside the range~\eqref{eq:D_Omega} makes the cut the sum of another cut of the family and of a valid inequality (Proposition~\ref{prop:term_comparison}(a)--(b)), and that each condition that fails makes the cut lose dimension in one of the ways listed above: (C1) and (C2) yield a decomposition of the cut into other valid inequalities, whereas (C3) and (C4) yield an additional equation satisfied by $F_\Omega$. Sufficiency shows that, under (C1)--(C4), every equation satisfied by $F_\Omega$ is a multiple of the cut; it does so by exhibiting pairs of points of $F_\Omega$ that differ only in a few coordinates, typically because a small amount of energy is moved between two periods through the storage. The statement for~\eqref{eq:VId_D} follows by applying the statement for~\eqref{eq:VIc_D} to the mirrored unit, since the reflection of Lemma~\ref{lem:reflection} is an affine bijection that maps $\operatorname{conv}(\mathcal{P}^{0\text{-}1})$ onto the corresponding convex hull for the mirrored unit and~\eqref{eq:VId_D} onto~\eqref{eq:VIc_D}.
\hfill$\square$
\end{proof}

\subsection{Dominance over the Standard Storage Capacity Inequalities}
\label{sec:dominance}

The standard way of tightening the linear relaxation $\mathcal{S}^{\mathcal{R}}$ is the pair of capacity inequalities of~\citet{pozo2023convex}, which in the normalized units of Section~\ref{sec:Formulation} read
\begin{equation}\label{eq:pozo_cuts}
s_{t-1} + c_t \leq \overline{s} \qquad \text{and} \qquad s_{t-1} - d_t \geq \underline{s}, \qquad t \in [T].
\end{equation}
Their meaning is simple. In a point of $\mathcal{S}^{0\text{-}1}$, if the unit charges in period~$t$, then $d_t = 0$ and $s_t = s_{t-1} + c_t$, so the first inequality is the SoC upper bound $s_t \leq \overline{s}$; if it does not charge, then $c_t = 0$ and the inequality reduces to $s_{t-1} \leq \overline{s}$. The second inequality is the mirrored statement for discharging. Hence, both are valid for $\mathcal{S}^{0\text{-}1}$, and they are not implied by $\mathcal{S}^{\mathcal{R}}$: a fractional $u_t$ allows $c_t$ and $d_t$ to be positive at the same time, and then $s_t \leq \overline{s}$ no longer bounds $s_{t-1} + c_t$. Substituting $s_{t-1} = s_0 + \sum_{q < t}(c_q - d_q)$, inequalities~\eqref{eq:pozo_cuts} become the cumulative bounds $c([1,t]) - d([1,t-1]) \leq \overline{s} - s_0$ and $d([1,t]) - c([1,t-1]) \leq s_0 - \underline{s}$ in the space of $\mathcal{P}^{0\text{-}1}$, where they define the same faces by Remark~\ref{rem:lifting}. Their expression in physical units follows from the inverse scaling of Section~\ref{sec:Formulation}.

We show that the family of Proposition~\ref{prop:cut_strengthening} contains, for the sets $\Omega = [1,t]$ of all periods up to $t$, inequalities that dominate~\eqref{eq:pozo_cuts} and coincide with them only when the storage can be full, respectively empty, at the end of period $t-1$.

\begin{proposition}\label{prop:reachable_soc_bounds}
Let $t \in [T]$. For $\Omega = [1,t]$ and $D = [1,t-1]$, inequalities~\eqref{eq:VIc_D} and~\eqref{eq:VId_D} read
\begin{align}
s_{t-1} + c_t &\leq \overline{\sigma}_{t-1}\,(1 - u_t) + \overline{\sigma}_t\, u_t, \label{eq:U_star}\\
s_{t-1} - d_t &\geq \underline{\sigma}_{t-1}\, u_t + \underline{\sigma}_t\,(1 - u_t), \label{eq:L_star}
\end{align}
where $\overline{\sigma}_t$ and $\underline{\sigma}_t$ are the reachable SoC bounds~\eqref{eq:sigma_bounds}. In particular, both inequalities are valid for $\mathcal{S}^{0\text{-}1}$. Every point of $\mathcal{S}^{\mathcal{R}}$ that satisfies~\eqref{eq:U_star}, respectively~\eqref{eq:L_star}, satisfies the corresponding inequality in~\eqref{eq:pozo_cuts}. Moreover, \eqref{eq:U_star} coincides with the first inequality in~\eqref{eq:pozo_cuts} if and only if $\overline{\sigma}_{t-1} = \overline{s}$, and \eqref{eq:L_star} coincides with the second one if and only if $\underline{\sigma}_{t-1} = \underline{s}$; otherwise they are strictly stronger.
\end{proposition}

The proof, given in Appendix~\ref{ec:proof_reachable}, computes the coefficients of the cut for $\Omega = [1,t]$: the greedy rule charges from the start, so $f([1,q]) = \overline{\sigma}_q - s_0$ for every $q$, which gives $\pi_t = -(\overline{\sigma}_t - \overline{\sigma}_{t-1})$, and the right-hand side of~\eqref{eq:U_star} is a convex combination of $\overline{\sigma}_{t-1}$ and $\overline{\sigma}_t$, both at most $\overline{s}$.

In words, \eqref{eq:pozo_cuts} bound the SoC by the physical capacity in every period, whereas~\eqref{eq:U_star} and~\eqref{eq:L_star} use the capacity that the SoC dynamics actually allow: early in the horizon, the storage may not be able to be full or empty yet, the same refinement that the singleton cuts~\eqref{eq:singleton_cuts} apply to the power limits. Note that the choice $D = [1,t-1]$ may include periods with $\pi_q \leq 0$, which the convention of Remark~\ref{rem:dominating_cuts_convention} keeps outside $D$; by Theorem~\ref{th:facets}, \eqref{eq:U_star} then defines no facet, but it remains valid, because Proposition~\ref{prop:cut_strengthening} holds for every $D \subseteq \Omega$. Since \eqref{eq:U_star} and~\eqref{eq:L_star} are only $2T$ inequalities, they are added a priori, together with the singleton cuts, to all the models that use the proposed inequalities in Section~\ref{sec:case_study}, where they replace~\eqref{eq:pozo_cuts}.

\paragraph{The role of $D$.} For the same set $\Omega = [1,\bar{t}]$, the family~\eqref{eq:VIc_D} contains one cut for each choice of $D$, and~\eqref{eq:U_star} is the one with $D = [1,\bar{t}-1]$. The next result shows which choices of $D$ keep the dominance over~\eqref{eq:pozo_cuts}. Recall that, by~\eqref{eq:gain_formula}, $\pi_{\bar{t}} = -x^{(\Omega)}_{\bar{t}} \leq 0$, because no period of $\Omega$ follows $\bar{t}$.

\begin{corollary}\label{cor:initial_dominance}
Let $\Omega = [1,\bar{t}]$ and let $D \subseteq [1,\bar{t}-1]$ contain every period $\tau \in \Omega$ with $\pi_\tau > 0$. Then every point of $\mathcal{S}^{\mathcal{R}}$ that satisfies the cut~\eqref{eq:VIc_D} satisfies the first inequality in~\eqref{eq:pozo_cuts} for $t = \bar{t}$. The mirrored statement holds for~\eqref{eq:VId_D} and the second inequality in~\eqref{eq:pozo_cuts}.
\end{corollary}

\begin{proof}{Proof}
In~\eqref{eq:VIc_D}, $f(\Omega) = \overline{\sigma}_{\bar{t}} - s_0 \leq \overline{s} - s_0$, $d(D) \leq d([1,\bar{t}-1])$ because $D \subseteq [1,\bar{t}-1]$ and $\boldsymbol{d} \geq \boldsymbol{0}$, and $\pi_\tau (1 - u_\tau) \leq 0$ for every $\tau \in \Omega \setminus D$, because $\pi_\tau \leq 0$ there and $u_\tau \leq 1$. Hence, the right-hand side of~\eqref{eq:VIc_D} is at most $\overline{s} - s_0 + d([1,\bar{t}-1])$, which is the right-hand side of the first inequality in~\eqref{eq:pozo_cuts} written in the space of $\mathcal{P}^{0\text{-}1}$.
\hfill$\square$
\end{proof}

In words, the periods with $\pi_\tau \leq 0$ can be credited through their binary variables without losing the dominance, because their terms only tighten the cut; the periods with $\pi_\tau > 0$ must be credited through their discharges $d_\tau$, as in~\eqref{eq:pozo_cuts}. The convention of Remark~\ref{rem:dominating_cuts_convention} puts the periods with $\pi_\tau = \overline{d}_\tau$ in $D$, but allows the periods with $0 < \pi_\tau < \overline{d}_\tau$ to stay outside it. For those, crediting the gain $\pi_\tau$ through $u_\tau$ and crediting the discharge $d_\tau$ are two different things, and neither implies the other on the linear relaxation: the cut with $D = D_\Omega$ and~\eqref{eq:U_star} are incomparable in general, and Appendix~\ref{ec:example_incomparable} gives an instance with $T = 2$ and two points of $\mathcal{S}^{\mathcal{R}}$, each of which violates exactly one of them. The two ways of crediting are therefore complementary: the family contains a cut that dominates~\eqref{eq:pozo_cuts}, by Corollary~\ref{cor:initial_dominance}, and cuts that are not comparable with it, and the separation algorithm of Section~\ref{sec:separation_algorithm} chooses $D$ for each candidate $\Omega$ so as to maximize the violation.

For sets $\Omega$ that are not of the form $[1,\bar{t}]$, the cuts~\eqref{eq:VIc_D} have no counterpart among the SoC bounds: for a set $\Omega$ with gaps, i.e., periods of $[l,k] \setminus \Omega$ between periods of $\Omega$, the right-hand side $f(\Omega)$ and the gains $\pi_\tau$ account for the discharges that the greedy rule performs in the gaps, which no inequality in $\boldsymbol{s}$ alone can express. Theorem~\ref{th:facets} characterizes when these cuts define facets, and Section~\ref{sec:separation_algorithm} shows how to separate them.

\subsection{Separation and Strengthening of the Cuts}
\label{sec:separation_algorithm}

Theorem~\ref{th:LVI} and Proposition~\ref{prop:cut_strengthening} define exponentially many inequalities, so they are used within a cutting-plane loop: the relaxation is solved, violated cuts are identified and added, and the process is repeated. Throughout this section, $\boldsymbol{y}^* = (\boldsymbol{d}^*, \boldsymbol{c}^*, \boldsymbol{u}^*)$ denotes the current solution of the relaxation, which contains the linear relaxation $\mathcal{S}^{\mathcal{R}}$, the singleton cuts~\eqref{eq:singleton_cuts}, the reachable SoC bounds~\eqref{eq:U_star}--\eqref{eq:L_star}, and the cuts generated so far. We write $\boldsymbol{x}^* := \boldsymbol{c}^* - \boldsymbol{d}^*$ for its net-charging vector and $\boldsymbol{s}^*$ for its SoC trajectory. The \emph{separation problem} asks for sets $D \subseteq \Omega \subseteq [T]$ such that~\eqref{eq:VIc_D}, or its mirrored counterpart~\eqref{eq:VId_D}, is violated by $\boldsymbol{y}^*$, that is, such that the \emph{violation}
\begin{equation}\label{eq:violation}
v_{\Omega,D}(\boldsymbol{y}^*) := c^*(\Omega) - d^*(D) - \sum_{\tau \in \Omega \setminus D} \pi_\tau\,(1 - u^*_\tau) - f(\Omega)
\end{equation}
is positive. Since the discharging cuts are the charging cuts of the mirrored unit of Appendix~\ref{app:reflection}, we describe the procedure for the charging cuts only: the discharging cuts are separated by running the same procedure on the mirrored unit, with $\boldsymbol{y}^*$ replaced by its reflection $\Theta(\boldsymbol{y}^*) = (\boldsymbol{c}^*, \boldsymbol{d}^*, \boldsymbol{1} - \boldsymbol{u}^*)$, and reflecting the resulting cuts back.

Two observations guide the design. First, since $\boldsymbol{x}^*$ satisfies the SoC constraints, $x^*(\Omega) \leq f(\Omega)$ for every $\Omega$ (Proposition~\ref{prop:submodularity}(i)), so
\begin{equation}\label{eq:violation_split}
v_{\Omega,D}(\boldsymbol{y}^*) = \underbrace{\big[x^*(\Omega) - f(\Omega)\big]}_{\leq 0} + \sum_{\tau \in \Omega \setminus D} \big[d^*_\tau - \pi_\tau\,(1 - u^*_\tau)\big]:
\end{equation}
a cut can only be violated through periods of $\Omega \setminus D$ in which $\boldsymbol{y}^*$ discharges more than the credit $\pi_\tau (1 - u^*_\tau)$ that the cut grants for not committing to the charging mode. Second, the purpose of the cuts is to remove the periods with simultaneous charging and discharging. Let
\[
\mathcal{J}(\boldsymbol{y}^*) := \{t \in [T] : c^*_t > 0 \text{ and } d^*_t > 0\}.
\]
If $\mathcal{J}(\boldsymbol{y}^*) = \emptyset$, then setting $u_t = 1$ when $c^*_t > 0$ and $u_t = 0$ when $d^*_t > 0$ turns $\boldsymbol{y}^*$ into a point of $\mathcal{P}^{0\text{-}1}$ with the same power profile, which is therefore optimal for the original problem: the loop stops. Otherwise, the separation looks for violated cuts whose set $\Omega$ contains at least one period of $\mathcal{J}(\boldsymbol{y}^*)$, i.e., $\Omega \cap \mathcal{J}(\boldsymbol{y}^*) \neq \emptyset$; we then say that $\Omega$ \emph{meets} $\mathcal{J}(\boldsymbol{y}^*)$. The other sets can be discarded: their cuts can only be violated through the fractional modes of $\boldsymbol{y}^*$, and their robust violation, defined in Section~\ref{sec:choice_of_D}, is never positive (Lemma~\ref{lem:robust_no_J}).

The coefficients $\pi_\tau = \pi_\tau(\Omega)$ depend on $\Omega$, so $v_{\Omega,D}$ is, in general, neither a submodular nor a supermodular function of $\Omega$, and the tools of submodular optimization do not apply directly to the separation problem. Since linear optimization over $\mathcal{P}^{0\text{-}1}$ is polynomial~\citep{bakhshi2023polynomial}, so is, in principle, separation over $\operatorname{conv}(\mathcal{P}^{0\text{-}1})$ by the equivalence of optimization and separation~\citep[Theorem~6.4.9]{grotschel1988geometric}; our aim is a fast combinatorial separation of the explicit family. Our procedure is exact over the contiguous sets $\Omega = [l,k]$ within the cycling envelopes of Section~\ref{sec:search} and relies on dominance arguments and local search for the sets with gaps. Sections~\ref{sec:two_passes}--\ref{sec:strengthening} describe its ingredients, and the strengthening of the cuts already generated when the search finds nothing although $\mathcal{J}(\boldsymbol{y}^*) \neq \emptyset$.

\subsubsection{All the Coefficients of a Set in Linear Time}
\label{sec:two_passes}

Fix $\Omega$ with $l := \min \Omega$ and $k := \max \Omega$, and let $\boldsymbol{s} := \boldsymbol{s}^{(\Omega)}$ and $\boldsymbol{x} := \boldsymbol{x}^{(\Omega)}$ be the greedy SoC trajectory and net charging for $\Omega$ on $[T]$ from $s_0$. Recall that $s_{l-1} = \underline{\sigma}_{l-1}$, that $f(\Omega) = x(\Omega)$ (Lemma~\ref{lem:greedy_polymatroid}), and that the thresholds $\kappa_t$ of Lemma~\ref{lem:thresholds} satisfy $\kappa_t = \overline{s}$ for $t \geq k$ and, for $t = k, k-1, \ldots, l$,
\begin{equation}\label{eq:kappa_backward}
\kappa_{t-1} = \max\{\kappa_t - \overline{c},\ \underline{s}\} \ \text{ if } t \in \Omega, \qquad \kappa_{t-1} = \min\{\kappa_t + \overline{d},\ \overline{s}\} \ \text{ if } t \notin \Omega.
\end{equation}
The forward recursion~\eqref{eq:greedy_rule} and the backward recursion~\eqref{eq:kappa_backward} are the two passes of the following result.

\begin{proposition}\label{prop:two_passes}
For every $\tau \in \Omega$,
\begin{equation}\label{eq:pi_two_passes}
\pi_\tau = \big(\phi^+(s_{\tau-1}) - \kappa_\tau\big)^+ - \big(\phi^-(s_{\tau-1}) - \kappa_\tau\big)^+ - x_\tau,
\end{equation}
where $\phi^+(\sigma) = \min\{\sigma + \overline{c}, \overline{s}\}$ and $\phi^-(\sigma) = \max\{\sigma - \overline{d}, \underline{s}\}$. Consequently, $f(\Omega)$, all the gains $\pi_\tau$, the sets $D_\Omega$ and $D^+_\Omega$ of~\eqref{eq:D_Omega}, and the conditions (C1) and (C2) of Definition~\ref{def:facet_conditions} are obtained from one forward pass of the greedy rule over $[l,k]$ and one backward pass of~\eqref{eq:kappa_backward}, in $\mathcal{O}(k - l + 1)$ time.
\end{proposition}

The proof, given in Appendix~\ref{ec:proof_two_passes}, substitutes the threshold form of the value function of Lemma~\ref{lem:thresholds} into the gain formula~\eqref{eq:gain_formula}.

Formula~\eqref{eq:pi_two_passes} has the interpretation of Section~\ref{sec:facet_defining_conditions}: relaxing $\tau$ moves the SoC from $\phi^+(s_{\tau-1})$ down to $\phi^-(s_{\tau-1})$, and only the part of this range that lies above the threshold $\kappa_\tau$ is recovered as later charge; the charge $x_\tau$ of period $\tau$ itself is lost. Proposition~\ref{prop:two_passes} replaces the $|\Omega|$ greedy passes of the computation following~\eqref{eq:gain_formula} by two passes, and it is the workhorse of the search: every candidate $\Omega$ costs $\mathcal{O}(T)$ operations. Figure~\ref{fig:two_passes} illustrates the two passes on a small instance, which is used again in Section~\ref{sec:strengthening}.

\begin{figure}
\centering
\begin{tikzpicture}[x=2.4cm, y=3.4cm, >=stealth]
    \draw[->] (-0.25,0) -- (3.5,0) node[right]{$t$};
    \draw[->] (0,-0.08) -- (0,1.22) node[above]{SoC};
    \foreach \t in {0,1,2,3} { \draw (\t,0.015) -- (\t,-0.015) node[below]{$\t$}; }
    \draw (0.04,1) -- (-0.04,1) node[left]{$\overline{s} = 1$};
    \draw (0.04,0.5) -- (-0.04,0.5) node[left]{$0.5$};
    \node[left] at (-0.04,0) {$\underline{s} = 0$};
    \draw[dotted] (0,1) -- (3.4,1);
    \draw[very thick] (0,0.5) -- (1,1) -- (2,0.5) -- (3,1);
    \fill (0,0.5) circle (1.6pt); \fill (1,1) circle (1.6pt); \fill (2,0.5) circle (1.6pt); \fill (3,1) circle (1.6pt);
    \node[above left] at (0.6,0.78) {\small $x_1 = 0.5$};
    \node[below left] at (1.42,0.78) {\small $x_2 = -0.5$};
    \node[above left] at (2.62,0.78) {\small $x_3 = 0.5$};
    \draw[thick, dashed] (3,1) -- (2,0) -- (1,0.5) -- (0,0);
    \foreach \p in {(3,1), (2,0), (1,0.5), (0,0)} { \draw[thick] \p circle (2pt); }
    \node[right] at (3.06,1.0) {\small $\kappa_3 = \overline{s}$};
    \node[below right] at (2.15,0.0) {\small $\kappa_2 = 0$};
    \node[right] at (1.06,0.5) {\small $\kappa_1 = 0.5$};
    \node[below right] at (0.15,0.0) {\small $\kappa_0 = 0$};
    \draw[very thick] (0.15,1.17) -- (0.45,1.17) node[right]{\small $s^{(\Omega)}_t$ (forward pass)};
    \draw[thick, dashed] (1.75,1.17) -- (2.05,1.17) node[right]{\small $\kappa_t$ (backward pass)};
\end{tikzpicture}
\caption{The two passes of Proposition~\ref{prop:two_passes} for $T = 3$, $[\underline{s}, \overline{s}] = [0,1]$, $s_0 = 0.5$, $\overline{c} = 1$, $\overline{d} = 0.5$, and $\Omega = \{1,3\}$. Solid: the greedy trajectory $s^{(\Omega)}$ (forward pass), with $f(\Omega) = 1$. Dashed: the thresholds $\kappa_t$ (backward pass from $\kappa_3 = \overline{s}$). By~\eqref{eq:pi_two_passes}, $\pi_1 = (1 - 0.5)^+ - 0.5 = 0$ and $\pi_3 = 0 - 0.5 = -0.5$, so $D_\Omega = D^+_\Omega = \emptyset$, and the cut, with $D = \emptyset$, reads $c_1 + c_3 + 0.5\,(1 - u_3) \leq 1$.}
\label{fig:two_passes}
\end{figure}

\subsubsection{The Choice of $D$}
\label{sec:choice_of_D}

For a given $\Omega$, the violation~\eqref{eq:violation} is maximized by choosing $D$ period by period: $\tau \in D$ if and only if $d^*_\tau < \pi_\tau (1 - u^*_\tau)$. This choice respects Remark~\ref{rem:dominating_cuts_convention}: a period with $\pi_\tau \leq 0$ has $\pi_\tau(1 - u^*_\tau) \leq 0 \leq d^*_\tau$ and stays outside $D$, and a period with $\pi_\tau = \overline{d}_\tau$ has $d^*_\tau \leq \overline{d}_\tau (1 - u^*_\tau)$ by the singleton cut, so it enters $D$ unless the singleton cut is tight at $\boldsymbol{y}^*$, in which case both choices give the same violation; we resolve these ties in favor of $D$, which keeps $D$ within the range~\eqref{eq:D_Omega}, so that Theorem~\ref{th:facets} applies to every cut generated. The mode $\boldsymbol{u}^*$, however, is the least reliable part of $\boldsymbol{y}^*$: it often does not appear in the objective, as in the arbitrage problem of Section~\ref{sec:case_study}, and then the solver may return any $u^*_\tau$ in the interval $[c^*_\tau/\overline{c}_\tau,\ 1 - d^*_\tau/\overline{d}_\tau]$ allowed by the singleton cuts. We therefore also use the \emph{robust violation}, the smallest violation over this interval,
\begin{equation}\label{eq:robust_D}
\bar{v}_{\Omega,D}(\boldsymbol{y}^*) := c^*(\Omega) - d^*(D) - \sum_{\tau \in \Omega \setminus D:\, \pi_\tau > 0} \pi_\tau\Big(1 - \frac{c^*_\tau}{\overline{c}_\tau}\Big) - \sum_{\tau \in \Omega \setminus D:\, \pi_\tau \leq 0} \pi_\tau\,\frac{d^*_\tau}{\overline{d}_\tau} - f(\Omega),
\end{equation}
whose pointwise-optimal $D$ places $\tau$ in $D$ if and only if $d^*_\tau < \pi_\tau(1 - c^*_\tau/\overline{c}_\tau)$. Both choices give valid cuts, which differ only in the periods with $0 < \pi_\tau < \overline{d}_\tau$ of Proposition~\ref{prop:term_comparison}(c), and the procedure adds both whenever both are violated at $\boldsymbol{y}^*$. To decide which candidate sets $\Omega$ to explore, however, the search of Section~\ref{sec:search} compares them by their robust violation, which does not depend on $\boldsymbol{u}^*$.

\subsubsection{Two Moves That Do Not Decrease the Violation}
\label{sec:dominance_moves}

The search starts from windows $[l,k]$ and removes periods from them. Two kinds of removal are safe, in the sense that they cannot decrease the violation. The first one is the decomposition behind the necessity of (C1) and (C2) in Theorem~\ref{th:facets}.

\begin{lemma}\label{lem:decomposition_moves}
Let $\Omega$ with $l = \min \Omega$ and $k = \max \Omega$, and let $v_\Omega$ denote the violation, or the robust violation~\eqref{eq:robust_D}, of the cut~\eqref{eq:VIc_D} at $\boldsymbol{y}^*$ with the pointwise-optimal $D$ of Section~\ref{sec:choice_of_D}.
\begin{enumerate}[(a)]
    \item If (C1) fails at $t = l-1$, the cut is the sum of the singleton cuts $c_q \leq \overline{c}_q u_q$, $q \in \Omega$, and is never violated.
    \item If (C1) fails at a gap $t \in [l,k-1] \setminus \Omega$, then $v_\Omega \leq v_{\Omega \cap [l,t]}$.
    \item If (C2) fails at a gap $t$, then $v_\Omega \leq v_{\Omega \cap [l,t]} + v_{\Omega \cap [t+1,k]}$.
\end{enumerate}
\end{lemma}

The proof, given in Appendix~\ref{ec:proof_moves}, rests on Lemma~\ref{lem:decomposition}: the violation of a sum of inequalities is the sum of their violations, and the singleton cuts are satisfied by $\boldsymbol{y}^*$.

The second move removes periods that $\boldsymbol{y}^*$ does not use for charging. Removing $t$ from $\Omega$ loses the term $c^*_t$ on the left-hand side, but it also changes the right-hand side: the capacity grows by $\pi_t$, the term of $t$ disappears, and the gains $\pi_\tau$ of the other periods can only decrease: by submodularity, the marginal value $f(\tau \mid \Omega \setminus \{\tau\})$ can only increase when the context loses $t$, and $\pi_\tau$ is its negative.

\begin{lemma}\label{lem:free_removal}
Let $t \in \Omega$ with $|\Omega| \geq 3$ and $c^*_t = 0$, and let $\pi_t = \pi_t(\Omega)$. If $\pi_t \leq 0$, or if $u^*_t = 0$ and $d^*_t \geq \pi_t$, then $v_{\Omega \setminus \{t\}} \geq v_\Omega$, where both violations use the pointwise-optimal $D$. The same holds for the robust violation~\eqref{eq:robust_D}.
\end{lemma}

The proof, given in Appendix~\ref{ec:proof_moves}, compares the two violations term by term: removing $t$ adds $\pi_t$ to the capacity, and, by submodularity, the gains $\pi_\tau$ of the remaining periods can only decrease. So can the credits $\min\{d^*_\tau, \pi_\tau (1 - u^*_\tau)\}$ that the cut grants for them, and since the credits reduce the violation, the violation can only increase.

We call a period satisfying the hypotheses of Lemma~\ref{lem:free_removal} \emph{removable for $\boldsymbol{y}^*$}. In particular, a period in which $\boldsymbol{y}^*$ discharges at the maximum rate $\overline{d}_t$ is removable, because $u^*_t = 0$ by the singleton cut and $\pi_t \leq \overline{d}_t$ by~\eqref{eq:pi_bound}; so is every period with $c^*_t = 0$ and $\pi_t \leq 0$. Since the gains of the remaining periods can only decrease, removing a removable period keeps the other removable periods removable, so all of them can be removed at once, and it may make further periods removable, so the removal is iterated. Note that the periods of $\mathcal{J}(\boldsymbol{y}^*)$ are never removable.

\subsubsection{The Search}
\label{sec:search}

Algorithm~\ref{alg:separation} summarizes the procedure. For $t \in \mathcal{J}(\boldsymbol{y}^*)$, the \emph{cycling envelope} of $t$ is $\mathcal{E}_t := [l_t, k_t]$, with $l_t := 1 + \max\{\tau \in [0, t-1] : s^*_\tau = \underline{s}\}$ and $k_t := \min\{\tau \in [t, T] : s^*_\tau = \overline{s}\}$, where $\max \emptyset := 0$ and $\min \emptyset := T$. In words, $\mathcal{E}_t$ starts right after the storage was last empty before $t$ and ends when it next becomes full: between an empty and a full storage, the relaxed solution cycles energy, and the cuts that remove this cycling live in $\mathcal{E}_t$. For a set $\Omega$, \textsc{Settle}$(\Omega)$ applies the moves of Lemmas~\ref{lem:decomposition_moves} and~\ref{lem:free_removal} until no move applies, evaluates every set visited with Proposition~\ref{prop:two_passes}, records the violated cuts of the sets that satisfy (C1) and (C2), and returns the sets at which it stopped. We call a set \emph{settled} if none of the moves of \textsc{Settle} applies to it; \textsc{Settle}$(\Omega)$ thus returns settled sets and, if the robust violation of $\Omega$ is positive, so is that of at least one of them (Proposition~\ref{prop:separation_properties}(iii)). Stage~1 evaluates every window $[l,k]$ that contains some $t \in \mathcal{J}(\boldsymbol{y}^*)$ and is contained in $\mathcal{E}_t$; for a window, the greedy trajectory and the thresholds are explicit, $s_{\tau-1} = \min\{\underline{\sigma}_{l-1} + (\tau - l)\,\overline{c},\ \overline{s}\}$ and $\kappa_\tau = \max\{\overline{s} - (k - \tau)\,\overline{c},\ \underline{s}\}$, so all windows, their removable periods included, are evaluated in $\mathcal{O}(T^3)$ vectorized operations. Stage~2 settles every such window, starting from the window minus its removable periods, which is where the sets with gaps are produced: a window whose cut is not violated typically contains periods in which $\boldsymbol{y}^*$ discharges at full rate, and removing them yields, by Lemma~\ref{lem:free_removal}, a violation at least as large. Stage~3 is a local search, run only if Stages~1--2 found no cut of the family being separated (the discharging cuts are separated in the same way on the mirrored unit): starting from the settled envelopes and from the $n$ settled sets with the largest robust violation around each $t \in \mathcal{J}(\boldsymbol{y}^*)$, where $n$ is a parameter ($n = 2$ in Section~\ref{sec:case_study}), it repeatedly removes the period whose removal increases the robust violation most, settling the result each time. No set is evaluated twice, and all the violated cuts found are added to the relaxation (Section~\ref{sec:case_study} compares this choice with adding one cut per period).

\begin{algorithm}[htpb]
\caption{Separation of the charging cuts at $\boldsymbol{y}^*$}
\label{alg:separation}
\begin{algorithmic}[1]
\Require $\boldsymbol{y}^* \in \mathcal{S}^{\mathcal{R}}_1$ with $\mathcal{J}(\boldsymbol{y}^*) \neq \emptyset$, tolerance $\varepsilon > 0$, and number $n$ of starting sets per period.
\Ensure The set $\mathcal{K}$ of violated cuts found.
\State $\mathcal{K} \gets \emptyset$; \ $\mathcal{M} \gets \emptyset$ \Comment{$\mathcal{M}$: sets already evaluated (memoization)}
\Procedure{Evaluate}{$\Omega$} \Comment{Proposition~\ref{prop:two_passes}; adds $\Omega$ to $\mathcal{M}$}
    \State compute $f(\Omega)$, $(\pi_\tau)_{\tau \in \Omega}$, the first $t'$ where (C1) or (C2) fails, and the removable periods $Z$
    \State for the optimal $D$ of~\eqref{eq:violation} and~\eqref{eq:robust_D}: if (C1)--(C2) hold and $v_{\Omega,D}(\boldsymbol{y}^*) > \varepsilon$, add $(\Omega, D)$ to $\mathcal{K}$
\EndProcedure
\Procedure{Settle}{$\Omega$}
    \State \textbf{if} $|\Omega| < 2$ or $\Omega$ does not meet $\mathcal{J}(\boldsymbol{y}^*)$ \textbf{then return} $\emptyset$; \textbf{if} $\Omega \notin \mathcal{M}$ \textbf{then} \Call{Evaluate}{$\Omega$}
    \If{(C1) fails at $t' = l-1$} \Return $\emptyset$ \Comment{Lemma~\ref{lem:decomposition_moves}(a)}
    \ElsIf{(C1) fails at a gap $t'$} \Return \Call{Settle}{$\Omega \cap [l,t']$} \Comment{Lemma~\ref{lem:decomposition_moves}(b)}
    \ElsIf{(C2) fails at a gap $t'$} \Comment{Lemma~\ref{lem:decomposition_moves}(c)}
        \State \Return \Call{Settle}{$\Omega \cap [l,t']$} $\cup$ \Call{Settle}{$\Omega \cap [t'+1,k]$}
    \ElsIf{$Z \neq \emptyset$ and $|\Omega \setminus Z| \geq 2$} \Comment{Lemma~\ref{lem:free_removal}}
        \State \Return \Call{Settle}{$\Omega \setminus Z$}
    \Else\ \Return $\{\Omega\}$
    \EndIf
\EndProcedure
\State \textbf{Stage 1.} \Call{Evaluate}{$[l,k]$} in closed form for every window with $t \in [l,k] \subseteq \mathcal{E}_t$, $t \in \mathcal{J}(\boldsymbol{y}^*)$
\State \textbf{Stage 2.} $\mathcal{I} \gets \bigcup_{[l,k]} \Call{Settle}{[l,k] \setminus Z_{[l,k]}}$ over the same windows
\State \textbf{Stage 3} (only if $\mathcal{K} = \emptyset$). For every $t \in \mathcal{J}(\boldsymbol{y}^*)$: let $\mathcal{W}_t$ contain the sets of \Call{Settle}{$\mathcal{E}_t$} and the $n$ sets of $\mathcal{I}$ containing $t$ with the largest robust violation; $\mathcal{V} \gets \emptyset$ \Comment{$\mathcal{V}$: sets explored}
\For{each $\Omega_0 \in \bigcup_t \mathcal{W}_t$, by increasing size}
    \While{$\mathcal{U} \neq \emptyset$, starting from $\mathcal{U} := \{\Omega_0\}$} \Comment{$\mathcal{U}$: sets to explore}
        \State remove from $\mathcal{U}$ the set $\Omega$ added last; \textbf{if} $\Omega \in \mathcal{V}$ \textbf{then} skip it, \textbf{else} add it to $\mathcal{V}$
        \If{$\Omega$ is not settled} add the sets of \Call{Settle}{$\Omega$} to $\mathcal{U}$
        \Else\ among the $t \in \Omega$ such that $\Omega \setminus \{t\} \notin \mathcal{V}$ has at least two periods and meets $\mathcal{J}(\boldsymbol{y}^*)$, pick the one maximizing $\bar{v}_{\Omega \setminus \{t\}}$ ($-\infty$ if (C1) or (C2) fails), and add $\Omega \setminus \{t\}$ to $\mathcal{U}$
        \EndIf
    \EndWhile
\EndFor
\State \Return $\mathcal{K}$
\end{algorithmic}
\end{algorithm}

\begin{proposition}\label{prop:separation_properties}
Algorithm~\ref{alg:separation} has the following properties.
\begin{enumerate}[(i)]
    \item Every cut it returns is valid for $\mathcal{P}^{0\text{-}1}$ and violated by $\boldsymbol{y}^*$ by more than $\varepsilon$.
    \item The search is exact over the windows scanned in Stage~1: if some cut~\eqref{eq:VIc_D} with $\Omega = [l,k]$ such a window and $D \subseteq \Omega$ is violated by more than $\varepsilon$, the algorithm returns a cut for $[l,k]$ violated by more than $\varepsilon$. Stage~1 scans the windows within the cycling envelopes; scanning instead all the windows that meet $\mathcal{J}(\boldsymbol{y}^*)$, which keeps the $\mathcal{O}(T^3)$ cost, makes the search exact over all of them.
    \item If a set $\Omega$ reached by the algorithm has a positive robust violation, so does one of the sets at which \textsc{Settle}$(\Omega)$ stops, and its robust violation is at least that of $\Omega$ divided by the number of parts into which $\Omega$ was split.
    \item Each set is evaluated at most once, in $\mathcal{O}(T)$ time. Stage~1 costs $\mathcal{O}(T^3)$ vectorized operations; Stage~2 evaluates $\mathcal{O}(T)$ sets per window, hence $\mathcal{O}(T^3)$ sets; and Stage~3 evaluates $\mathcal{O}\big((n + |\mathcal{J}(\boldsymbol{y}^*)|)\,|\mathcal{J}(\boldsymbol{y}^*)|\,T^2\big)$ sets. The algorithm thus runs in $\mathcal{O}\big((n + T)\,T^4\big)$ time in the worst case.
\end{enumerate}
\end{proposition}

The proof is given in Appendix~\ref{ec:proof_separation_properties}.

Within the cutting-plane loop, every round that does not stop adds at least one cut violated by the current solution, hence a new cut of a finite family, so the loop terminates.

\subsubsection{Strengthening When the Search Stops Early}
\label{sec:strengthening}

Algorithm~\ref{alg:separation} may return no cut although $\mathcal{J}(\boldsymbol{y}^*) \neq \emptyset$. In that case, $\boldsymbol{y}^*$ typically satisfies several of the cuts already in the relaxation with equality, and the mechanism is often the one described by condition (C3): $\boldsymbol{y}^*$ charges in a period $t$ outside the set $\Omega$ of one of these tight cuts, whereas no point of the face $F_\Omega$ of $\operatorname{conv}(\mathcal{P}^{0\text{-}1})$ charges in $t$. The cut can then be strengthened by \emph{lifting} the variable $c_t$, that is, by adding a term $\alpha\,c_t$ to its left-hand side with the largest valid coefficient $\alpha$. The next result gives this coefficient in closed form. It is built from the greedy rule and from the relaxation deficits
\begin{equation}\label{eq:deficit}
\xi(R) := f(\Omega) + \sum_{\tau \in R} \pi_\tau - f(\Omega \setminus R) \geq 0, \qquad R \subseteq \Omega \setminus D,
\end{equation}
which are nonnegative by Lemma~\ref{lem:submodular_bound} and vanish exactly on the family $\mathcal{R}_\Omega$ of Section~\ref{sec:facet_defining_conditions}. We also need the greedy rule with period-dependent bounds: Lemma~\ref{lem:greedy_polymatroid} and its proof hold verbatim when the box $[-\overline{d}, \overline{c}]$ is replaced by period-dependent boxes $[\underline{x}_\tau, \overline{x}_\tau]$ with $\underline{x}_\tau \leq 0 \leq \overline{x}_\tau$, as in~\eqref{eq:X_u}, with $\phi^\pm$ replaced by $\phi^+_\tau(\sigma) = \min\{\sigma + \overline{x}_\tau, \overline{s}\}$ and $\phi^-_\tau(\sigma) = \max\{\sigma + \underline{x}_\tau, \underline{s}\}$; the resulting set of net-charging sequences is again a g-polymatroid~\citep[Theorem~2]{11369438}.

Fix $\Omega$, $D \subseteq \Omega$, and $t \in [T] \setminus \Omega$. For $R \subseteq \Omega \setminus D$, let $B_R := \Omega \setminus R$ and let
\[
\mathcal{X}^R_t := \Big\{\boldsymbol{x} \in \mathcal{X}_{[T]}(s_0) : x_\tau \leq 0 \ (\tau \in R),\ \ x_\tau \geq 0 \ \big(\tau \in (\Omega \setminus (D \cup R)) \cup \{t\}\big)\Big\}
\]
be the net-charging sequences of the points of $\mathcal{P}^{0\text{-}1}$ that relax exactly the periods of $R$, charge (or idle) in the other periods of $\Omega \setminus D$, and charge (or idle) in $t$; it is a g-polymatroid, with rank function $b^R_t(A) := \max\{x(A) : \boldsymbol{x} \in \mathcal{X}^R_t\}$ computed by the greedy rule with the corresponding bounds. Define
\begin{equation}\label{eq:lifting_quantities}
\begin{aligned}
\gamma^R_t &:= f(B_R \cup \{t\}) - b^R_t(B_R), & \overline{c}^{\,R}_t &:= b^R_t(\{t\}), & \xi_t(R) &:= \xi(R) + f(B_R) - b^R_t(B_R).
\end{aligned}
\end{equation}
In words, $b^R_t(B_R) \leq f(B_R)$ is the largest charge of $B_R$ when period $t$ is not allowed to discharge, $\gamma^R_t \geq 0$ is the additional charge that $t$ contributes in that situation, $\overline{c}^{\,R}_t \leq \overline{c}_t$ is the largest charge of $t$ when the periods of $\Omega \setminus (D \cup R)$ are not allowed to discharge, and $\xi_t(R) \geq \xi(R)$ is the slack of the cut on the points of pattern $R$ that do not discharge in $t$. All four quantities are obtained from three greedy passes.

\begin{proposition}\label{prop:lifting}
Let $\Omega \subseteq [T]$, $D \subseteq \Omega$, and $t \in [T] \setminus \Omega$. For $R \subseteq \Omega \setminus D$ define
\begin{equation}\label{eq:alpha_R}
\alpha_R :=
\begin{cases}
\xi_t(R) / \gamma^R_t & \text{if } \gamma^R_t > \xi_t(R),\\[2pt]
1 + \big(\xi_t(R) - \gamma^R_t\big) / \overline{c}^{\,R}_t & \text{if } \gamma^R_t \leq \xi_t(R) \text{ and } \overline{c}^{\,R}_t > 0,\\[2pt]
+\infty & \text{otherwise},
\end{cases}
\qquad \alpha^*_t := \min_{R \subseteq \Omega \setminus D} \alpha_R.
\end{equation}
Then the lifted inequality
\begin{equation}\label{eq:lifted_cut}
\alpha\,c_t + \sum_{\tau \in \Omega} c_\tau \leq f(\Omega) + \sum_{\tau \in D} d_\tau + \sum_{\tau \in \Omega \setminus D} \pi_\tau\,(1 - u_\tau)
\end{equation}
is valid for $\mathcal{P}^{0\text{-}1}$ if and only if $\alpha \leq \alpha^*_t$. Moreover, $\alpha^*_t < \infty$ under Assumption~\ref{ass:interior_s0}, $\alpha_R \geq \min\{1, \xi(R)/\overline{c}_t\}$ for every $R$, and $\xi$ is nondecreasing with respect to inclusion.
\end{proposition}

The proof, given in Appendix~\ref{ec:lifting}, fixes the relaxation pattern $R$ of a point with $u_t = 1$ and reduces the validity of~\eqref{eq:lifted_cut} on those points to a linear maximization over the g-polymatroid $\mathcal{X}^R_t$, with weight $1$ on $B_R$ and weight $\alpha$ on $t$. The greedy algorithm for g-polymatroids solves it in closed form; its value is piecewise linear in $\alpha$ with a breakpoint at $\alpha = 1$, where the order of the two weight classes changes, and this yields the two finite cases of~\eqref{eq:alpha_R}. Physically, a small coefficient ($\alpha \leq 1$) is limited by the charge $\gamma^R_t$ that $t$ can add without disturbing $B_R$, and a large one by the full charge $\overline{c}^{\,R}_t$ of $t$.

The lifting coefficient is thus a minimum over the relaxation patterns $R$, and it must include the patterns outside $\mathcal{R}_\Omega$: the binding pattern often has $\xi(R) > 0$ (Appendix~\ref{ec:lifting} gives an example). The monotonicity of $\xi$ and the bound $\alpha_R \geq \min\{1, \xi(R)/\overline{c}_t\}$ prune the enumeration: once an incumbent $\alpha \leq 1$ is known, every $R$ with $\xi(R) \geq \alpha\,\overline{c}_t$ and all its supersets can be skipped. The next lemma connects the coefficient to condition (C3).

\begin{lemma}\label{lem:C3_equiv}
Let $D \subseteq \Omega$ and $t \in [T] \setminus \Omega$. The following are equivalent: (a) (C3) holds at $t$, i.e., $s^R_{t-1} < \kappa^R_t$ for some $R \in \mathcal{R}_\Omega$; (b) some point of $F_\Omega$ has $c_t > 0$; (c) there exists $R \in \mathcal{R}_\Omega$ with $f(B_R \cup \{t\}) > f(B_R)$; (d) $\alpha^*_t = 0$.
\end{lemma}

Condition (a) says that some attainment set reaches period $t$ below the threshold that follows it, so that the energy charged in $t$ is harmless; (b) says that the face is not contained in $\{c_t = 0\}$; (c) says that adding $t$ to the charging periods of such a set increases its capacity; and (d) says that lifting $c_t$ brings nothing. The proof is given in Appendix~\ref{ec:lifting}.

Hence (C3) can be checked with greedy passes alone, without linear programming, and the lifted cut~\eqref{eq:lifted_cut} with $\alpha = \alpha^*_t$ strictly dominates the original one exactly when (C3) fails at $t$. Several periods can be lifted sequentially, with a coefficient that depends on the order (Appendix~\ref{ec:lifting}).

\paragraph{Example.} In the instance of Figure~\ref{fig:two_passes}, the cut $c_1 + c_3 + 0.5\,(1 - u_3) \leq 1$ for $\Omega = \{1,3\}$ satisfies (C1), (C2), and (C4), but (C3) fails at $t = 2$: on the face, the storage must discharge in period~2, end it empty, or start it full, so $c_2 = 0$. With $D = \emptyset$, the four patterns $R \subseteq \{1,3\}$ give $\alpha_\emptyset = 2$, $\alpha_{\{1\}} = \alpha_{\{3\}} = 1$, and $\alpha_{\{1,3\}} = 1/2$: for $R = \{1,3\}$, $B_R = \emptyset$, $\xi(R) = 1 + 0 - 0.5 - 0 = 0.5$, and $\gamma^R_2 = f(\{2\}) = 1$ (discharge in~1, then fill the storage in~2), so $\alpha_R = 0.5/1$. Hence $\alpha^*_2 = 1/2$, and the lifted cut $c_1 + \tfrac{1}{2}\,c_2 + c_3 + 0.5\,(1 - u_3) \leq 1$ is valid; it is tight, for instance, at the point that discharges $0.5$ in period~1 and charges $1$ in period~2.

\paragraph{The strengthening stage.} The strengthening stage (Algorithm~\ref{alg:strengthening} in Appendix~\ref{ec:lifting}) is invoked when Algorithm~\ref{alg:separation} returns no cut and $\mathcal{J}(\boldsymbol{y}^*) \neq \emptyset$. It considers the cuts of the current relaxation that are tight at $\boldsymbol{y}^*$ and, for each of them, the periods $t \notin \Omega$ in which $\boldsymbol{y}^*$ charges (for a discharging cut, discharges): if $\alpha^*_t > 0$, the lifted cut is violated by $\alpha^*_t c^*_t > 0$ and is added. A lifted cut dominates its base cut, which can then be dropped from the relaxation. The lifted cuts are no longer members of the family of Proposition~\ref{prop:cut_strengthening}, but they are valid by Proposition~\ref{prop:lifting}, so the correctness of the loop is unaffected.

\section{Case Study}\label{sec:case_study}
In this section, we evaluate the cuts~\eqref{eq:VIc_D}--\eqref{eq:VId_D}, separated with the procedure of Section~\ref{sec:separation_algorithm}, on the problem of scheduling a battery to maximize its profit through energy arbitrage. Under continuous relaxations, this problem frequently yields solutions with simultaneous charging and discharging when electricity prices are negative~\citep{prat2025sufficient}, a situation that is increasingly common in many electricity markets due to the growing penetration of renewable energy sources. The experiments address three questions: how much the cuts tighten the linear relaxation of the battery, how many rounds of separation are needed to do so, and how their computational cost compares with that of the mixed-integer formulation, both for daily and for longer horizons. All instances use hourly periods ($\Delta = 1$~h) and the database of 100 battery configurations built by~\citet{pozo2023convex}. All the data and code necessary to reproduce the results are available from the GitHub repository~\citep{mygithub}.

\paragraph{Problem and data.} The problem of scheduling the charging and discharging of a battery on the day before delivery, so as to maximize its revenues given the electricity prices of the delivery day, reads
\begin{equation}\label{eq:Scheduling}
 \max_{(\boldsymbol{p}^d, \boldsymbol{p}^c, \boldsymbol{s}) \in \mathcal{S}}  \quad \sum_{t=1}^T \lambda_{t}\,(p_t^d-p_t^c),
\end{equation}
where $\mathcal{S}$ is the operational set~\eqref{eq:Sset}, written with the physical powers $p^d_t$ and $p^c_t$ in place of $d_t$ and $c_t$, and $\lambda_t$ is the electricity price of hour $t$, assumed to be known. For simplicity, we impose no terminal condition on the state of charge. Despite its simplicity, the stylized formulation~\eqref{eq:Scheduling} suffices for the purpose of the experiments, which is to assess how closely each linear relaxation of $\mathcal{S}^{0\text{-}1}$ describes the operation of the battery.

We selected 10 price vectors from the Danish day-ahead electricity market (zone DK1), each corresponding to a day with negative prices. In the order in which we number them, these days are June 2, 8, 9, 15, and 16, 2024; January 1, 2024; June 28, 2024; July 4, 2024; July 2, 2023; and July 7, 2024. The data are publicly available from the ENTSO-E Transparency Platform \citep{entsoe}. Combined with the 100 battery configurations of~\citet{pozo2023convex}, they yield 1000 instances of problem~\eqref{eq:Scheduling} with $T = 24$. To study the effect of the horizon length, we also built 1000 instances with $T = 72$ by concatenating, for $j = 1, \ldots, 10$, the price vectors $j$, $j+1$, and $j+2$ (indices modulo 10). These ten 72-hour price sequences are not chronological, as the selected days are not consecutive; they serve only to assess the effect of a longer horizon. They are also distinct but not independent, as each daily vector appears in three of them.

\paragraph{Formulations.} Each instance was solved with the following formulations, all of which are written in physical units as in Appendix~\ref{app:physical_units}.
\begin{description}
    \item[MILP:] the mixed-integer formulation~\eqref{eq:Sset_MIP}, solved with the default settings of the solver.
    \item[MILP$^{-}$:] the same formulation, solved with presolve, cutting planes, and primal heuristics disabled, so that the solver relies on branch-and-bound over its linear relaxation~\eqref{eq:Prelax}.
    \item[HCH-LP:] the linear relaxation proposed by~\citet{pozo2023convex}.
    \item[HCH-LP$^{\rm S}$:] HCH-LP with the constant power limits replaced by the singleton cuts~\eqref{eq:singleton_cuts}, which bound the charging and discharging power of each period by the limits $\overline{c}_t$ and $\overline{d}_t$.
    \item[RB-LP:] HCH-LP$^{\rm S}$ with the state-of-charge bounds of HCH-LP replaced by the reachable bounds~\eqref{eq:U_star}--\eqref{eq:L_star} of Proposition~\ref{prop:reachable_soc_bounds}. This is the relaxation from which the cutting-plane procedure starts.
    \item[CP:] RB-LP strengthened by the cutting-plane procedure of Section~\ref{sec:separation_algorithm}: at each round, Algorithm~\ref{alg:separation} is applied to the current LP solution, all the violated cuts~\eqref{eq:VIc_D} and~\eqref{eq:VId_D} found are added, and the LP is re-optimized; the strengthening stage of Algorithm~\ref{alg:strengthening} is invoked if no cut is found while some period still charges and discharges simultaneously. The loop stops when the LP solution has no such period.
    \item[CP$_1$:] as CP, but adding, at each round, only the most violated cut that contains each period with simultaneous charging and discharging, for each family (charging and discharging cuts).
\end{description}
The stopping rule of CP has a useful consequence. If the LP solution has no period with simultaneous charging and discharging, setting $u_t = 1$ when $p^c_t > 0$ and $u_t = 0$ otherwise yields a feasible point of~\eqref{eq:Sset_MIP} with the same objective value, which is therefore optimal for the MILP: CP then returns an optimal schedule together with a certificate of optimality, without any binary variable. Within the algorithm, a period charges and discharges simultaneously when both $c_t = \Delta \eta_c p^c_t$ and $d_t = \Delta p^d_t / \eta_d$ exceed $10^{-6}$~kWh, and a cut is added when its violation exceeds $10^{-6}$~kWh. To report the results, we follow~\citet{pozo2023convex} and count the hours with $p^c_t p^d_t > 10^{-4}$~kW$^2$; both criteria identify exactly the same hours in all the solutions reported below.

\paragraph{Implementation and computational setting.} All formulations were implemented in Python with the \emph{gurobipy} interface and solved with Gurobi 10.0.3 \citep{gurobi} on a Linux-based server with a 2.6~GHz CPU, using a single thread and 8~GB of RAM. The LPs were solved with presolve disabled, and each round of CP re-optimizes the LP with the dual simplex method from the previous basis. The separation and strengthening routines are written in Python, with NumPy for the closed-form evaluation of the windows, whereas the LPs and MILPs are solved by compiled code, so the separation times reported below are conservative. For the MILPs and the LPs without cuts, we report the solution time measured by the solver. For CP and CP$_1$, we report the sum of the solution times of all the LPs solved and the CPU time of the Python code executed in the loop (reading the LP solution, separating, selecting, and adding the cuts); the time to build the models is excluded in all cases. Repeated runs differed by 2--3\% in the aggregate times. With $T = 72$, MILP$^{-}$ was stopped after 30 seconds. In all cases, the optimality gap of a relaxation is $(z - z^*)/|z^*|$, where $z$ is its objective value and $z^*$ is the best objective value found by the two MILPs.

\paragraph{Results for $T = 24$.} Table~\ref{tab:results_scheduling} compares the formulations, and Table~\ref{tab:results_cp} provides further statistics on CP.
\begin{table}
\renewcommand{\arraystretch}{1.3}
\centering
\small
\begin{tabular}{lcccc}
\hline
 & Hours with $p^c_t p^d_t > 10^{-4}$ (\%) & $\boldsymbol{p}^c \cdot \boldsymbol{p}^d$ (kW$^2$) & Gap (\%) & Time (ms) \\
\hline
MILP & 0 & 0 & -- & 1.84 \\
MILP$^{-}$ & 0 & 0 & -- & 11.55 \\
HCH-LP & 2.73 & 25.27 & 0.044 & 0.28 \\
HCH-LP$^{\rm S}$ & 2.73 & 25.13 & 0.043 & 0.29 \\
RB-LP & 2.73 & 24.80 & 0.043 & 0.29 \\
CP$_1$ & 0 & 0 & 0.000 & 2.41 \\
CP & 0 & 0 & 0.000 & 1.92 \\
\hline
\end{tabular}
\vspace{2mm}
\caption{Battery scheduling problem with $T = 24$ (1000 instances, 24\,000 hours): percentage of hours with simultaneous charging and discharging, average value of $\boldsymbol{p}^c \cdot \boldsymbol{p}^d$ per instance, average optimality gap, and average solution time per instance.}
\label{tab:results_scheduling}
\end{table}
\begin{table}
\renewcommand{\arraystretch}{1.3}
\centering
\small
\begin{tabular}{lcc}
\hline
 & $T = 24$ & $T = 72$ \\
\hline
Affected instances (HCH-LP) & 428 & 772 \\
Hours with simultaneous charging and discharging, HCH-LP (\%) & 2.73 & 2.63 \\
Optimality gap of HCH-LP: mean / max (\%) & 0.044 / 1.96 & 0.043 / 1.22 \\
Instances in which CP returns an optimal schedule & 1000 & 1000 \\
Affected instances solved by CP in one round (\%) & 94.6 & 91.1 \\
Rounds of CP per affected instance: mean / max & 1.06 / 4 & 1.10 / 4 \\
Cuts added per affected instance: mean / max & 8.6 / 31 & 14.6 / 111 \\
Separation time per round (ms) & 3.0 & 5.8 \\
\hline
\multicolumn{3}{l}{Time per instance (ms): mean / median / max} \\
\quad MILP & 1.84 / 1.84 / 7.0 & 5.42 / 4.81 / 15.9 \\
\quad MILP$^{-}$ & 11.55 / 1.13 / 396 & 2790$^{\dagger}$ / 105 / 30\,000$^{\dagger}$ \\
\quad HCH-LP & 0.28 / 0.28 / 0.4 & 1.11 / 1.12 / 1.4 \\
\quad CP & 1.92 / 0.39 / 30.4 & 6.90 / 5.74 / 53.4 \\
\hline
\end{tabular}
\vspace{2mm}
\caption{Statistics of the cutting-plane procedure CP and solution times for $T = 24$ and $T = 72$ (1000 instances each). Affected instances are those in which HCH-LP charges and discharges simultaneously. $^{\dagger}$With $T = 72$, MILP$^{-}$ reached the time limit of 30~s in 39 instances, so its mean and maximum are lower bounds.}
\label{tab:results_cp}
\end{table}

The main result is that CP eliminates simultaneous charging and discharging in all 1000 instances. HCH-LP charges and discharges simultaneously in 654 hours (2.73\%), spread over 428 instances, with an average optimality gap of 0.044\%, which rises to 0.10\% over the affected instances and reaches 1.96\% in the worst case. CP stopped in every instance because its LP solution had no period with simultaneous charging and discharging, so, by the argument above, it returned an optimal schedule of the MILP in all of them. Its objective value was never smaller than $z^*$, and in a few instances it was slightly larger (by at most $2.5 \cdot 10^{-5}$ in relative terms), because the solver stops the MILP once its relative gap falls below the default tolerance of $10^{-4}$. The cutting-plane loop is also short: 95\% of the affected instances were solved in a single round, with an average of 8.6 cuts per affected instance, and the strengthening stage of Algorithm~\ref{alg:strengthening} was never needed, as Algorithm~\ref{alg:separation} found violated cuts at every round.

The comparison of HCH-LP, HCH-LP$^{\rm S}$, and RB-LP shows that this improvement is entirely due to the cuts separated by Algorithm~\ref{alg:separation}. The singleton cuts and the reachable bounds~\eqref{eq:U_star}--\eqref{eq:L_star} slightly reduce the magnitude of simultaneous charging and discharging, but they do not remove a single hour. This is explained by Proposition~\ref{prop:reachable_soc_bounds}: the reachable bounds coincide with the bounds of HCH-LP as soon as the battery can be filled or emptied from its initial state, which happens within the first few hours, whereas simultaneous charging and discharging typically occurs later in the day, during the hours of negative prices.

Adding all the violated cuts at each round (CP) is preferable to adding only the most violated one per period (CP$_1$): both have the same median time, but CP$_1$ needs more rounds in 131 of the 428 affected instances (642 rounds in total, against 454), which increases its average time by 25\%, and the larger LPs of CP do not offset this advantage. Ordering the cuts by violation or by efficacy (violation divided by the Euclidean norm of the coefficients) yields the same rounds, the same number of cuts, and the same optimal values when all the violated cuts are added.

Regarding solution times, CP is about seven times slower than HCH-LP and slightly slower than the MILP solved with default settings (1.92 ms against 1.84 ms per instance), which is solved at the root node in all 1000 instances: for a single battery, the presolve and the generic cutting planes of a modern solver close the integrality gap. Without them, MILP$^{-}$ requires 11.55 ms per instance on average, with a maximum of 396 ms, dominated by branching, even for the ninth price vector (July 2, 2023), for which HCH-LP has no simultaneous charging and discharging and MILP$^{-}$ needs 84.6 ms per instance. Most of the time of CP is spent in the separation (3.0 ms per round, 71\% of the total, in Python), and the separation time per instance grows with the number of hours in which HCH-LP charges and discharges simultaneously, from 2.7 ms for instances with one such hour to 5.9 ms for instances with four.

\paragraph{Results for $T = 72$.} The right-hand column of Table~\ref{tab:results_cp} shows how these conclusions extend to a longer horizon. HCH-LP is affected more often: 772 of the 1000 instances (77\%) have at least one hour with simultaneous charging and discharging, with 2.45 such hours on average per affected instance and an average value of $\boldsymbol{p}^c \cdot \boldsymbol{p}^d$ of 73.9~kW$^2$ per instance. CP again eliminates simultaneous charging and discharging and returns an optimal schedule in all 1000 instances, and 91\% of the affected instances are solved in a single round.

The solution times scale very differently. When the horizon is tripled, the average time of the MILP solved with default settings grows by a factor of 2.9 (to 5.42 ms), with all instances still solved at the root node, and that of CP by a factor of 3.6 (to 6.90 ms), mainly because each round evaluates larger sets (5.8 ms per round). MILP$^{-}$, in contrast, requires at least 2.79 seconds per instance on average, and 39 instances reach the time limit of 30 seconds with a residual gap of up to 0.94\%, including instances that are easy for the relaxations: for the sequence formed by the seventh, eighth, and ninth vectors, CP stops after the first LP (1.0 ms per instance), whereas MILP$^{-}$ needs 3.6 seconds.

\paragraph{Discussion.} The experiments show that the cuts~\eqref{eq:VIc_D}--\eqref{eq:VId_D}, separated with Algorithm~\ref{alg:separation}, provide a linear description of the battery that is exact in practice for this problem: in all 2000 instances, the strengthened relaxation yields an optimal schedule without binary variables, typically after a single round of separation, whereas the best linear relaxation available in the literature charges and discharges simultaneously in 43\% to 77\% of the instances. This is a test of the quality of the description, not a proposal of a solution method: for a single battery with a linear objective, the MILP solved with default settings, or the shortest-path algorithm of~\citet{bakhshi2023polynomial}, remains the fastest way to obtain an optimal schedule. The cuts are meant for the models in which neither applies. They are linear and involve only the variables of one battery, so they can be added to any model in which the battery is one of many components, such as unit-commitment, network-constrained, or stochastic problems, and to models that must remain linear, such as bilevel, equilibrium, and pricing models. There, a tighter description of each unit strengthens the relaxation. The extended formulation of~\citet{bakhshi2023polynomial} describes the convex hull of each unit exactly, but for the batteries of this case study and $T = 72$ it has several hundred thousand variables per unit, whereas the cuts are separated as needed in the original space and remain valid under period-dependent limits and self-discharge (Remark~\ref{rem:extensions}).

\section{Conclusions} \label{sec:conclusions}
This paper has studied the polyhedral structure of the multi-period operational set of an energy storage unit with conversion losses. We have characterized the extreme points of its convex hull and shown that the maximum energy that can be charged, or discharged, during a subset of periods is a submodular function of that subset. This structure yields a family of linear valid inequalities, strengthened with the continuous power variables, for which we have established necessary and sufficient facet-defining conditions. The family contains inequalities that dominate the storage capacity inequalities of the literature and, through its members with gaps between periods, captures interactions between periods that no inequality on the SoC alone can express. A separation algorithm that evaluates each candidate cut in linear time, can be made exact over contiguous sets of periods, and is complemented by a closed-form lifting of the cuts that fail a facet condition, makes the family usable within a cutting-plane procedure.

In the arbitrage experiments, the linear relaxation strengthened with the separated cuts eliminated simultaneous charging and discharging, and thus returned an optimal schedule, in all instances, for both daily and three-day horizons, and mostly after a single round of separation. The strongest relaxation of the literature, by contrast, charged and discharged simultaneously in 43\% to 77\% of the instances, and neither the single-period cuts nor the reachable SoC bounds removed any of these occurrences on their own. For a single battery, a mixed-integer solver with default settings, or the shortest-path algorithm of~\citet{bakhshi2023polynomial}, remains faster; the cuts provide a tight linear description of the storage in the original variables at a cost of milliseconds per instance, even with a separation routine implemented in Python.

Future research will address the facets of $\operatorname{conv}(\mathcal{P}^{0\text{-}1})$ that lie outside the proposed family, which may leave residual simultaneous charging and discharging in other instances; the extension of the facet characterization and of the separation procedure to period-dependent limits and self-discharge, under which the cuts remain valid, and to terminal SoC conditions and units that provide reserves; and the use of the cuts within branch-and-cut, in models with several storage units, and in models that require a linear representation of the storage, such as bilevel, equilibrium, and pricing models.

\section{Code and Data Disclosure}\label{sec:Code and Data Disclosure}The code and data to support the numerical experiments in this paper can be found at the GitHub repository~\citep{mygithub}.

\ACKNOWLEDGMENT{This work was supported by the Consejería de Universidad, Investigación e Innovación de la Junta de Andalucía through project DGP-PIDI-2024-00851. Juan M. Morales is with the research group OASYS, University of Malaga, Malaga 29071, Spain: juan.morales@uma.es. The author thankfully acknowledges the computer resources, technical expertise, and assistance provided by the SCBI (Supercomputing and Bioinformatics) center of the University of Málaga.

}

\bibliographystyle{informs2014} %
\bibliography{references} %

\clearpage
\begin{APPENDICES}

\section{Supplementary Material for Section~\ref{sec:extreme_points}}\label{ec:proof_binary_faces}

This section proves Lemma~\ref{lem:binary_faces}, which describes the faces of $\operatorname{conv}(\mathcal{P}^{0\text{-}1})$ obtained by fixing the operating modes $\boldsymbol{u} = \bar{\boldsymbol{u}}$.

\begin{proof}{Proof of Lemma~\ref{lem:binary_faces}}
The inequalities $u_t \leq 1$ for $t$ with $\bar{u}_t = 1$, and $u_t \geq 0$ for $t$ with $\bar{u}_t = 0$, are valid for $\operatorname{conv}(\mathcal{P}^{0\text{-}1})$, and $\{\boldsymbol{u} = \bar{\boldsymbol{u}}\}$ is the set where all of them are tight. Hence, $F := \operatorname{conv}(\mathcal{P}^{0\text{-}1}) \cap \{\boldsymbol{u} = \bar{\boldsymbol{u}}\}$ is a face. Clearly, $\mathcal{P}_{\bar{\boldsymbol{u}}} \subseteq F$. Conversely, let $\boldsymbol{y} \in F$ and write $\boldsymbol{y} = \sum_{i} \lambda_i \boldsymbol{y}^i$ with $\boldsymbol{y}^i \in \mathcal{P}^{0\text{-}1}$, $\lambda_i > 0$, and $\sum_i \lambda_i = 1$. If $\bar{u}_t = 1$, then $\sum_i \lambda_i u^i_t = 1$ and $u^i_t \leq 1$ force $u^i_t = 1$ for all $i$; the case $\bar{u}_t = 0$ is analogous. Thus, every $\boldsymbol{y}^i$ lies in the convex set $\mathcal{P}_{\bar{\boldsymbol{u}}}$, and so does $\boldsymbol{y}$.

Since $\mathcal{P}^{0\text{-}1}$ is the union of finitely many polytopes, its convex hull is a polytope. If $\boldsymbol{y} \in \operatorname{ext}(\operatorname{conv}(\mathcal{P}^{0\text{-}1}))$, then $\boldsymbol{y} \in \mathcal{P}^{0\text{-}1}$, because the extreme points of the convex hull of a set belong to the set. Hence, $\boldsymbol{y} \in \mathcal{P}_{\bar{\boldsymbol{u}}}$ for some $\bar{\boldsymbol{u}}$, and $\boldsymbol{y}$ is extreme in $\mathcal{P}_{\bar{\boldsymbol{u}}} \subseteq \operatorname{conv}(\mathcal{P}^{0\text{-}1})$. Conversely, the extreme points of a face of a polytope are extreme points of the polytope.
\hfill$\square$
\end{proof}

\section{Supplementary Material for Section~\ref{sec:submodular_structure}}\label{ec:sec_submodular}

This section proves Lemma~\ref{lem:greedy_polymatroid} and Corollary~\ref{cor:submodular_disch}. Recall that, for a window $[l,k]$ and an initial SoC $\sigma$, $\mathcal{X}(\sigma)$ is the set~\eqref{eq:Xset} of net-charging sequences, $f(\Omega; \sigma)$ and $h(\Omega; \sigma)$ are the maximum energy that can be charged and discharged in $\Omega$, and the greedy rule for $\Omega$ charges as much as possible in $\Omega$ and discharges as much as possible outside $\Omega$.

\subsection{Proof of Lemma~\ref{lem:greedy_polymatroid}}\label{ec:proof_greedy}

\begin{proof}{Proof of Lemma~\ref{lem:greedy_polymatroid}}
Let $\mathcal{L}$ be the class of functions $\beta : [\underline{s}, \overline{s}] \to \mathbb{R}$ satisfying $0 \leq \beta(\sigma) - \beta(\sigma') \leq \sigma' - \sigma$ for all $\sigma \leq \sigma'$. Equivalently, $\beta$ is nonincreasing and $\sigma \mapsto \sigma + \beta(\sigma)$ is nondecreasing. Let also $\phi^+(\sigma) := \min\{\sigma + \overline{c}, \overline{s}\}$ and $\phi^-(\sigma) := \max\{\sigma - \overline{d}, \underline{s}\}$ be the SoC reached at the end of a period that starts at SoC $\sigma$ and charges, respectively discharges, as much as possible. Both maps are nondecreasing and $1$-Lipschitz, and both return an admissible SoC, i.e., $\phi^\pm(\sigma) \in [\underline{s}, \overline{s}]$ for every $\sigma \in [\underline{s}, \overline{s}]$; this guarantees that the compositions below are well defined. We use the following elementary fact:
\begin{equation}\label{eq:class_L}
\text{if } \beta \in \mathcal{L} \text{ and } \phi \text{ is nondecreasing and $1$-Lipschitz, then } \beta \circ \phi \in \mathcal{L}.
\end{equation}
Indeed, for $\sigma \leq \sigma'$, we have $\phi(\sigma) \leq \phi(\sigma')$, so $0 \leq \beta(\phi(\sigma)) - \beta(\phi(\sigma')) \leq \phi(\sigma') - \phi(\sigma) \leq \sigma' - \sigma$.

We prove (i) and (ii) by induction on the window length. For an empty window, $b \equiv 0 \in \mathcal{L}$ and there is nothing to prove. Consider $[l,k]$ with $l \leq k$, and assume the claims hold for $[l+1,k]$. Let $\Omega' := \Omega \setminus \{l\}$ and $\beta' := b_{[l+1,k]}(\Omega'; \cdot\,) \in \mathcal{L}$. A vector $\boldsymbol{x}$ belongs to $\mathcal{X}_{[l,k]}(\sigma)$ if and only if $\sigma_1 := \sigma + x_l \in [\phi^-(\sigma), \phi^+(\sigma)]$ and $(x_{l+1}, \ldots, x_k) \in \mathcal{X}_{[l+1,k]}(\sigma_1)$. Optimizing first over the remaining periods,
\[
b(\Omega; \sigma) = \max_{\sigma_1 \in [\phi^-(\sigma),\, \phi^+(\sigma)]} \Big\{ \chi\,(\sigma_1 - \sigma) + \beta'(\sigma_1) \Big\}, \qquad \chi := \begin{cases} 1 & \text{if } l \in \Omega,\\ 0 & \text{otherwise.}\end{cases}
\]

\emph{Case $l \notin \Omega$.} The objective $\beta'(\sigma_1)$ is nonincreasing, so the maximum is attained at $\sigma_1 = \phi^-(\sigma)$, which is the greedy choice $x_l^{(\Omega)} = \max\{-\overline{d}, \underline{s} - \sigma\}$. Hence, $b(\Omega; \cdot\,) = \beta' \circ \phi^- \in \mathcal{L}$ by~\eqref{eq:class_L}.

\emph{Case $l \in \Omega$.} The objective is $\psi(\sigma_1) - \sigma$, where $\psi(\varsigma) := \varsigma + \beta'(\varsigma)$. Since $\beta' \in \mathcal{L}$, the function $\psi$ is nondecreasing and $1$-Lipschitz, so the maximum is attained at $\sigma_1 = \phi^+(\sigma)$, which is the greedy choice $x_l^{(\Omega)} = \min\{\overline{c}, \overline{s} - \sigma\}$. Hence, $b(\Omega; \sigma) = \psi(\phi^+(\sigma)) - \sigma$. For $\sigma \leq \sigma'$, the composition $\psi \circ \phi^+$ is nondecreasing and $1$-Lipschitz, which gives
\[
b(\Omega; \sigma) - b(\Omega; \sigma') = (\sigma' - \sigma) - \big[\psi(\phi^+(\sigma')) - \psi(\phi^+(\sigma))\big] \in [0, \sigma' - \sigma],
\]
so $b(\Omega; \cdot\,) \in \mathcal{L}$.

In both cases, the optimal first decision is the greedy one. By the induction hypothesis, the greedy rule started at $s^{(\Omega)}_l = \sigma_1$ is optimal on $[l+1,k]$. Therefore, $\boldsymbol{x}^{(\Omega)}$ is optimal on $[l,k]$. The rule performs one pass over the window, hence the $\mathcal{O}(k-l+1)$ complexity.
\hfill$\square$
\end{proof}

\subsection{The Mirrored Storage Unit and Proof of Corollary~\ref{cor:submodular_disch}}\label{app:reflection}

This section makes precise the symmetry between the charging and discharging capacity functions used in Section~\ref{sec:submodular_structure}. Reducing discharging to charging requires exchanging the power limits, so we temporarily make them explicit and write $\mathcal{P}_{[l,k]}(\sigma \mid \overline{c}, \overline{d})$ and $f(\Omega; \sigma \mid \overline{c}, \overline{d})$ for the sub-horizon set~\eqref{eq:Pset_window} and the charging capacity function~\eqref{eq:LVIopt} of a storage unit with maximum charging power $\overline{c}$ and maximum discharging power $\overline{d}$. Nothing in Lemma~\ref{lem:greedy_polymatroid} or Proposition~\ref{prop:submodularity} depends on the particular values of $\overline{c}, \overline{d} > 0$, so both results hold for any such pair.

The \emph{mirrored storage unit} has the same window and SoC limits $[\underline{s}, \overline{s}]$ as the original one, but its power limits are exchanged: its maximum charging power is $\overline{d}$ and its maximum discharging power is $\overline{c}$. Its operation is obtained from that of the original unit by reading every charge as a discharge and vice versa, and by measuring the stored energy from the opposite end of the SoC interval. Table~\ref{tab:reflection} summarizes the correspondence.
\begin{table}[htbp]
\centering
\caption{Correspondence between the original and the mirrored storage unit on the window $[l,k]$.}
\label{tab:reflection}
\begin{tabular}{lcc}
\hline
 & Original unit & Mirrored unit \\
\hline
Maximum charging / discharging power & $\overline{c}$ / $\overline{d}$ & $\overline{d}$ / $\overline{c}$ \\
SoC limits & $[\underline{s}, \overline{s}]$ & $[\underline{s}, \overline{s}]$ \\
Charging / discharging power in period $t$ & $c_t$ / $d_t$ & $c'_t := d_t$ / $d'_t := c_t$ \\
Operating mode in period $t$ & $u_t$ & $u'_t := 1 - u_t$ \\
Net charging in period $t$ & $x_t$ & $x'_t = -x_t$ \\
SoC at the end of period $t$ & $s_t$ & $s'_t := \underline{s} + \overline{s} - s_t$ \\
Initial SoC & $\sigma$ & $\sigma' := \underline{s} + \overline{s} - \sigma$ \\
\hline
\end{tabular}
\end{table}

\begin{lemma}\label{lem:reflection}
Let $\sigma \in [\underline{s}, \overline{s}]$ and $\sigma' := \underline{s} + \overline{s} - \sigma$. The reflection
\[
\Theta(\boldsymbol{d}, \boldsymbol{c}, \boldsymbol{u}) := (\boldsymbol{c}, \boldsymbol{d}, \boldsymbol{1} - \boldsymbol{u})
\]
is a bijection from $\mathcal{P}_{[l,k]}(\sigma \mid \overline{c}, \overline{d})$ onto $\mathcal{P}_{[l,k]}(\sigma' \mid \overline{d}, \overline{c})$. Moreover, if $\boldsymbol{s}$ is the SoC trajectory of a point and $\boldsymbol{s}'$ that of its image, then $s'_t = \underline{s} + \overline{s} - s_t$ for all $t \in [l-1, k]$.
\end{lemma}

\begin{proof}{Proof}
Let $(\boldsymbol{d}', \boldsymbol{c}', \boldsymbol{u}') := \Theta(\boldsymbol{d}, \boldsymbol{c}, \boldsymbol{u})$, that is, $\boldsymbol{d}' = \boldsymbol{c}$, $\boldsymbol{c}' = \boldsymbol{d}$, and $\boldsymbol{u}' = \boldsymbol{1} - \boldsymbol{u}$. We check the constraints of $\mathcal{P}_{[l,k]}(\sigma' \mid \overline{d}, \overline{c})$ one by one. Nonnegativity and binariness are clearly preserved. The discharging limit of the mirrored unit is $\overline{c}$, and indeed $\boldsymbol{d}' = \boldsymbol{c} \leq \overline{c}\,\boldsymbol{u} = \overline{c}\,(\boldsymbol{1} - \boldsymbol{u}')$. Its charging limit is $\overline{d}$, and indeed $\boldsymbol{c}' = \boldsymbol{d} \leq \overline{d}\,(\boldsymbol{1} - \boldsymbol{u}) = \overline{d}\,\boldsymbol{u}'$. For the SoC, we have $\boldsymbol{x}' = \boldsymbol{c}' - \boldsymbol{d}' = -\boldsymbol{x}$, so for every $t \in [l,k]$,
\[
s'_t = \sigma' + x'([l,t]) = (\underline{s} + \overline{s} - \sigma) - x([l,t]) = \underline{s} + \overline{s} - s_t,
\]
and $s'_{l-1} = \sigma' = \underline{s} + \overline{s} - s_{l-1}$. Hence, $\underline{s} \leq s'_t \leq \overline{s}$ if and only if $\underline{s} \leq s_t \leq \overline{s}$. This shows that $\Theta$ maps $\mathcal{P}_{[l,k]}(\sigma \mid \overline{c}, \overline{d})$ into $\mathcal{P}_{[l,k]}(\sigma' \mid \overline{d}, \overline{c})$. Applying the same argument to the mirrored unit, whose mirror is the original unit, shows that $\Theta$ maps $\mathcal{P}_{[l,k]}(\sigma' \mid \overline{d}, \overline{c})$ back into $\mathcal{P}_{[l,k]}(\sigma \mid \overline{c}, \overline{d})$. Since $\Theta \circ \Theta$ is the identity, $\Theta$ is a bijection.
\hfill$\square$
\end{proof}

Lemma~\ref{lem:reflection} turns maximum discharging of the original unit into maximum charging of the mirrored unit. The charging capacity function of the mirrored unit, denoted by $f^{\leftrightarrow}$ in Section~\ref{sec:submodular_structure}, is therefore
\begin{equation}\label{eq:f_mirrored}
\begin{aligned}
f^{\leftrightarrow}(\Omega; \sigma') &:= f(\Omega; \sigma' \mid \overline{d}, \overline{c}) \\
&\phantom{:}= \max\big\{ c'(\Omega) : (\boldsymbol{d}', \boldsymbol{c}', \boldsymbol{u}') \in \mathcal{P}_{[l,k]}(\sigma' \mid \overline{d}, \overline{c}),\ \boldsymbol{u}' = \boldsymbol{1}_\Omega \big\}.
\end{aligned}
\end{equation}

\paragraph{Proof of Corollary~\ref{cor:submodular_disch}.}

Let $\sigma' := \underline{s} + \overline{s} - \sigma$. By Lemma~\ref{lem:reflection}, $\Theta$ maps the feasible set of~\eqref{eq:LVIopt_dis} bijectively onto that of~\eqref{eq:f_mirrored}. It maps the mode constraint $\boldsymbol{u} = \boldsymbol{1}_{[l,k] \setminus \Omega}$ onto $\boldsymbol{u}' = \boldsymbol{1}_\Omega$, and it preserves the objective value, since $c'(\Omega) = d(\Omega)$. This proves~\eqref{eq:h_mirror}. The mirrored unit satisfies all the assumptions of Section~\ref{sec:Formulation}: $\sigma' \in [\underline{s}, \overline{s}]$, and both power limits are at most $\overline{s} - \underline{s}$. We can therefore apply Lemma~\ref{lem:greedy_polymatroid} and Proposition~\ref{prop:submodularity} to $f^{\leftrightarrow}$.

(i) By Lemma~\ref{lem:greedy_polymatroid}, the mirrored unit attains $f^{\leftrightarrow}(\Omega; \sigma')$ by charging as much as possible in $\Omega$, i.e., $c'_t = \min\{\overline{d}, \overline{s} - s'_{t-1}\}$, and discharging as much as possible outside $\Omega$, i.e., $d'_t = \min\{\overline{c}, s'_{t-1} - \underline{s}\}$. By Lemma~\ref{lem:reflection}, $d_t = c'_t$, $c_t = d'_t$, and $s'_{t-1} = \underline{s} + \overline{s} - s_{t-1}$. Hence, $\overline{s} - s'_{t-1} = s_{t-1} - \underline{s}$ and $s'_{t-1} - \underline{s} = \overline{s} - s_{t-1}$, which gives exactly the rule in (i).

(ii) $h(\cdot\,; \sigma) = f^{\leftrightarrow}(\cdot\,; \sigma')$ is submodular by Proposition~\ref{prop:submodularity}(ii).

(iii) Let $\sigma \leq \tilde{\sigma}$, so that $\tilde{\sigma}' \leq \sigma'$. By Proposition~\ref{prop:submodularity}(iii) applied to $f^{\leftrightarrow}$, we have $0 \leq f^{\leftrightarrow}(\Omega; \tilde{\sigma}') - f^{\leftrightarrow}(\Omega; \sigma') \leq \sigma' - \tilde{\sigma}' = \tilde{\sigma} - \sigma$, that is, $0 \leq h(\Omega; \tilde{\sigma}) - h(\Omega; \sigma) \leq \tilde{\sigma} - \sigma$.

(iv) The reflection applies equally to the whole horizon. It maps $\boldsymbol{y} \in \mathcal{P}^{0\text{-}1}$ to a point $\boldsymbol{y}' := \Theta(\boldsymbol{y})$ of the set $\mathcal{P}^{0\text{-}1}$ of the mirrored unit with initial SoC $\underline{s} + \overline{s} - s_0$. For $\boldsymbol{y}'$, the charging periods in the window are $U_{\boldsymbol{y}'} = [l,k] \setminus U_{\boldsymbol{y}}$, and the SoC at the end of period $l-1$ is $s'_{l-1}(\boldsymbol{y}') = \underline{s} + \overline{s} - s_{l-1}(\boldsymbol{y})$. Moreover, $\sigma \geq s_{l-1}(\boldsymbol{y})$ is equivalent to $\sigma' \leq s'_{l-1}(\boldsymbol{y}')$. Applying Proposition~\ref{prop:submodularity}(iv) to $\boldsymbol{y}'$ yields
\[
d(\Omega) = c'(\Omega) \leq f^{\leftrightarrow}\big(\Omega \cap U_{\boldsymbol{y}'}; \sigma'\big) = h\big(\Omega \setminus U_{\boldsymbol{y}}; \sigma\big).
\]
\hfill$\square$

\section{Supplementary Material for Section~\ref{sec:cut_strengthening}}\label{ec:sec_strengthening}

This section complements Section~\ref{sec:cut_strengthening}, where the cuts are strengthened with the continuous power variables. Recall that $\pi_\tau := -f(\tau \mid \Omega \setminus \{\tau\})$ and $\rho_\tau := -h(\tau \mid \Omega \setminus \{\tau\})$ are the gains of relaxing $\tau$ in the charging and discharging cuts, and that $\overline{c}_t$ and $\overline{d}_t$ are the limits of the singleton cuts~\eqref{eq:singleton_cuts}.

\subsection{The Role of the Singleton Cuts in Proposition~\ref{prop:term_comparison}}\label{ec:example_term_comparison}

Let $T = 2$, $[\underline{s}, \overline{s}] = [0,1]$, $s_0 = 0.5$, $\overline{c} = \overline{d} = 1$, and $\Omega = \{1,2\}$. By Lemma~\ref{lem:greedy_polymatroid}, $f(\{1,2\}) = f(\{1\}) = 0.5$ and $f(\{2\}) = 1$, so $\pi_1 = 0.5$ and $\pi_2 = 0$. Inequality~\eqref{eq:VIc_D} reads $c_1 + c_2 \leq 0.5 + 0.5\,(1 - u_1)$ for $D = \emptyset$ and $c_1 + c_2 \leq 0.5 + d_1$ for $D = \{1\}$. Over $\mathcal{S}^{\mathcal{R}}$, neither implies the other. The point with $c_1 = d_1 = u_1 = 0$, $c_2 = 0.75$, $d_2 = 0.25$, $u_2 = 0.75$ satisfies the first and violates the second. The point with $u_1 = 0.5$, $c_1 = d_1 = 0.5$, $u_2 = 1$, $c_2 = 0.5$, $d_2 = 0$ violates the first and satisfies the second. However, $\pi_1 = 0.5 = \overline{d}_1$, because only $s_0 - \underline{s} = 0.5$ can be discharged in period~1. The singleton cut $d_1 \leq 0.5\,(1 - u_1)$ removes the second point, and by Proposition~\ref{prop:term_comparison}(b) the inequality for $D = \{1\}$ implies the one for $D = \emptyset$ over $\mathcal{S}^{\mathcal{R}}_1$.

\subsection{Valid Inequalities in Physical Units}\label{app:physical_units}

The results of Section~\ref{sec:LVI} are stated for the normalized model of Section~\ref{sec:Formulation}. This section writes them in terms of the original data. To avoid ambiguity, we denote by $p^c_t$ and $p^d_t$ the physical charging and discharging power in period $t$, by $\overline{p}^c$ and $\overline{p}^d$ their limits, and keep $c_t$ and $d_t$ for the normalized variables. The normalization of Section~\ref{sec:Formulation} reads
\begin{equation}\label{eq:normalization_map}
c_t = \Delta \eta_c\, p^c_t, \qquad d_t = \frac{\Delta}{\eta_d}\, p^d_t, \qquad \overline{c} = \Delta \eta_c\, \overline{p}^c, \qquad \overline{d} = \frac{\Delta}{\eta_d}\, \overline{p}^d,
\end{equation}
while $\boldsymbol{u}$, the SoC, and its limits are unchanged. Thus, $c_t$ is the energy stored and $d_t$ the energy withdrawn from the storage during period~$t$, both measured inside the storage.

\emph{Coefficients.} All coefficients are energies, computed with the normalized limits in~\eqref{eq:normalization_map}. In the greedy rule~\eqref{eq:greedy_rule}, a charging period stores $\min\{\Delta \eta_c\, \overline{p}^c,\ \overline{s} - s_{t-1}\}$ and a discharging period withdraws $\min\{\Delta \overline{p}^d / \eta_d,\ s_{t-1} - \underline{s}\}$. Hence, $f(\Omega)$ is the maximum energy that can be stored during $\Omega$, and $\pi_\tau = f(\Omega \setminus \{\tau\}) - f(\Omega)$ is the stored energy that switching $\tau$ to discharging can add to it. The reachable SoC bounds~\eqref{eq:sigma_bounds} become $\underline{\sigma}_t = \max\{s_0 - t\,\Delta \overline{p}^d / \eta_d,\ \underline{s}\}$ and $\overline{\sigma}_t = \min\{s_0 + t\,\Delta \eta_c\, \overline{p}^c,\ \overline{s}\}$.

\emph{Charging inequalities.} Substituting~\eqref{eq:normalization_map} into~\eqref{eq:VIc_D} gives
\[
\Delta \eta_c \sum_{t \in \Omega} p^c_t \leq f(\Omega) + \frac{\Delta}{\eta_d} \sum_{t \in D} p^d_t + \sum_{\tau \in \Omega \setminus D} \pi_\tau\,(1 - u_\tau).
\]
Both sides are stored energy. Dividing by $\Delta \eta_c$ expresses the inequality in units of charging power:
\begin{equation}\label{eq:VIc_D_physical}
\sum_{t \in \Omega} p^c_t \leq \frac{f(\Omega)}{\Delta \eta_c} + \frac{1}{\eta_c \eta_d} \sum_{t \in D} p^d_t + \sum_{\tau \in \Omega \setminus D} \frac{\pi_\tau}{\Delta \eta_c}\,(1 - u_\tau).
\end{equation}
The factor $1/(\eta_c \eta_d)$ has a direct physical meaning. Discharging at power $p^d_\tau$ during $\Delta$ withdraws $\Delta p^d_\tau / \eta_d$ units of stored energy. Restoring this energy requires a charging power of $(\Delta p^d_\tau / \eta_d)/(\Delta \eta_c) = p^d_\tau / (\eta_c \eta_d)$, which is larger than $p^d_\tau$ because energy is lost both on the way out and on the way back in.

\emph{Discharging inequalities.} Substituting~\eqref{eq:normalization_map} into~\eqref{eq:VId_D} and dividing by $\Delta / \eta_d$ gives, analogously,
\begin{equation}\label{eq:VId_D_physical}
\sum_{t \in \Omega} p^d_t \leq \frac{\eta_d\, h(\Omega)}{\Delta} + \eta_c \eta_d \sum_{t \in D} p^c_t + \sum_{\tau \in \Omega \setminus D} \frac{\eta_d\, \rho_\tau}{\Delta}\, u_\tau.
\end{equation}
Here, charging at power $p^c_\tau$ stores $\Delta \eta_c\, p^c_\tau$, which can later be discharged at a power of only $\eta_c \eta_d\, p^c_\tau$.

\emph{Singleton cuts and convention.} In physical units, the singleton cuts~\eqref{eq:singleton_cuts} read
\[
p^c_t \leq \min\Big\{\overline{p}^c,\ \frac{\overline{s} - \underline{\sigma}_{t-1}}{\Delta \eta_c}\Big\}\, u_t \qquad \text{and} \qquad p^d_t \leq \min\Big\{\overline{p}^d,\ \frac{\eta_d\,(\overline{\sigma}_{t-1} - \underline{s})}{\Delta}\Big\}\,(1 - u_t).
\]
The condition $\pi_\tau = \overline{d}_\tau$ in Remark~\ref{rem:dominating_cuts_convention} is equivalent to
\[
\frac{\pi_\tau}{\Delta \eta_c} = \frac{1}{\eta_c \eta_d}\, \min\Big\{\overline{p}^d,\ \frac{\eta_d\,(\overline{\sigma}_{\tau-1} - \underline{s})}{\Delta}\Big\}.
\]
In words, the coefficient of $1 - u_\tau$ in~\eqref{eq:VIc_D_physical} equals the coefficient $1/(\eta_c \eta_d)$ of $p^d_\tau$ times the largest discharging power allowed in $\tau$. The condition $\rho_\tau = \overline{c}_\tau$ is analogous.

\emph{Lifted cuts.} The lifted term $\alpha\,c_t = \alpha\,\Delta\eta_c\,p^c_t$ becomes $\alpha\,p^c_t$ once~\eqref{eq:lifted_cut} is divided by $\Delta\eta_c$, so a lifted period enters the charging inequality~\eqref{eq:VIc_D_physical} with the coefficient $\alpha$ in front of the physical charging power; in the discharging inequality~\eqref{eq:VId_D_physical}, the coefficient of a lifted $p^d_t$ is likewise $\alpha$.

\subsection{Period-Dependent Limits and Self-Discharge}\label{ec:extensions}

This section proves the claims of Remark~\ref{rem:extensions}. Consider the dynamics $s_t = \eta\, s_{t-1} + c_t - d_t$ with $\eta \in (0,1]$, period-dependent limits $c^{\max}_t, d^{\max}_t > 0$, and SoC limits $[\underline{s}, \overline{s}]$, with $\underline{s} = 0$ if $\eta < 1$. Let $\mathcal{X}$ be the set of net-charging vectors $\boldsymbol{x}$ with $-d^{\max}_t \leq x_t \leq c^{\max}_t$ and $\underline{s} \leq s_t \leq \overline{s}$ for all $t \in [T]$, starting from $s_0$, and let $b(A) := \max\{x(A) : \boldsymbol{x} \in \mathcal{X}\}$ and $p(A) := \min\{x(A) : \boldsymbol{x} \in \mathcal{X}\}$ for $A \subseteq [T]$. The capacity functions $f$ and $h$ are defined by~\eqref{eq:LVIopt}--\eqref{eq:LVIopt_dis} with these dynamics, on the window $[T]$ and from $s_0$.

\begin{lemma}\label{lem:extensions}
\begin{enumerate}[(i)]
    \item For every $A \subseteq [T]$, the greedy rule that sets $x_t := \min\{c^{\max}_t,\ \overline{s} - \eta\, s_{t-1}\}$ for $t \in A$ and $x_t := \max\{-d^{\max}_t,\ \underline{s} - \eta\, s_{t-1}\}$ for $t \notin A$ attains $b(A)$, with $x_t \geq 0$ on $A$ and $x_t \leq 0$ outside $A$. Hence $f = b$, and $f$ is submodular.
    \item Symmetrically, the greedy rule that discharges as much as possible in $A$ and charges as much as possible outside $A$ attains $-p(A)$, with the corresponding signs. Hence $h = -p$, and $h$ is submodular.
\end{enumerate}
\end{lemma}

\begin{proof}{Proof of Lemma~\ref{lem:extensions}}
\emph{Rescaling.} Let $\tilde{x}_t := \eta^{-t} x_t$ and $\tilde{s}_t := \eta^{-t} s_t$. Then $\tilde{s}_t = \tilde{s}_{t-1} + \tilde{x}_t$ with $\tilde{s}_0 = s_0$, and the constraints become $-\eta^{-t}d^{\max}_t \leq \tilde{x}_t \leq \eta^{-t}c^{\max}_t$ and $\eta^{-t}\underline{s} \leq \tilde{s}_t \leq \eta^{-t}\overline{s}$. After translating the SoC by $s_0$, the rescaled set $\tilde{\mathcal{X}}$ is the flexibility set of~\citet[Definition~1]{11369438}, with period-dependent power and energy bounds, so it is a g-polymatroid by their Theorem~2. Let $\tilde{b}$ and $\tilde{p}$ be its paramodular pair, which are submodular and supermodular, respectively, by Lemma~\ref{lem:gpolymatroid}.

\emph{Step 1: we show that the greedy rule attains $\tilde{b}(A \cap [1,j])$ for every $j \in [T]$.} In rescaled units, the greedy rule of (i) charges as much as possible in $A$ and discharges as much as possible outside $A$. The proof of Lemma~\ref{lem:greedy_polymatroid} applies to the rescaled unit with the period-dependent maps $\phi^+_t(\sigma) := \min\{\sigma + \eta^{-t}c^{\max}_t,\ \eta^{-t}\overline{s}\}$ and $\phi^-_t(\sigma) := \max\{\sigma - \eta^{-t}d^{\max}_t,\ \eta^{-t}\underline{s}\}$. These maps are nondecreasing and $1$-Lipschitz, and they map $[\eta^{-(t-1)}\underline{s}, \eta^{-(t-1)}\overline{s}]$ into $[\eta^{-t}\underline{s}, \eta^{-t}\overline{s}]$, because $\eta^{-t}\underline{s} = \underline{s}$ (either $\underline{s} = 0$ or $\eta = 1$) and $\eta^{-t}\overline{s}$ is nondecreasing in $t$; in particular, idling is always feasible. Hence the greedy rule for $A$ attains $\tilde{b}(A)$. Since the greedy rules for $A$ and for $A \cap [1,j]$ coincide on $[1,j]$, the greedy rule for $A$ also attains $\tilde{b}(A \cap [1,j])$ for every $j \in [T]$.

\emph{Step 2: we rewrite $x(A)$ by Abel summation.} Let $w_j := \eta^j$ for $j \in [T]$ and $w_{T+1} := 0$, so that $w_j - w_{j+1} \geq 0$. For every $\boldsymbol{x} \in \mathcal{X}$,
\[
x(A) \overset{(1)}{=} \sum_{q \in A} w_q\, \tilde{x}_q \overset{(2)}{=} \sum_{j=1}^{T} (w_j - w_{j+1})\, \tilde{x}(A \cap [1,j]) \overset{(3)}{\leq} \sum_{j=1}^{T} (w_j - w_{j+1})\, \tilde{b}(A \cap [1,j]).
\]
Equality (1) is the definition of $\tilde{\boldsymbol{x}}$. Equality (2) is an Abel summation, or summation by parts: expanding each partial sum and exchanging the order of summation,
\[
\sum_{j=1}^{T} (w_j - w_{j+1})\, \tilde{x}(A \cap [1,j]) = \sum_{j=1}^{T} (w_j - w_{j+1}) \sum_{q \in A,\, q \leq j} \tilde{x}_q = \sum_{q \in A} \tilde{x}_q \sum_{j=q}^{T} (w_j - w_{j+1}) = \sum_{q \in A} w_q\, \tilde{x}_q,
\]
because the inner sum telescopes to $w_q - w_{T+1} = w_q$. Inequality (3) follows from $\tilde{x}(A') \leq \tilde{b}(A')$, valid for every $A' \subseteq [T]$, and $w_j - w_{j+1} \geq 0$. By Step~1, the greedy point satisfies (3) with equality. Since (3) bounds $x(A)$ for every $\boldsymbol{x} \in \mathcal{X}$, the greedy point maximizes $x(A)$ over $\mathcal{X}$, that is, it attains $b(A)$, and
\[
b(A) = \sum_{j=1}^{T} (w_j - w_{j+1})\, \tilde{b}(A \cap [1,j]).
\]
Each map $A \mapsto \tilde{b}(A \cap [1,j])$ is submodular, because $\tilde{b}$ is submodular and intersecting with $[1,j]$ commutes with unions and intersections, and the coefficients are nonnegative, so $b$ is submodular. Finally, the greedy rule charges $\min\{c^{\max}_t, \overline{s} - \eta\, s_{t-1}\} \geq 0$ in $A$, because $\eta\, s_{t-1} \leq \overline{s}$, and discharges $\max\{-d^{\max}_t, \underline{s} - \eta\, s_{t-1}\} \leq 0$ outside $A$, because $\eta\, s_{t-1} \geq \underline{s}$ (either $\underline{s} = 0$ or $\eta = 1$). As in the proof of Proposition~\ref{prop:submodularity}(i), this gives $f = b$.

Part (ii) follows in the same way, with $-\tilde{x}(A') \leq -\tilde{p}(A')$ in (3), the greedy rule of Corollary~\ref{cor:submodular_disch}(i) in Step~1, and $-\tilde{p}$ submodular.
\hfill$\square$
\end{proof}

\paragraph{Validity of the cuts.} The proofs of Theorem~\ref{th:LVI} and of the validity part of Proposition~\ref{prop:cut_strengthening} for~\eqref{eq:VIc_D} only use $x(B) \leq b(B) = f(B)$ for $\boldsymbol{x} \in \mathcal{X}$ and the submodularity of $f$, which Lemma~\ref{lem:extensions}(i) provides. For~\eqref{eq:VId_D}, the reflection of Appendix~\ref{app:reflection} does not apply when $\eta < 1$, because the mirrored dynamics acquire an affine term; instead, the same argument as for~\eqref{eq:VIc_D}, applied to $-\boldsymbol{x}$ with $B := (\Omega \setminus U_{\boldsymbol{y}}) \cup D$, gives $d(\Omega) - c(D) = -x(B) \leq -p(B) = h(B)$, and Lemma~\ref{lem:submodular_bound} for the submodular $h$ concludes. Hence~\eqref{eq:VIc}, \eqref{eq:VId}, \eqref{eq:VIc_D}, and~\eqref{eq:VId_D} remain valid, with $f$, $h$, and the gains computed by the greedy rules of Lemma~\ref{lem:extensions}.

\section{Supplementary Material for Section~\ref{sec:facet_defining_conditions}: Proof of Theorem~\ref{th:facets}}\label{ec:proof_facets}

This section proves Theorem~\ref{th:facets}. Appendix~\ref{ec:auxiliary} collects auxiliary results, including the proof of Lemma~\ref{lem:full_dimensional}, Appendix~\ref{ec:thresholds} develops the tools about thresholds and attainment sets on which the proof relies, Appendix~\ref{ec:necessity} proves that each of (C1)--(C4) is necessary for the cut~\eqref{eq:VIc_D} to define a facet, and Appendix~\ref{ec:sufficiency} proves that together they are sufficient. Throughout, Assumption~\ref{ass:interior_s0} holds, $\Omega \subseteq [T]$ with $|\Omega| \geq 2$, $l := \min \Omega$, $k := \max \Omega$, and, except in the first part of the proof of Proposition~\ref{prop:necessity}, $D$ satisfies $D_\Omega \subseteq D \subseteq D^+_\Omega$ as in~\eqref{eq:D_Omega}; we use the notation of Section~\ref{sec:facet_defining_conditions}, in particular the face $F_\Omega$ of~\eqref{eq:face_Omega}, the relaxation family $\mathcal{R}_\Omega$, the attainment sets $\mathcal{A}_R$, the thresholds $\kappa^B_t$ of Lemma~\ref{lem:thresholds}, and the baseline functions $L_t$ of~\eqref{eq:L_t}. We write $E := [1,k] \setminus \Omega$ for the periods before $l$ and the gaps of $\Omega$, and $\boldsymbol{z} \in \mathcal{P}^{0\text{-}1}$ for the greedy point of $\Omega$, with net charging $\boldsymbol{x} := \boldsymbol{x}^{(\Omega)}$ and SoC $\boldsymbol{s} := \boldsymbol{s}^{(\Omega)}$: $\boldsymbol{c} = \boldsymbol{x}^+$, $\boldsymbol{d} = \boldsymbol{x}^-$, and $\boldsymbol{u} = \boldsymbol{1}_\Omega$. Since $c(\Omega) = f(\Omega)$, $d(D) = 0$, and $u_\tau = 1$ for $\tau \in \Omega$, the point $\boldsymbol{z}$ satisfies~\eqref{eq:VIc_D} with equality, so $\boldsymbol{z} \in F_\Omega$.

\subsection{Auxiliary Results}\label{ec:auxiliary}

Lemma~\ref{lem:full_dimensional} guarantees that the facets of $\operatorname{conv}(\mathcal{P}^{0\text{-}1})$ are its faces of dimension $3T - 1$; it is used in both parts of the proof.

\begin{proof}{Proof of Lemma~\ref{lem:full_dimensional}}
Let $0 < \varepsilon \leq \min\{\overline{c}, \overline{d}, \overline{s} - s_0, s_0 - \underline{s}\}$, which is positive by Assumption~\ref{ass:interior_s0}. The idle point $\boldsymbol{y}^{0} := (\boldsymbol{0}, \boldsymbol{0}, \boldsymbol{0})$ belongs to $\mathcal{P}^{0\text{-}1}$, with $s_t = s_0$ for all $t$. For each $\tau \in [T]$, consider the following three points, which differ from $\boldsymbol{y}^{0}$ only in period $\tau$:
\begin{itemize}
    \item idle in charging mode: $u_\tau = 1$;
    \item charging: $u_\tau = 1$ and $c_\tau = \varepsilon$;
    \item discharging: $d_\tau = \varepsilon$.
\end{itemize}
Each of them shifts the SoC by at most $\varepsilon$ from period $\tau$ onwards, so all of them belong to $\mathcal{P}^{0\text{-}1}$. Their differences with $\boldsymbol{y}^{0}$ are the vectors $\boldsymbol{e}_{u_\tau}$, $\boldsymbol{e}_{u_\tau} + \varepsilon\,\boldsymbol{e}_{c_\tau}$, and $\varepsilon\,\boldsymbol{e}_{d_\tau}$, $\tau \in [T]$, which are $3T$ linearly independent vectors. Hence, the $3T + 1$ points are affinely independent.
\hfill$\square$
\end{proof}

We use the following standard property of facets of full-dimensional polytopes.

\begin{lemma}\label{lem:facet_multiple}
Let $P \subseteq \mathbb{R}^n$ be a full-dimensional polytope, let $\boldsymbol{\theta}^\top \boldsymbol{y} \leq \theta_0$ define a facet $F$ of $P$, and let $\boldsymbol{\vartheta}^\top \boldsymbol{y} \leq \vartheta_0$ be valid for $P$ with $F \subseteq \{\boldsymbol{y} : \boldsymbol{\vartheta}^\top \boldsymbol{y} = \vartheta_0\}$. Then $(\boldsymbol{\vartheta}, \vartheta_0) = \lambda\,(\boldsymbol{\theta}, \theta_0)$ for some $\lambda \geq 0$. In particular, if $\boldsymbol{\theta}^\top \boldsymbol{y} \leq \theta_0$ is the sum of two valid inequalities, each of them is a nonnegative multiple of it.
\end{lemma}

\begin{proof}{Proof}
Since $\dim F = n - 1$, the affine hull of $F$ is the hyperplane $\{\boldsymbol{\theta}^\top \boldsymbol{y} = \theta_0\}$. The equation $\boldsymbol{\vartheta}^\top \boldsymbol{y} = \vartheta_0$ holds on this hyperplane, so $(\boldsymbol{\vartheta}, \vartheta_0) = \lambda (\boldsymbol{\theta}, \theta_0)$ for some $\lambda \in \mathbb{R}$. Since $P$ is full-dimensional, it contains a point $\hat{\boldsymbol{y}}$ with $\boldsymbol{\theta}^\top \hat{\boldsymbol{y}} < \theta_0$; validity of the second inequality at $\hat{\boldsymbol{y}}$ gives $\lambda\,\boldsymbol{\theta}^\top \hat{\boldsymbol{y}} \leq \lambda \theta_0$, hence $\lambda \geq 0$. For the last claim, if the inequality is the sum of two valid inequalities, both of them hold with equality wherever their sum does, in particular on $F$.
\hfill$\square$
\end{proof}

The next two lemmas describe how the greedy trajectories of Lemma~\ref{lem:greedy_polymatroid} depend on the set of charging periods. Throughout, $s^{(A)}$ denotes the SoC trajectory of the greedy rule for a set $A \subseteq [T]$ on the whole horizon $[T]$ with initial SoC $s_0$.

\begin{lemma}\label{lem:greedy_monotone}
If $A \subseteq A' \subseteq [T]$, then $s^{(A)}_t \leq s^{(A')}_t$ for all $t \in [0,T]$. Moreover, $s^{(A)}_t \geq \underline{\sigma}_t$ for all $t$.
\end{lemma}

\begin{proof}{Proof}
Both trajectories start at $s_0$. In each period, the greedy rule maps the SoC $\sigma$ at the end of the previous period to $\phi^+(\sigma) = \min\{\sigma + \overline{c}, \overline{s}\}$ if the period is a charging period and to $\phi^-(\sigma) = \max\{\sigma - \overline{d}, \underline{s}\}$ otherwise. Both maps are nondecreasing and $\phi^- \leq \phi^+$. Hence, if $s^{(A)}_{t-1} \leq s^{(A')}_{t-1}$, then $s^{(A)}_t \leq s^{(A')}_t$, because every charging period of $A$ is also a charging period of $A'$. The first claim follows by induction. The second holds because $\underline{\sigma}_t$ is the minimum SoC reachable at the end of period~$t$.
\hfill$\square$
\end{proof}

\begin{lemma}\label{lem:decomposition}
Let $\Omega \subseteq [T]$ with $|\Omega| \geq 2$, $l := \min \Omega$, $k := \max \Omega$, and $s := s^{(\Omega)}$. For $t \in [l-1, k-1]$, let $\Omega_1 := \Omega \cap [1,t]$ and $\Omega_2 := \Omega \cap [t+1,T]$, and for $A \subseteq \Omega$ let $A_1 := A \cap \Omega_1$ and $A_2 := A \cap \Omega_2$.
\begin{enumerate}[(a)]
    \item If $\overline{s} - s_t \geq \sum_{q \in \Omega_2} \overline{c}_q$, then $f(A) = f(A_1) + \sum_{q \in A_2} \overline{c}_q$ for every $A \subseteq \Omega$.
    \item If $t \notin \Omega$ and $s_t = \underline{s}$, then $f(A) = f(A_1) + f(A_2)$ for every $A \subseteq \Omega$.
\end{enumerate}
\end{lemma}

\begin{proof}{Proof}
Let $A \subseteq \Omega$. Up to period $t$, the greedy rules for $A$ and for $A_1$ make the same decisions, because $A$ and $A_1$ contain the same periods of $[1,t]$. Hence, $s^{(A)}_q = s^{(A_1)}_q$ for $q \leq t$, and the energy charged by the greedy rule for $A$ in $A_1$ equals $f(A_1)$. By Lemma~\ref{lem:greedy_polymatroid}(i), it remains to determine the energy charged by the greedy rule for $A$ in $A_2$.

(a) First, in every period $q' \in A$ the greedy rule for $A$ charges at most $\overline{c}_{q'}$: since $s^{(A)}_{q'-1} \geq \underline{\sigma}_{q'-1}$ by Lemma~\ref{lem:greedy_monotone}, it charges $\min\{\overline{c}, \overline{s} - s^{(A)}_{q'-1}\} \leq \min\{\overline{c}, \overline{s} - \underline{\sigma}_{q'-1}\} = \overline{c}_{q'}$. In the periods not in $A$, it discharges. Now let $q \in A_2$. Using also $s^{(A)}_t \leq s_t$ (Lemma~\ref{lem:greedy_monotone}), we get
\[
s^{(A)}_{q-1} \leq s_t + \sum_{q' \in \Omega,\ t < q' < q} \overline{c}_{q'} \leq \overline{s} - \overline{c}_q,
\]
where the last inequality follows from the hypothesis. Hence, $\overline{s} - s^{(A)}_{q-1} \geq \overline{c}_q$. Since also $\overline{c} \geq \overline{c}_q$, both arguments of the minimum are at least $\overline{c}_q$, so the greedy rule charges $\min\{\overline{c}, \overline{s} - s^{(A)}_{q-1}\} \geq \overline{c}_q$ in period~$q$. Together with the upper bound shown first, it charges exactly $\overline{c}_q$.

(b) By Lemma~\ref{lem:greedy_monotone}, $\underline{s} \leq s^{(A)}_t \leq s_t = \underline{s}$ and, likewise, $s^{(A_2)}_t = \underline{s}$. From period $t+1$ on, the greedy rules for $A$ and for $A_2$ start from the same SoC and contain the same periods, so they make the same decisions. Hence, the energy charged by the greedy rule for $A$ in $A_2$ equals $f(A_2)$.
\hfill$\square$
\end{proof}

\subsection{Thresholds, Attainment Sets, and the Baseline}\label{ec:thresholds}

This section collects the tools used in the proofs of necessity and sufficiency: the greedy points of the attainment sets, the thresholds of Lemma~\ref{lem:thresholds}, and the formula for the baseline. The reader may skip it and return to it when a result is invoked.

\paragraph{Charging periods, free periods, and attainment sets.} Recall the maps $\phi^+(\sigma) = \min\{\sigma + \overline{c}, \overline{s}\}$ and $\phi^-(\sigma) = \max\{\sigma - \overline{d}, \underline{s}\}$, which give the SoC reached at the end of a period that starts with SoC $\sigma$ and charges, respectively discharges, as much as possible.

\begin{definition}\label{def:free_periods}
Let $R \in \mathcal{R}_\Omega$. The \emph{charging periods} of $R$ are the periods of $B_R := \Omega \setminus R$. The \emph{free periods} of $R$ are the periods of $[T] \setminus B_R = E \cup R \cup (k, T]$, where $E = [1,k] \setminus \Omega$ contains the periods before $l$ and the gaps of $\Omega$.
\end{definition}

On the attainment set $\mathcal{A}_R$, the cut reduces to $x(B_R) = f(B_R)$ (Lemma~\ref{lem:face_structure}): only the charging periods appear in it. The decisions in the free periods are constrained only by the feasibility of the storage. They can be used to make room or to keep energy, as long as the charging periods still collect $f(B_R)$. Free periods are of three kinds: the periods of $E$, the relaxed periods $\tau \in R$, which discharge instead of charging, and the periods after $k$.

The \emph{greedy point} $\boldsymbol{z}^R \in \mathcal{P}^{0\text{-}1}$ of $\mathcal{A}_R$ is defined by the greedy rule for $B_R$ on $[T]$: it charges as much as possible in every charging period and discharges as much as possible, i.e., makes as much room as possible, in every free period. We write $\boldsymbol{x}^R := \boldsymbol{x}^{(B_R)}$ and $\boldsymbol{s}^R := \boldsymbol{s}^{(B_R)}$ for its net charging and SoC, and set $\boldsymbol{c} = (\boldsymbol{x}^R)^+$, $\boldsymbol{d} = (\boldsymbol{x}^R)^-$, and $\boldsymbol{u} = \boldsymbol{1}_{B_R}$. By Lemma~\ref{lem:greedy_polymatroid}(i), $x^R(B_R) = f(B_R)$, so $\boldsymbol{z}^R \in \mathcal{A}_R$. Note that $\boldsymbol{z}^\emptyset = \boldsymbol{z}$, the greedy point of $\Omega$, and we write $\boldsymbol{z}^\tau := \boldsymbol{z}^{\{\tau\}}$ for the greedy point of the attainment set that relaxes only~$\tau$.

It is convenient to describe $\mathcal{A}_R$ in the space of net charging. For $R \in \mathcal{R}_\Omega$, let
\[
\Phi_R := \Big\{ \boldsymbol{x} \in \mathcal{X}_{[T]}(s_0) :\ x_\tau \leq 0 \ (\tau \in R),\ \ x_\tau \geq 0 \ (\tau \in \Omega \setminus (D \cup R)),\ \ x(B_R) = f(B_R) \Big\}
\]
be the set of feasible schedules that discharge (or idle) in the relaxed periods, charge (or idle) in the other periods of $\Omega \setminus D$, and collect the maximum $f(B_R)$ in the charging periods. Without the last condition, these constraints define a polytope contained in $\mathcal{X}_{[T]}(s_0)$, so $x(B_R) \leq f(B_R)$ on it, and $\boldsymbol{x}^R$ belongs to it and attains this bound. Hence, $\Phi_R$ is the set of maximizers of a linear function over a polytope, and therefore a nonempty \emph{convex} polytope. It describes $\mathcal{A}_R$ exactly:
\begin{itemize}
    \item the net charging of every $\boldsymbol{y} \in \mathcal{A}_R$ lies in $\Phi_R$;
    \item conversely, for $\boldsymbol{x} \in \Phi_R$, the point $(\boldsymbol{x}^-, \boldsymbol{x}^+, \boldsymbol{u})$ belongs to $\mathcal{A}_R$, and hence to $F_\Omega$ by Lemma~\ref{lem:face_structure}, for every $\boldsymbol{u}$ with $u_\tau = 0$ for $\tau \in R$, $u_\tau = 1$ for $\tau \in \Omega \setminus (D \cup R)$, $u_t = 1$ if $x_t > 0$, and $u_t = 0$ if $x_t < 0$.
\end{itemize}
We call such a point a \emph{realization} of $\boldsymbol{x}$; the modes of the periods $t \notin \Omega \setminus D$ with $x_t = 0$ can be chosen freely.

\paragraph{Thresholds.} Fix a set $B \subseteq [T]$ of charging periods and a period $t$. Consider the maximum energy that can still be charged in the periods of $B$ after $t$, as a function of the SoC $\sigma$ at the end of~$t$:
\[
V^B_{t+1}(\sigma) := f_{[t+1,T]}(B \cap [t+1,T];\, \sigma), \qquad \sigma \in [\underline{s}, \overline{s}].
\]
By Lemma~\ref{lem:greedy_polymatroid}, it is attained by the greedy rule for $B$ on $[t+1,T]$ started at $\sigma$. A fuller storage can only reduce the remaining charge, so $V^B_{t+1}$ is nonincreasing. Lemma~\ref{lem:thresholds} describes its shape through the thresholds $\kappa^B_t$ (Figure~\ref{fig:threshold}); we prove it here. For example, take $T = 3$, $\underline{s} = 0$, $\overline{s} = 1$, $\overline{c} = \overline{d} = 0.5$, and $B = \{3\}$. Then $\kappa^B_3 = 1$, $\kappa^B_2 = 0.5$, and $\kappa^B_1 = \kappa^B_0 = 1$. The charge of period~3 is not reduced as long as the storage holds at most $0.5$ at the end of period~2, and any SoC at the end of period~1 is harmless, because period~2 can discharge $0.5$.

\begin{figure}
\centering
\begin{tikzpicture}[x=1cm, y=0.9cm, >=stealth]
    \draw[->] (0,0) -- (6.8,0) node[right]{$\sigma$};
    \draw[->] (0,0) -- (0,4.9) node[above]{$V^B_{t+1}(\sigma)$};
    \draw[very thick] (0,4) -- (2.5,4) -- (6,0.5);
    \draw[dashed] (2.5,0) -- (2.5,4);
    \draw[dashed] (6,0) -- (6,0.5);
    \node[below] at (0,0) {$\underline{s}$};
    \node[below] at (2.5,0) {$\kappa^B_t$};
    \node[below] at (6,0) {$\overline{s}$};
    \node[left] at (0,4) {$V^B_{t+1}(\underline{s})$};
    \draw[dotted] (1.5,0) -- (1.5,4);
    \node[below] at (1.5,0) {$\zeta$};
    \draw[dotted] (4.5,0) -- (4.5,2);
    \node[below] at (4.5,0) {$\zeta'$};
    \draw[dotted] (2.5,4) -- (5.3,4);
    \draw[dotted] (4.5,2) -- (5.3,2);
    \draw[<->] (5.3,2) -- (5.3,4) node[midway, right]{$V(\zeta) - V(\zeta')$};
    \draw[<->] (2.5,-0.9) -- (4.5,-0.9) node[midway, below]{$\zeta' - \kappa^B_t$};
    \node at (1.25,4.45) {\small harmless};
    \node at (3.9,3.4) {\small slope $-1$};
\end{tikzpicture}
\caption{The threshold of Lemma~\ref{lem:thresholds}. The remaining charge $V^B_{t+1}(\sigma)$ of $B$ after period~$t$ is constant as long as the SoC $\sigma$ at the end of $t$ does not exceed $\kappa^B_t$, and decreases with slope $-1$ beyond it. Raising the SoC from $\zeta$ to $\zeta'$ reduces the remaining charge by $V(\zeta) - V(\zeta') = \zeta' - \kappa^B_t$, the part of $[\zeta,\zeta']$ above the threshold; see~\eqref{eq:V_difference}.}
\label{fig:threshold}
\end{figure}

\begin{proof}{Proof of Lemma~\ref{lem:thresholds}}
Recall that $V^B_t(\sigma)$ is the maximum charge of $B$ in $[t,T]$ when the SoC at the end of period~$t-1$ is $\sigma$. By Lemma~\ref{lem:greedy_polymatroid}(i), this maximum is attained by the greedy rule. In period~$t$, the greedy rule moves the SoC to $\phi^+(\sigma)$ and charges $\phi^+(\sigma) - \sigma$ if $t \in B$, and moves it to $\phi^-(\sigma)$ if $t \notin B$; from period $t+1$ on, it is the greedy rule started at the SoC just reached, which collects the maximum charge of $B$ in $[t+1,T]$ from that SoC, again by Lemma~\ref{lem:greedy_polymatroid}(i). Hence, for $\sigma \in [\underline{s}, \overline{s}]$, $V^B$ satisfies the backward recursion
\begin{equation}\label{eq:bellman}
V^B_t(\sigma) = \begin{cases}
\phi^+(\sigma) - \sigma + V^B_{t+1}\big(\phi^+(\sigma)\big) & \text{if } t \in B,\\[4pt]
V^B_{t+1}\big(\phi^-(\sigma)\big) & \text{if } t \notin B,
\end{cases} \qquad V^B_{T+1} \equiv 0.
\end{equation}
We also use the elementary identity
\begin{equation}\label{eq:min_identity}
\zeta - (\zeta - \zeta')^+ = \min\{\zeta, \zeta'\} \qquad (\zeta, \zeta' \in \mathbb{R}),
\end{equation}
which follows by distinguishing $\zeta \leq \zeta'$ (both sides equal $\zeta$) and $\zeta > \zeta'$ (both sides equal $\zeta'$).

We proceed by backward induction on $t$. For $t = T$, $V^B_{T+1} \equiv 0 = 0 - (\sigma - \overline{s})^+$ on $[\underline{s}, \overline{s}]$, so the claim holds with $\kappa^B_T = \overline{s}$. Assume the claim for $t$, i.e., $V^B_{t+1}(\sigma') = v - (\sigma' - \kappa)^+$ with $v := V^B_{t+1}(\underline{s})$ and $\kappa := \kappa^B_t \in [\underline{s}, \overline{s}]$. We prove it for $t-1$.

\emph{Case $t \in B$.} By~\eqref{eq:bellman}, the induction hypothesis, and~\eqref{eq:min_identity},
\[
V^B_t(\sigma) = \phi^+(\sigma) - \sigma + v - \big(\phi^+(\sigma) - \kappa\big)^+ = v + \min\{\phi^+(\sigma), \kappa\} - \sigma.
\]
Since $\phi^+(\sigma) = \min\{\sigma + \overline{c}, \overline{s}\}$ and $\kappa \leq \overline{s}$, we have $\min\{\phi^+(\sigma), \kappa\} = \min\{\sigma + \overline{c}, \kappa\}$, and therefore
\[
V^B_t(\sigma) = v + \min\{\sigma + \overline{c}, \kappa\} - \sigma = v + \min\{\overline{c},\ \kappa - \sigma\}.
\]
It remains to write this in threshold form.
\begin{itemize}
    \item If $\kappa - \overline{c} \geq \underline{s}$, then $\min\{\overline{c}, \kappa - \sigma\} = \overline{c} - (\sigma - (\kappa - \overline{c}))^+$ (both sides equal $\overline{c}$ if $\sigma \leq \kappa - \overline{c}$, and $\kappa - \sigma$ otherwise). Hence $V^B_t(\sigma) = (v + \overline{c}) - (\sigma - (\kappa - \overline{c}))^+$, which is the claim with $\kappa^B_{t-1} = \kappa - \overline{c}$ and $V^B_t(\underline{s}) = v + \overline{c}$.
    \item If $\kappa - \overline{c} < \underline{s}$, then $\kappa - \sigma \leq \kappa - \underline{s} < \overline{c}$ for every $\sigma \geq \underline{s}$, so $V^B_t(\sigma) = v + \kappa - \sigma = (v + \kappa - \underline{s}) - (\sigma - \underline{s})$. This is the claim with $\kappa^B_{t-1} = \underline{s}$.
\end{itemize}
In both cases $\kappa^B_{t-1} = \max\{\kappa - \overline{c}, \underline{s}\}$.

\emph{Case $t \notin B$.} By~\eqref{eq:bellman}, $V^B_t(\sigma) = V^B_{t+1}(\phi^-(\sigma))$, with $\phi^-(\sigma) = \max\{\sigma - \overline{d}, \underline{s}\}$. By the induction hypothesis,
\[
V^B_t(\sigma) = v - \big(\max\{\sigma - \overline{d}, \underline{s}\} - \kappa\big)^+ = v - (\sigma - \overline{d} - \kappa)^+,
\]
where the last equality holds because $\underline{s} - \kappa \leq 0$, so replacing $\max\{\sigma - \overline{d}, \underline{s}\}$ by $\sigma - \overline{d}$ does not change the positive part. Writing $(\sigma - \overline{d} - \kappa)^+ = (\sigma - (\kappa + \overline{d}))^+$, this is the claim with $\kappa^B_{t-1} = \kappa + \overline{d}$ if $\kappa + \overline{d} \leq \overline{s}$. If $\kappa + \overline{d} > \overline{s}$, the positive part vanishes for all $\sigma \leq \overline{s}$, so $V^B_t$ is constant on $[\underline{s}, \overline{s}]$, which is the claim with $\kappa^B_{t-1} = \overline{s}$. In both cases $\kappa^B_{t-1} = \min\{\kappa + \overline{d}, \overline{s}\}$, and $V^B_t(\underline{s}) = v$.
\hfill$\square$
\end{proof}

Read backwards in time, the recursion of Lemma~\ref{lem:thresholds} has a direct meaning:
\begin{itemize}
    \item A charging period can absorb up to $\overline{c}$. For it not to be blocked, the storage must have $\overline{c}$ more room before it, so the threshold drops by $\overline{c}$.
    \item A free period can discharge up to $\overline{d}$ and thereby create room, so the threshold rises by $\overline{d}$.
    \item The truncations at $\underline{s}$ and $\overline{s}$ reflect that the SoC cannot leave its bounds.
\end{itemize}
The form of $V^B_{t+1}$ implies that, for $\underline{s} \leq \zeta \leq \zeta' \leq \overline{s}$,
\begin{equation}\label{eq:V_difference}
V^B_{t+1}(\zeta) - V^B_{t+1}(\zeta') = (\zeta' - \kappa^B_t)^+ - (\zeta - \kappa^B_t)^+,
\end{equation}
which is the length of $[\zeta, \zeta'] \cap [\kappa^B_t, \overline{s}]$. In words, raising the SoC at the end of $t$ from $\zeta$ to $\zeta'$ reduces the remaining charge of $B$ exactly by the part of the increase that lies above the threshold.

The gains $\pi_\tau$, condition (C3), and the lifting coefficients all involve the same quantity: the effect on the capacity of charging, instead of discharging, in a single period $t$, while all the other periods keep their roles. The next lemma computes it once. Charging in $t$ gains the charge of period $t$; but the storage is then fuller, and the periods of $B$ after $t$ lose the part of the SoC range between discharging and charging in $t$ that lies above the threshold.

\begin{lemma}\label{lem:marginal_value}
Let $B \subseteq [T]$ and $t \in [T] \setminus B$. Let $\sigma := s^{(B)}_{t-1}$ be the SoC of the greedy rule for $B$ at the end of period $t-1$, and let $\kappa := \kappa^B_t$. Then
\begin{equation}\label{eq:marginal_t}
\begin{aligned}
f(B \cup \{t\}) - f(B) &= \big(\phi^+(\sigma) - \sigma\big) - \Big[\big(\phi^+(\sigma) - \kappa\big)^+ - \big(\phi^-(\sigma) - \kappa\big)^+\Big] \\
&= \min\{\phi^+(\sigma), \kappa\} - \sigma + \big(\phi^-(\sigma) - \kappa\big)^+.
\end{aligned}
\end{equation}
In particular, $f(B \cup \{t\}) > f(B)$ if and only if $\sigma < \kappa$.
\end{lemma}

\begin{proof}{Proof}
Let $V := V^B_{t+1}$. By Lemma~\ref{lem:greedy_polymatroid}(i), $f(B \cup \{t\})$ and $f(B)$ are attained by the greedy rule.

\emph{Step 1: we compare the two greedy trajectories period by period.}
\begin{itemize}
    \item \emph{Before $t$.} The sets $B \cup \{t\}$ and $B$ contain the same periods of $[1, t-1]$, so the two greedy rules make the same decisions there. They charge the same amount, say $X$, in these periods, and both reach the SoC $\sigma$ at the end of period $t-1$.
    \item \emph{In period $t$.} The greedy rule for $B \cup \{t\}$ charges as much as possible: it reaches the SoC $\phi^+(\sigma)$ and charges $\phi^+(\sigma) - \sigma$, which counts towards $B \cup \{t\}$. The greedy rule for $B$ discharges as much as possible: it reaches $\phi^-(\sigma)$, and period $t$ does not count towards $B$.
    \item \emph{After $t$.} The two sets contain the same periods of $[t+1, T]$, namely those of $B$. Hence each greedy rule charges, after $t$, the maximum $V$ of its SoC at the end of period $t$.
\end{itemize}
Adding up the three parts,
\[
f(B \cup \{t\}) = X + \big(\phi^+(\sigma) - \sigma\big) + V(\phi^+(\sigma)), \qquad f(B) = X + V(\phi^-(\sigma)),
\]
and therefore $f(B \cup \{t\}) - f(B) = \phi^+(\sigma) - \sigma - \big[V(\phi^-(\sigma)) - V(\phi^+(\sigma))\big]$. In words: charging in $t$ gains $\phi^+(\sigma) - \sigma$ in period $t$, but the bracket is lost later because the storage is fuller.

\emph{Step 2: we write this difference in threshold form.} By~\eqref{eq:V_difference} with $\zeta = \phi^-(\sigma) \leq \zeta' = \phi^+(\sigma)$, the bracket equals $(\phi^+(\sigma) - \kappa)^+ - (\phi^-(\sigma) - \kappa)^+$, which gives the first expression in~\eqref{eq:marginal_t}. The second one follows from~\eqref{eq:min_identity}, because $\phi^+(\sigma) - (\phi^+(\sigma) - \kappa)^+ = \min\{\phi^+(\sigma), \kappa\}$.

\emph{Step 3: we determine the sign.} Suppose first that $\sigma \geq \kappa$. Since $\sigma \leq \overline{s}$, we have $\phi^+(\sigma) \geq \sigma \geq \kappa$, so $\min\{\phi^+(\sigma), \kappa\} = \kappa$. Since $\sigma \geq \underline{s}$, we have $\phi^-(\sigma) \leq \sigma$; as $(\,\cdot - \kappa)^+$ is nondecreasing, $(\phi^-(\sigma) - \kappa)^+ \leq (\sigma - \kappa)^+ = \sigma - \kappa$. Hence, by~\eqref{eq:marginal_t}, $f(B \cup \{t\}) - f(B) \leq \kappa - \sigma + (\sigma - \kappa) = 0$. Suppose now that $\sigma < \kappa$. Then $\sigma < \overline{s}$, because $\kappa \leq \overline{s}$, so $\phi^+(\sigma) > \sigma$ and $\min\{\phi^+(\sigma), \kappa\} > \sigma$. As the positive part in~\eqref{eq:marginal_t} is nonnegative, $f(B \cup \{t\}) - f(B) > 0$.
\hfill$\square$
\end{proof}

Applied to $B = \Omega \setminus \{\tau\}$ and $t = \tau$, Lemma~\ref{lem:marginal_value} gives the gains $\pi_\tau$ (Proposition~\ref{prop:two_passes}); applied to $B = B_R$ and $t \notin \Omega$, it expresses (C3) through the capacity of an attainment set (Lemmas~\ref{lem:C3_thresholds} and~\ref{lem:C3_equiv}).

Two consequences of Lemma~\ref{lem:thresholds} are used repeatedly. On an attainment set, every schedule continues optimally from each period on. And modifying a schedule so that the storage holds a little more energy, without exceeding the threshold, keeps it on the attainment set.

\begin{corollary}\label{cor:thresholds}
Let $R \in \mathcal{R}_\Omega$, $B := B_R$, and $t \in [0,T]$.
\begin{enumerate}[(a)]
    \item \emph{(Optimal continuation.)} For every $\boldsymbol{x} \in \mathcal{X}_{[T]}(s_0)$, $x(B \cap [t+1,T]) \leq V^{B}_{t+1}(s_t(\boldsymbol{x}))$, with equality if $\boldsymbol{x} \in \Phi_R$.
    \item \emph{(Harmless energy.)} Let $\boldsymbol{x} \in \Phi_R$, and let $\boldsymbol{x}' \in \mathcal{X}_{[T]}(s_0)$ satisfy the sign constraints of $\Phi_R$, coincide with $\boldsymbol{x}$ on $B \cap [1,t]$, and follow the greedy rule for $B$ after $t$. If $s_t(\boldsymbol{x}) \leq \kappa^R_t$ and $s_t(\boldsymbol{x}') \leq \kappa^R_t$, then $\boldsymbol{x}' \in \Phi_R$.
\end{enumerate}
\end{corollary}

\begin{proof}{Proof}
(a) The restriction of $\boldsymbol{x}$ to $[t+1,T]$ is a feasible continuation from $s_t(\boldsymbol{x})$, so the inequality holds by definition of $V^B_{t+1}$. If $\boldsymbol{x} \in \Phi_R$ and the inequality were strict, replacing $\boldsymbol{x}$ after $t$ by the greedy continuation for $B$ from $s_t(\boldsymbol{x})$ would give a point of $\mathcal{X}_{[T]}(s_0)$ with $x(B) > f(B)$, contradicting the definition of $f$.

(b) We have
\[
x'(B) \overset{(1)}{=} x(B \cap [1,t]) + V^B_{t+1}(s_t(\boldsymbol{x}')) \overset{(2)}{=} x(B \cap [1,t]) + V^B_{t+1}(s_t(\boldsymbol{x})) \overset{(3)}{=} f(B).
\]
Equality (1) holds because $\boldsymbol{x}'$ coincides with $\boldsymbol{x}$ on $B \cap [1,t]$ and follows the greedy rule for $B$ after $t$, which attains $V^B_{t+1}(s_t(\boldsymbol{x}'))$ by Lemma~\ref{lem:greedy_polymatroid}(i). Equality (2) holds because $V^B_{t+1}$ is constant on $[\underline{s}, \kappa^R_t]$ by Lemma~\ref{lem:thresholds}, and both $s_t(\boldsymbol{x})$ and $s_t(\boldsymbol{x}')$ lie in this interval. Equality (3) is part~(a) for $\boldsymbol{x} \in \Phi_R$, together with $x(B) = f(B)$. Finally, the greedy continuation for $B$ charges in $B$ and discharges elsewhere, so $\boldsymbol{x}'$ satisfies the sign constraints of $\Phi_R$ after $t$ as well. Hence $\boldsymbol{x}' \in \Phi_R$.
\hfill$\square$
\end{proof}

\emph{What (C3) means for the face.} Condition (C3) at a period $t \notin \Omega$ asks that some greedy schedule $\boldsymbol{z}^R$ reaches period $t$ with an SoC below the threshold after $t$. The next lemma shows that this is exactly what a point of the attainment set $\mathcal{A}_R$ needs in order to charge in~$t$: charging in $t$ is possible on $\mathcal{A}_R$ if and only if the energy charged is harmless for the charging periods that follow. Hence (C3) fails at $t$ precisely when no point of the face charges in~$t$.

\begin{lemma}\label{lem:C3_thresholds}
Let $t \in [T] \setminus \Omega$ and $R \in \mathcal{R}_\Omega$. Some $\boldsymbol{y} \in \mathcal{A}_R$ has $c_t > 0$ if and only if $s^R_{t-1} < \kappa^R_t$. Consequently, (C3) holds at $t$ if and only if $F_\Omega \not\subseteq \{c_t = 0\}$.
\end{lemma}

\begin{proof}{Proof}
Let $B := B_R$, $\sigma := s^R_{t-1}$, and $\kappa := \kappa^R_t$. Note that $t \notin B$, because $t \notin \Omega$; thus $t$ is a free period of $R$, and no sign constraint of $\Phi_R$ involves it.

\emph{Only if.} Let $\boldsymbol{y} \in \mathcal{A}_R$ with $c_t > 0$, and let $\boldsymbol{x}$ be its net charging. Since $\boldsymbol{y} \in \mathcal{A}_R$, $x(B) = f(B)$; since $c_t > 0$ forces $u_t = 1$ and hence $d_t = 0$, $x_t = c_t$. Therefore $x(B \cup \{t\}) = f(B) + c_t$. As $\boldsymbol{x} \in \mathcal{X}_{[T]}(s_0)$ and $f(B \cup \{t\})$ is the maximum of $x(B \cup \{t\})$ over $\mathcal{X}_{[T]}(s_0)$, we get $f(B \cup \{t\}) \geq f(B) + c_t > f(B)$. By Lemma~\ref{lem:marginal_value}, $\sigma < \kappa$.

\emph{If.} Let $\sigma < \kappa$. Let $\boldsymbol{x}'$ coincide with $\boldsymbol{x}^R$ before $t$, charge $x'_t := \varepsilon$ in period $t$, with $0 < \varepsilon \leq \min\{\overline{c},\ \kappa - \sigma\}$, and follow the greedy rule for $B$ after $t$. Its SoC at the end of period $t$ is $\sigma + \varepsilon \leq \kappa \leq \overline{s}$, so $\boldsymbol{x}' \in \mathcal{X}_{[T]}(s_0)$. It coincides with $\boldsymbol{x}^R \in \Phi_R$ on $B \cap [1,t]$ and, like $\boldsymbol{x}^R$, satisfies the sign constraints of $\Phi_R$ up to period $t$. Moreover, $\boldsymbol{x}^R$ discharges or idles in the free period $t$, so $s^R_t \leq \sigma < \kappa$. By Corollary~\ref{cor:thresholds}(b), $\boldsymbol{x}' \in \Phi_R$, and its realization with $u_t = 1$ is a point of $\mathcal{A}_R$ with $c_t = \varepsilon > 0$.

\emph{Consequence.} By Lemma~\ref{lem:face_structure}, $F_\Omega$ is the convex hull of the attainment sets $\mathcal{A}_R$, $R \in \mathcal{R}_\Omega$, and $c_t \geq 0$ on all of them. Hence $F_\Omega \not\subseteq \{c_t = 0\}$ if and only if some $\mathcal{A}_R$ contains a point with $c_t > 0$, that is, by the first part, if and only if $s^R_{t-1} < \kappa^R_t$ for some $R \in \mathcal{R}_\Omega$, which is (C3) at $t$.
\hfill$\square$
\end{proof}

\paragraph{Two facts about relaxed periods.} A set $R \in \mathcal{R}_\Omega$ relaxes several periods whose gains add up: relaxing them together yields exactly the sum of the individual gains $\pi_\tau$. The first lemma shows that this property is inherited by subsets. If the gains of $R$ add up, so do those of any part of $R$, and removing some relaxations from $R$ changes $f$ by the sum of the corresponding gains. This allows us to split a relaxation set into the relaxations before and after a given period.

\begin{lemma}\label{lem:R_down_closed}
If $R \in \mathcal{R}_\Omega$ and $R_1 \subseteq R$, then $R_1 \in \mathcal{R}_\Omega$, and $f(\Omega \setminus R) - f(\Omega \setminus R_1) = \sum_{\tau \in R \setminus R_1} \pi_\tau$.
\end{lemma}

\begin{proof}{Proof}
Recall that $\pi_\tau = -f(\tau \mid \Omega \setminus \{\tau\})$ for $\tau \in \Omega$. We show the chain
\begin{equation}\label{eq:down_closed_chain}
\begin{aligned}
f(\Omega \setminus R) &\overset{(1)}{\leq} f(\Omega \setminus R_1) - \sum_{\tau \in R \setminus R_1} f\big(\tau \mid (\Omega \setminus R_1) \setminus \{\tau\}\big) \\
&\overset{(2)}{\leq} f(\Omega \setminus R_1) + \sum_{\tau \in R \setminus R_1} \pi_\tau \\
&\overset{(3)}{\leq} f(\Omega) + \sum_{\tau \in R} \pi_\tau.
\end{aligned}
\end{equation}
\begin{itemize}
    \item Inequality (1) is Lemma~\ref{lem:submodular_bound} applied to the sets $U := \Omega \setminus R \subseteq W := \Omega \setminus R_1$, i.e., with $W$ in place of $\Omega$. Indeed, $U \subseteq W$ because $R_1 \subseteq R$, and $W \setminus U = R \setminus R_1$.
    \item Inequality (2) holds term by term. Let $\tau \in R \setminus R_1$. Then $(\Omega \setminus R_1) \setminus \{\tau\} \subseteq \Omega \setminus \{\tau\}$ and $\tau \notin \Omega \setminus \{\tau\}$, so Definition~\ref{def:submodularity} gives $f(\tau \mid (\Omega \setminus R_1) \setminus \{\tau\}) \geq f(\tau \mid \Omega \setminus \{\tau\}) = -\pi_\tau$, that is, $-f(\tau \mid (\Omega \setminus R_1) \setminus \{\tau\}) \leq \pi_\tau$. Summing over $\tau \in R \setminus R_1$ gives (2).
    \item Inequality (3) follows from Lemma~\ref{lem:submodular_bound} applied to $U := \Omega \setminus R_1 \subseteq \Omega$, which gives $f(\Omega \setminus R_1) \leq f(\Omega) - \sum_{\tau \in R_1} f(\tau \mid \Omega \setminus \{\tau\}) = f(\Omega) + \sum_{\tau \in R_1} \pi_\tau$. Adding $\sum_{\tau \in R \setminus R_1} \pi_\tau$ to both sides gives (3), because $R_1$ and $R \setminus R_1$ partition $R$.
\end{itemize}
Since $R \in \mathcal{R}_\Omega$, the left-hand side of~\eqref{eq:down_closed_chain} equals its right-hand side, $f(\Omega) + \sum_{\tau \in R} \pi_\tau$. Hence (1)--(3) hold with equality. Equality in (3) reads $f(\Omega \setminus R_1) = f(\Omega) + \sum_{\tau \in R_1} \pi_\tau$; as $R_1 \subseteq R \subseteq \Omega \setminus D$, this means $R_1 \in \mathcal{R}_\Omega$. Equality in (1) and (2) reads $f(\Omega \setminus R) = f(\Omega \setminus R_1) + \sum_{\tau \in R \setminus R_1} \pi_\tau$, which is the second claim.
\hfill$\square$
\end{proof}

\emph{What happens when a period is relaxed.} Relaxing $\tau \in \Omega$ means discharging in $\tau$ instead of charging. Compared with the greedy point $\boldsymbol{z}$ of $\Omega$, which charges in $\tau$, the storage holds less energy after $\tau$, so more can be charged later. The gain $\pi_\tau$ is the balance between this extra charge after $\tau$ and the charge lost in $\tau$ itself. By~\eqref{eq:V_difference}, the extra charge after $\tau$ equals the part of the SoC range $[\phi^-(s_{\tau-1}), \phi^+(s_{\tau-1})]$, between discharging and charging in $\tau$ from the SoC $s_{\tau-1}$ of $\boldsymbol{z}$, that lies above the threshold $\kappa_\tau = \kappa^\Omega_\tau$ of $\Omega$ at the end of $\tau$.

The second lemma has two parts. Part~(a) concerns the periods with a positive gain: discharging exactly $\pi_\tau$ in $\tau$, instead of charging, brings the SoC down to the threshold $\kappa_\tau$, or to the lowest level that a discharge can reach if that level lies above the threshold. The periods of $\Omega$ after $\tau$ then recover as additional charge all the energy discharged; this yields the points of the face in which a period of $D$ discharges, used in Steps~2 and~4 of the proof of sufficiency. Part~(b) says that, for $\tau \notin D_\Omega$, the greedy discharge in $\tau$ always brings the SoC down to the threshold or below. Hence, after relaxing $\tau$, a little extra energy in storage is harmless, with one exception: the discharge empties the storage, and from an empty storage the charging periods after $\tau$ fill it completely, so that $\kappa_\tau = \underline{s}$. In that exceptional situation, the storage is a bottleneck after $\tau$.

\begin{lemma}\label{lem:relaxed_threshold}
Let $\tau \in \Omega$.
\begin{enumerate}[(a)]
    \item If $\pi_\tau > 0$, then $s_{\tau-1} - \pi_\tau = \max\{\kappa_\tau,\ \phi^-(s_{\tau-1})\}$. In particular, $\phi^-(s_{\tau-1}) \leq s_{\tau-1} - \pi_\tau < s_{\tau-1}$ and $V^\Omega_{\tau+1}(s_{\tau-1} - \pi_\tau) = V^\Omega_{\tau+1}(\phi^-(s_{\tau-1}))$.
    \item If $\tau \notin D_\Omega$, then $\phi^-(s_{\tau-1}) \leq \kappa_\tau$. If equality holds, then $\kappa_\tau = \phi^-(s_{\tau-1}) = \underline{s}$, $\pi_\tau = s_{\tau-1} - \underline{s} < \overline{d}$, and $s_{\tau-1} < \overline{\sigma}_{\tau-1}$.
\end{enumerate}
Since $D_\Omega \subseteq D$, part~(b) applies to every $\tau \in \Omega \setminus D$.
\end{lemma}

\begin{proof}{Proof}
Write $\sigma := s_{\tau-1}$ for the SoC of $\boldsymbol{z}$ at the start of period $\tau$, and $\kappa := \kappa_\tau$. The threshold at the end of $\tau$ only depends on the periods after $\tau$, so $\kappa$ is also the threshold of $\Omega \setminus \{\tau\}$ at $\tau$. Let $V := V^\Omega_{\tau+1}$. Recall that $\overline{d}_\tau = \min\{\overline{d}, \overline{\sigma}_{\tau-1} - \underline{s}\}$ and that $\pi_\tau \leq \overline{d}_\tau$ by~\eqref{eq:pi_bound}.

\emph{Step 1: we write the gain as a balance.} In period $\tau$, $\boldsymbol{z}$ charges $x_\tau = \phi^+(\sigma) - \sigma$ and reaches $\phi^+(\sigma)$, whereas the greedy point $\boldsymbol{z}^\tau$ discharges and reaches $\phi^-(\sigma)$. By Proposition~\ref{prop:two_passes}, that is, by Lemma~\ref{lem:marginal_value} applied to $B = \Omega \setminus \{\tau\}$ and $t = \tau$, the gain is the extra charge after $\tau$, i.e., the part of $[\phi^-(\sigma), \phi^+(\sigma)]$ above $\kappa$, minus the charge lost in $\tau$:
\begin{equation}\label{eq:pi_threshold}
\pi_\tau = (\phi^+(\sigma) - \kappa)^+ - (\phi^-(\sigma) - \kappa)^+ - (\phi^+(\sigma) - \sigma).
\end{equation}

\emph{Step 2: we distinguish three regimes.} Formula~\eqref{eq:pi_threshold} depends on the position of $\kappa$ relative to the interval $[\phi^-(\sigma), \phi^+(\sigma)]$.
\begin{enumerate}[(i)]
    \item If $\kappa \geq \phi^+(\sigma)$, both positive parts vanish and $\pi_\tau = -(\phi^+(\sigma) - \sigma) \leq 0$. The room created by relaxing $\tau$ is never needed, and only the charge in $\tau$ is lost.
    \item If $\phi^-(\sigma) \leq \kappa < \phi^+(\sigma)$, then $\pi_\tau = (\phi^+(\sigma) - \kappa) - (\phi^+(\sigma) - \sigma) = \sigma - \kappa$.
    \item If $\kappa < \phi^-(\sigma)$, then $\pi_\tau = (\phi^+(\sigma) - \kappa) - (\phi^-(\sigma) - \kappa) - (\phi^+(\sigma) - \sigma) = \sigma - \phi^-(\sigma)$. In this regime, the whole SoC range lies above the threshold: every unit discharged in $\tau$ is recovered as charge later, so the gain equals the full discharge $\sigma - \phi^-(\sigma) = \min\{\overline{d}, \sigma - \underline{s}\}$.
\end{enumerate}

\emph{Step 3: we prove (a).} Let $\pi_\tau > 0$. Regime (i) is excluded, because it gives $\pi_\tau \leq 0$. In regime (ii), $\sigma - \pi_\tau = \kappa \geq \phi^-(\sigma)$; in regime (iii), $\sigma - \pi_\tau = \phi^-(\sigma) > \kappa$. In both, $\sigma - \pi_\tau = \max\{\kappa, \phi^-(\sigma)\}$, which is at least $\phi^-(\sigma)$ and, as $\pi_\tau > 0$, smaller than $\sigma$. Finally, $V(\sigma - \pi_\tau) = V(\phi^-(\sigma))$: in regime (iii) the two arguments coincide, and in regime (ii) both lie in $[\underline{s}, \kappa]$, where $V$ is constant by Lemma~\ref{lem:thresholds}.

\emph{Step 4: we show that regime (iii) is impossible if $\tau \notin D_\Omega$, i.e., $\phi^-(\sigma) \leq \kappa$.} Suppose $\kappa < \phi^-(\sigma)$. Since $\kappa \geq \underline{s}$, we have $\phi^-(\sigma) > \underline{s}$: the discharge in $\tau$ is limited by the power $\overline{d}$, not by the empty storage. Hence $\phi^-(\sigma) = \sigma - \overline{d}$ and, by regime (iii), $\pi_\tau = \overline{d}$. As $\pi_\tau \leq \overline{d}_\tau \leq \overline{d}$, this gives $\pi_\tau = \overline{d}_\tau$, i.e., $\tau \in D_\Omega$, a contradiction.

\emph{Step 5: we characterize the case of equality in (b).} Let $\tau \notin D_\Omega$ and suppose $\phi^-(\sigma) = \kappa$, the boundary between regimes (ii) and (iii). Both formulas give $\pi_\tau = \sigma - \phi^-(\sigma) = \min\{\overline{d}, \sigma - \underline{s}\}$.
\begin{itemize}
    \item If $\phi^-(\sigma) > \underline{s}$, the argument of Step~4 gives $\tau \in D_\Omega$. Hence $\phi^-(\sigma) = \underline{s}$, and therefore $\kappa = \phi^-(\sigma) = \underline{s}$: the discharge in $\tau$ empties the storage, and $\pi_\tau = \sigma - \underline{s} \leq \overline{d}$.
    \item If $\sigma - \underline{s} = \overline{d}$, then $\pi_\tau = \overline{d}$, and again $\tau \in D_\Omega$. Hence $\sigma - \underline{s} < \overline{d}$.
    \item Finally, $\tau \notin D_\Omega$ means $\pi_\tau < \overline{d}_\tau$, i.e., $\sigma - \underline{s} < \min\{\overline{d}, \overline{\sigma}_{\tau-1} - \underline{s}\}$, which gives $\sigma < \overline{\sigma}_{\tau-1}$.
\end{itemize}
\hfill$\square$
\end{proof}

\paragraph{The baseline.} The last tool is an explicit formula for the baseline $L_t$ of~\eqref{eq:L_t} on each attainment set. On $\mathcal{A}_{\hat{R}}$, the baseline at $t$ depends only on the relaxations of $\hat{R}$ after $t$ and on whether the SoC at the end of $t$ lies below the threshold that they create.

\begin{lemma}\label{lem:baseline}
Let $t \in [l-1, k-1]$, $\hat{R} \in \mathcal{R}_\Omega$, $R'' := \hat{R} \cap (t, k]$, and let $\hat{\boldsymbol{y}} \in \mathcal{A}_{\hat{R}}$ have SoC $\hat{s}_t$ at the end of period~$t$. Then
\[
L_t(\hat{\boldsymbol{y}}) = C_t(R'') + \min\{\hat{s}_t,\ \kappa^{R''}_t\}, \qquad \text{where } C_t(R'') := V^{B_{R''}}_{t+1}(\underline{s}) - f(\Omega) - \sum_{\tau \in R''} \pi_\tau.
\]
\end{lemma}

\begin{proof}{Proof of Lemma~\ref{lem:baseline}}
Let $\hat{\boldsymbol{x}}$ be the net charging of $\hat{\boldsymbol{y}}$, $\hat{B} := B_{\hat{R}}$, $R_< := \hat{R} \cap [1,t]$, and $Y := \hat{x}(\hat{B} \cap [1,t])$. We compute the bracket in~\eqref{eq:L_t}. A period $q \in \Omega_{\leq t}$ contributes $x_q$ to $c(\Omega_{\leq t}) - d(D \cap \Omega_{\leq t})$ if $q \in \hat{B}$ (it charges, or it belongs to $D$) and $0$ if $q \in R_<$ (it does not charge and is not in $D$), so this difference equals $Y$; and $u_q = 0$ exactly on $R_<$ among the periods of $\Omega_{\leq t} \setminus D$, so the sum of binary terms equals $\sum_{\tau \in R_<} \pi_\tau$. Hence
\[
L_t(\hat{\boldsymbol{y}}) = \hat{s}_t - Y + \sum_{\tau \in R_<} \pi_\tau.
\]
By Corollary~\ref{cor:thresholds}(a), $\hat{\boldsymbol{x}}$ continues optimally after $t$, so $f(\hat{B}) = \hat{x}(\hat{B}) = Y + V^{\hat{B}}_{t+1}(\hat{s}_t)$. The sets $\hat{B}$ and $B_{R''}$ contain the same periods after $t$, so $V^{\hat{B}}_{t+1} = V^{B_{R''}}_{t+1} =: V''$ and $\kappa^{\hat{B}}_t = \kappa^{R''}_t =: \kappa''$, and Lemma~\ref{lem:thresholds} gives $V''(\hat{s}_t) = V''(\underline{s}) - (\hat{s}_t - \kappa'')^+$. Finally, $f(\hat{B}) = f(\Omega) + \sum_{\tau \in \hat{R}} \pi_\tau$ because $\hat{R} \in \mathcal{R}_\Omega$, and $\hat{R}$ is the disjoint union of $R_<$ and $R''$. Substituting,
\[
L_t(\hat{\boldsymbol{y}}) = \hat{s}_t - f(\Omega) - \sum_{\tau \in R''} \pi_\tau + V''(\underline{s}) - (\hat{s}_t - \kappa'')^+ = C_t(R'') + \min\{\hat{s}_t, \kappa''\},
\]
by the identity $\zeta - (\zeta - \zeta')^+ = \min\{\zeta, \zeta'\}$ of~\eqref{eq:min_identity}.
\hfill$\square$
\end{proof}

\subsection{Necessity of (C1)--(C4)}\label{ec:necessity}

\begin{proposition}\label{prop:necessity}
Let Assumption~\ref{ass:interior_s0} hold, and let $D \subseteq \Omega \subseteq [T]$ with $|\Omega| \geq 2$. If~\eqref{eq:VIc_D} defines a facet of $\operatorname{conv}(\mathcal{P}^{0\text{-}1})$, then $D_\Omega \subseteq D \subseteq D^+_\Omega$ and $(\Omega, D)$ satisfies (C1)--(C4).
\end{proposition}

\begin{proof}{Proof}
Throughout, for $D' \subseteq A \subseteq [T]$, we write the cut~\eqref{eq:VIc_D} for the sets $A$ and $D'$ as $\Psi_{A,D'}(\boldsymbol{y}) \leq f(A)$, with
\[
\Psi_{A,D'}(\boldsymbol{y}) := c(A) - d(D') - \sum_{\tau \in A \setminus D'} \pi_\tau(A)\,(1 - u_\tau),
\]
where $\pi_\tau(A) := -f(\tau \mid A \setminus \{\tau\})$ is the gain of $\tau$ for the set $A$, and $\pi_\tau = \pi_\tau(\Omega)$. This inequality is valid by Proposition~\ref{prop:cut_strengthening}. By Lemma~\ref{lem:full_dimensional}, Lemma~\ref{lem:facet_multiple} applies to $P = \operatorname{conv}(\mathcal{P}^{0\text{-}1})$.

\emph{The range~\eqref{eq:D_Omega} of $D$.} A set $D$ outside this range makes the cut the sum of another cut of the family and of one of the two inequalities compared in Proposition~\ref{prop:term_comparison}. Suppose first that some $\tau \in D$ has $\pi_\tau \leq 0$, and let $D' := D \setminus \{\tau\}$. Then
\[
\Psi_{\Omega,D}(\boldsymbol{y}) = \Psi_{\Omega,D'}(\boldsymbol{y}) + \big[\pi_\tau\,(1 - u_\tau) - d_\tau\big],
\]
so the cut for $D$ is the sum of the cut for $D'$ and of the inequality $\pi_\tau (1 - u_\tau) - d_\tau \leq 0$, which is valid because $\pi_\tau \leq 0$, $u_\tau \leq 1$, and $d_\tau \geq 0$. Suppose now that some $\tau \in \Omega \setminus D$ has $\pi_\tau = \overline{d}_\tau$, and let $D' := D \cup \{\tau\}$. Then
\[
\Psi_{\Omega,D}(\boldsymbol{y}) = \Psi_{\Omega,D'}(\boldsymbol{y}) + \big[d_\tau - \overline{d}_\tau\,(1 - u_\tau)\big],
\]
so the cut for $D$ is the sum of the cut for $D'$ and of the singleton cut $d_\tau \leq \overline{d}_\tau (1 - u_\tau)$ of~\eqref{eq:singleton_cuts}. In both cases, the second inequality has coefficient $0$ for $c_\tau$ and a nonzero coefficient for $d_\tau$, whereas the cut for $D$ has coefficient $1$ for $c_\tau$. Hence, it is not a multiple of the cut for $D$, and by Lemma~\ref{lem:facet_multiple} the latter is not facet-defining. In the rest of the proof, $D_\Omega \subseteq D \subseteq D^+_\Omega$.

\emph{(C1).} Suppose that (C1) fails at some $t \in \{l-1\} \cup ([l, k-1] \setminus \Omega)$, and define $\Omega_1$ and $\Omega_2$ as in Lemma~\ref{lem:decomposition}. Note that $k \in \Omega_2$. The hypothesis of Lemma~\ref{lem:decomposition}(a) only involves $s = s^{(\Omega)}$, so its conclusion holds for every $A \subseteq \Omega$. Applied to $A = \Omega$ and, for $\tau \in \Omega$, to $A = \Omega \setminus \{\tau\}$, it gives
\begin{align*}
f(\Omega) &= f(\Omega_1) + \textstyle\sum_{q \in \Omega_2} \overline{c}_q, \tag{i}\\
f(\Omega \setminus \{\tau\}) &= f(\Omega_1) + \textstyle\sum_{q \in \Omega_2 \setminus \{\tau\}} \overline{c}_q \quad \text{if } \tau \in \Omega_2, \tag{ii}\\
f(\Omega \setminus \{\tau\}) &= f(\Omega_1 \setminus \{\tau\}) + \textstyle\sum_{q \in \Omega_2} \overline{c}_q \quad \text{if } \tau \in \Omega_1. \tag{iii}
\end{align*}
Subtracting, we obtain the marginal value $f(\tau \mid \Omega \setminus \{\tau\}) = f(\Omega) - f(\Omega \setminus \{\tau\})$:
\begin{itemize}
    \item for $\tau \in \Omega_2$, (i) $-$ (ii) gives $f(\tau \mid \Omega \setminus \{\tau\}) = \overline{c}_\tau > 0$, so $\pi_\tau = -\overline{c}_\tau < 0$ and $\tau \notin D$, because $D \subseteq D^+_\Omega$;
    \item for $\tau \in \Omega_1$, (i) $-$ (iii) gives $f(\tau \mid \Omega \setminus \{\tau\}) = f(\tau \mid \Omega_1 \setminus \{\tau\})$, so the gain of $\tau$ is the same for $\Omega$ and for $\Omega_1$: $\pi_\tau(\Omega) = \pi_\tau(\Omega_1)$.
\end{itemize}
By (i), and since $D \subseteq \Omega_1$, the cut for $\Omega$ and $D$ is, coefficient by coefficient, the sum of the cut $\Psi_{\Omega_1,D}(\boldsymbol{y}) \leq f(\Omega_1)$ (absent if $\Omega_1 = \emptyset$) and the singleton cuts $c_q - \overline{c}_q\,u_q \leq 0$, $q \in \Omega_2$, of~\eqref{eq:singleton_cuts}. All these inequalities are valid. The singleton cut for $q = k$ has coefficient $1$ for $c_k$ and $0$ for $c_{q'}$, $q' \in \Omega \setminus \{k\} \neq \emptyset$, whereas the cut for $\Omega$ has coefficient $1$ for all of them. Hence, it is not a multiple of the cut for $\Omega$, and by Lemma~\ref{lem:facet_multiple} the latter is not facet-defining.

\emph{(C2).} Suppose that (C2) fails at some $t \in [l, k-1] \setminus \Omega$, i.e., $s^{(\Omega)}_t = \underline{s}$. Then $\Omega_1 \ni l$ and $\Omega_2 \ni k$ are both nonempty. As for (C1), Lemma~\ref{lem:decomposition}(b) holds for every $A \subseteq \Omega$, so $f(A) = f(A \cap \Omega_1) + f(A \cap \Omega_2)$; applied to $A = \Omega$ and $A = \Omega \setminus \{\tau\}$, it gives $f(\Omega) = f(\Omega_1) + f(\Omega_2)$ and, subtracting, $f(\tau \mid \Omega \setminus \{\tau\}) = f(\tau \mid \Omega_i \setminus \{\tau\})$ for $\tau \in \Omega_i$, $i = 1,2$, because the term of the part that does not contain $\tau$ cancels. Hence, $\pi_\tau(\Omega) = \pi_\tau(\Omega_i)$ for $\tau \in \Omega_i$, and the cut for $\Omega$ and $D$ is the sum of the cuts $\Psi_{\Omega_1, D \cap \Omega_1}(\boldsymbol{y}) \leq f(\Omega_1)$ and $\Psi_{\Omega_2, D \cap \Omega_2}(\boldsymbol{y}) \leq f(\Omega_2)$. The cut for $\Omega_1$ has coefficient $0$ for $c_k$, whereas the cut for $\Omega$ has coefficient $1$; hence, by Lemma~\ref{lem:facet_multiple}, the cut for $\Omega$ is not facet-defining.

\emph{(C3).} Suppose that (C3) fails at some $t \in [T] \setminus \Omega$. By Lemma~\ref{lem:C3_thresholds}, $F_\Omega \subseteq \{c_t = 0\}$: no attainment set reaches period $t$ below its threshold, so no point of the face can charge in $t$. The inequality $-c_t \leq 0$ is valid and holds with equality on $F_\Omega$. Its coefficient for $c_t$ is $-1$, whereas the cut for $\Omega$ has coefficient $0$ for $c_t$, because $t \notin \Omega$. Hence, by Lemma~\ref{lem:facet_multiple}, the cut for $\Omega$ is not facet-defining.

\emph{(C4).} Suppose that (C4) fails at some $\tau \in \Omega \setminus D$ with $\kappa_\tau = \underline{s}$, that is, $\kappa^R_\tau \leq s_\tau$ for every $R \in \mathcal{R}_\Omega$ with $R \subseteq (\tau, k]$. We show that $L_\tau$ is constant on $F_\Omega$. Let $\hat{R} \in \mathcal{R}_\Omega$ and $\hat{\boldsymbol{y}} \in \mathcal{A}_{\hat{R}}$, with net charging $\hat{\boldsymbol{x}}$ and SoC $\hat{s}_\tau$ at the end of $\tau$. Split $\hat{R}$ into $R_< := \hat{R} \cap [1,\tau]$ and $R'' := \hat{R} \cap (\tau, k]$; both belong to $\mathcal{R}_\Omega$ by Lemma~\ref{lem:R_down_closed}. Let $\hat{B} := B_{\hat{R}}$, let $V'' := V^{\hat{B}}_{\tau+1} = V^{B_{R''}}_{\tau+1}$, with threshold $\kappa'' := \kappa^{R''}_\tau \leq s_\tau$, and let $V := V^\Omega_{\tau+1}$, whose threshold is $\kappa_\tau = \underline{s}$, so that $V(\sigma) = V(\underline{s}) - (\sigma - \underline{s})$ by Lemma~\ref{lem:thresholds}. Finally, let $Y := \hat{x}(\hat{B} \cap [1,\tau])$ and $X := x(\Omega \cap [1,\tau])$.

\emph{Step 1: we show that $\hat{s}_\tau \geq \kappa''$.} We compare four greedy values.
\begin{itemize}
    \item[(1)] $f(\hat{B}) = Y + V''(\hat{s}_\tau)$, by Corollary~\ref{cor:thresholds}(a), because $\hat{\boldsymbol{x}} \in \Phi_{\hat{R}}$ continues optimally after $\tau$.
    \item[(2)] $f(\Omega \setminus R_<) \geq Y + V(\hat{s}_\tau)$, because $\hat{\boldsymbol{x}}$ on $[1,\tau]$ followed by the greedy rule for $\Omega$ from $\hat{s}_\tau$ is a feasible schedule, which charges $Y$ in $(\Omega \setminus R_<) \cap [1,\tau] = \hat{B} \cap [1,\tau]$ and $V(\hat{s}_\tau)$ after $\tau$.
    \item[(3)] $f(\Omega \setminus R'') = X + V''(s_\tau)$, because the greedy rule for $\Omega \setminus R''$ coincides with $\boldsymbol{z}$ on $[1,\tau]$.
    \item[(4)] $f(\Omega) = X + V(s_\tau)$, by definition of $\boldsymbol{z}$.
\end{itemize}
By Lemma~\ref{lem:R_down_closed}, the gains of $R_<$ add up with or without $R''$: $(1) - (3) = \sum_{\tau' \in R_<} \pi_{\tau'} = f(\Omega \setminus R_<) - (4)$. With (2), this gives $(1) - (3) \geq (2) - (4)$, that is, after canceling $Y$ and $X$,
\[
V''(\hat{s}_\tau) - V''(s_\tau) \geq V(\hat{s}_\tau) - V(s_\tau) = -(\hat{s}_\tau - s_\tau).
\]
Writing $V''$ in threshold form and using $s_\tau \geq \kappa''$, the left-hand side equals $-(\hat{s}_\tau - \kappa'')^+ + (s_\tau - \kappa'')$, so the inequality reads $\hat{s}_\tau - \kappa'' \geq (\hat{s}_\tau - \kappa'')^+$. A number is at least its positive part only if it is nonnegative, so $\hat{s}_\tau \geq \kappa''$.

\emph{Step 2: we show that $L_\tau$ is constant on $F_\Omega$.} By Lemma~\ref{lem:baseline} and Step~1, $L_\tau(\hat{\boldsymbol{y}}) = C_\tau(R'') + \kappa''$. The greedy point $\boldsymbol{z}^{R''}$ of $\mathcal{A}_{R''}$ coincides with $\boldsymbol{z}$ on $[1,\tau]$, because $R'' \subseteq (\tau, k]$; its SoC at the end of $\tau$ is $s_\tau \geq \kappa''$, so Lemma~\ref{lem:baseline} gives $L_\tau(\boldsymbol{z}^{R''}) = C_\tau(R'') + \kappa''$ as well, and $L_\tau(\boldsymbol{z}^{R''}) = L_\tau(\boldsymbol{z})$ because $L_\tau$ only involves periods up to $\tau$. Hence $L_\tau(\hat{\boldsymbol{y}}) = L_\tau(\boldsymbol{z})$ for every point of $F_\Omega \cap \mathcal{P}^{0\text{-}1}$ and, by Lemma~\ref{lem:face_structure} and linearity, for every point of $F_\Omega$.

\emph{Step 3: we conclude that $F_\Omega$ is not a facet.} Substituting $s_\tau(\boldsymbol{y}) = s_0 + \sum_{q \leq \tau} (c_q - d_q)$ into~\eqref{eq:L_t}, the terms $c_q - d_q$ of the periods in $D \cap \Omega_{\leq \tau}$ cancel, and so do the terms $c_q$ of the periods in $\Omega_{\leq \tau} \setminus D$, which leaves
\[
L_\tau(\boldsymbol{y}) = s_0 + \sum_{q \in [1,\tau] \setminus \Omega} (c_q - d_q) - \sum_{q \in \Omega_{\leq \tau} \setminus D} \big( d_q - \pi_q\,(1 - u_q) \big).
\]
Since $\tau \in \Omega_{\leq \tau} \setminus D$, the coefficient of $d_\tau$ in $L_\tau$ is $-1$, so $L_\tau$ is not a constant function. Moreover, $\tau < k$ because $\kappa_k = \overline{s} > \underline{s}$, so the coefficient of $c_k$ in $L_\tau$ is $0$, whereas in the cut it is $1$. Hence the equation $L_\tau(\boldsymbol{y}) = L_\tau(\boldsymbol{z})$ is not a multiple of the equation of the cut, and $F_\Omega$ is contained in the intersection of two distinct hyperplanes. Therefore, $\dim F_\Omega \leq 3T - 2$, and the cut is not facet-defining.
\hfill$\square$
\end{proof}

\subsection{Sufficiency of (C1)--(C4)}\label{ec:sufficiency}

Assume (C1)--(C4), and let $\boldsymbol{\theta}^\top \boldsymbol{y} = \theta_0$ be an equation satisfied by all points of $F_\Omega$, where $\boldsymbol{\theta} = (\boldsymbol{\theta}^d, \boldsymbol{\theta}^c, \boldsymbol{\theta}^u)$. We show that there exists $\mu \in \mathbb{R}$ such that
\begin{equation}\label{eq:target_coefficients}
\theta^c_q = \mu \ (q \in \Omega), \quad \theta^d_q = -\mu \ (q \in D), \quad \theta^u_\tau = \mu\,\pi_\tau \ (\tau \in \Omega \setminus D),
\end{equation}
and that all other coefficients vanish. Then $(\boldsymbol{\theta}, \theta_0)$ is a multiple of the coefficients of~\eqref{eq:VIc_D}, so $\dim F_\Omega = 3T - 1$. All arguments rely on the following observation:
\begin{itemize}
    \item[(O)] if $\boldsymbol{y} \in F_\Omega$ and $\boldsymbol{y} + \varepsilon\,\boldsymbol{\delta} \in F_\Omega$ for some $\varepsilon > 0$, then $\boldsymbol{\theta}^\top \boldsymbol{\delta} = 0$.
\end{itemize}
To check that a point belongs to $F_\Omega$, it suffices that it belongs to $\mathcal{P}^{0\text{-}1}$ and satisfies~\eqref{eq:VIc_D} with equality.

\emph{Strategy.} The proof identifies the coefficients block by block, following the role of each group of variables in a schedule that charges the maximum in $\Omega$: periods after $k$ (Step~0), charges in $\Omega$ (Step~1), discharges in $D$ (Step~2), free periods (Step~3), modes in $\Omega$ (Step~4), and the right-hand side (Step~5). Each step exhibits pairs of points of $F_\Omega$ that differ in a few coordinates only, typically because a small amount of energy is moved from one period to another through the storage. By (O), each such pair yields a linear relation among the coefficients of these coordinates.

Recall that the greedy point $\boldsymbol{z}$ of $\Omega$, with net charging $\boldsymbol{x}$ and SoC $\boldsymbol{s}$, belongs to $F_\Omega$. Two facts about $\boldsymbol{z}$ are used repeatedly:
\begin{itemize}
    \item[(Z1)] \emph{The greedy point respects the individual charging limits:} $x_q \leq \overline{c}_q$ for every $q \in \Omega$, by Lemma~\ref{lem:greedy_monotone} (see the proof of Lemma~\ref{lem:decomposition});
    \item[(Z2)] \emph{From period $l$ on, the greedy point never empties the storage:} $s_t > \underline{s}$ for every $t \in [l, k]$. For $t \notin \Omega$, this is (C2) (for $t < k$). For $t \in \Omega$, $s_t = \min\{s_{t-1} + \overline{c}, \overline{s}\} \geq \min\{\underline{s} + \overline{c}, \overline{s}\} > \underline{s}$.
\end{itemize}

\emph{Step 0: we show that the coefficients of the periods after $k$ vanish.} Once the last charging period $k$ has passed, the cut no longer constrains the storage, which can be operated freely within its limits. Let $t > k$. Let $\boldsymbol{y}^0$ be obtained from $\boldsymbol{z}$ by making every period after $k$ idle, with $c_{t'} = d_{t'} = 0$ for $t' > k$. The SoC then stays at $s_k \in (\underline{s}, \overline{s}]$ after $k$, and the cut does not involve periods after $k$, so $\boldsymbol{y}^0 \in F_\Omega$ for either value of $u_t$.
\begin{itemize}
    \item Comparing $\boldsymbol{y}^0$ with $u_t = 0$ and $u_t = 1$ gives $\theta^u_t = 0$ by (O).
    \item With $u_t = 0$, setting $d_t = \varepsilon \leq \min\{\overline{d}, s_k - \underline{s}\}$ keeps the point in $F_\Omega$, so $\theta^d_t = 0$.
    \item By (C3) and Lemma~\ref{lem:C3_thresholds}, some $\boldsymbol{y} \in F_\Omega \cap \mathcal{P}^{0\text{-}1}$ has $c_t > 0$. Making the periods after $t$ idle keeps it in $F_\Omega$. Decreasing $c_t$ by $\varepsilon \leq c_t$ then lowers the SoC from $t$ on, but not below $s_{t-1} \geq \underline{s}$, so the point stays in $F_\Omega$. Hence $\theta^c_t = 0$.
\end{itemize}

\emph{Step 1: we show that all the charges in $\Omega$ have the same coefficient.} Every point of the face charges $f(\Omega)$ in $\Omega$, but the face allows this charge to be redistributed: a period that charges can give part of its charge to a later period of $\Omega$ that is not saturated, the energy being kept in storage in between. This shows that all charges in $\Omega$ have the same coefficient. We call $q \in \Omega$ a \emph{giver} if $x_q > 0$ and a \emph{receiver} if $x_q < \overline{c}_q$. By (Z1), every period of $\Omega$ is a giver, a receiver, or both.
\begin{itemize}
    \item $l$ is a giver, because $x_l = \min\{\overline{c}, \overline{s} - \underline{\sigma}_{l-1}\} > 0$, as $\underline{\sigma}_{l-1} \leq s_0 < \overline{s}$.
    \item $k$ is a receiver. Let $t'$ be the last period of $\{l-1\} \cup ([l,k-1] \setminus \Omega)$, so that $(t', k] \subseteq \Omega$. In $(t',k]$ the greedy rule only charges, so it charges at most $\overline{s} - s_{t'}$ there. By (C1) at $t'$, this is less than $\sum_{q \in (t',k]} \overline{c}_q$. Hence $x_q < \overline{c}_q$ for some $q \in (t', k]$. If $q = k$, we are done. Otherwise, the greedy charge in $q$ is $\overline{s} - s_{q-1}$, so $s_q = \overline{s}$, and all later periods of $(t',k]$, which are charging periods, charge $0$. In particular, $x_k = 0 < \overline{c}_k$.
\end{itemize}
Let $q < r$ be a giver and a receiver, and let $\boldsymbol{y} := \boldsymbol{z} + \varepsilon(\boldsymbol{e}_{c_r} - \boldsymbol{e}_{c_q})$ with $0 < \varepsilon \leq \min\{x_q, \overline{c}_r - x_r\}$. The SoC of $\boldsymbol{y}$ decreases by $\varepsilon$ on $[q, r-1]$ and is unchanged elsewhere. By (Z2), it stays above $\underline{s}$ for $\varepsilon$ small enough. The modes are unchanged, $c_r \leq \overline{c}_r \leq \overline{c}\,u_r$, and $c(\Omega)$ and $d(D)$ are unchanged. Hence $\boldsymbol{y} \in F_\Omega$ and, by (O), $\theta^c_q = \theta^c_r$. Applying this to every giver $q \neq k$ with $r = k$, to every receiver $r \neq l$ with $q = l$, and to the pair $(l, k)$, we obtain $\theta^c_q = \theta^c_l =: \mu$ for all $q \in \Omega$.

\emph{Step 2: we determine the discharge coefficients of the periods in $D$.} A period $q \in D$ has a positive gain, and discharging exactly $\pi_q$ in $q$ costs nothing: the periods of $\Omega$ after $q$ recover all of it as additional charge. Moving a little of this compensating charge back links the coefficient of $d_q$ to that of the charges. Let $q \in D$. Since $D \subseteq D^+_\Omega$, we have $\pi_q > 0$; in particular, $q < k$, because $\pi_k = -x_k \leq 0$ by~\eqref{eq:gain_formula}. Let $\boldsymbol{y}^q$ be the point that coincides with $\boldsymbol{z}$ before $q$, discharges $d_q = \pi_q$ in period~$q$, and follows the greedy rule for $\Omega$ after $q$; its modes are those of $\boldsymbol{z}$, except for $u_q = 0$. We show that $\boldsymbol{y}^q \in F_\Omega$.
\begin{itemize}
    \item \emph{The point is feasible.} By Lemma~\ref{lem:relaxed_threshold}(a), its SoC at the end of period~$q$ is $s_{q-1} - \pi_q \in [\phi^-(s_{q-1}),\, s_{q-1})$, so the discharge $\pi_q$ respects the power and SoC limits, and the greedy rule keeps the SoC within its limits afterwards.
    \item \emph{The point satisfies the cut with equality.} It charges $x(\Omega \cap [1,q-1])$ before $q$ and, by Lemma~\ref{lem:greedy_polymatroid}(i), $V^\Omega_{q+1}(s_{q-1} - \pi_q)$ after $q$, which equals $V^\Omega_{q+1}(\phi^-(s_{q-1}))$ by Lemma~\ref{lem:relaxed_threshold}(a). The greedy rule for $\Omega \setminus \{q\}$ coincides with $\boldsymbol{z}$ before $q$, reaches the SoC $\phi^-(s_{q-1})$ at the end of $q$, and charges $V^\Omega_{q+1}(\phi^-(s_{q-1}))$ afterwards, so $f(\Omega \setminus \{q\}) = x(\Omega \cap [1,q-1]) + V^\Omega_{q+1}(\phi^-(s_{q-1}))$. Hence, $\boldsymbol{y}^q$ charges $c(\Omega) = f(\Omega \setminus \{q\}) = f(\Omega) + \pi_q$. Moreover, $q$ is its only period of $\Omega$ that does not charge, so $d(D) = d_q = \pi_q$ and $u_\tau = 1$ for every $\tau \in \Omega \setminus D$. Therefore, $c(\Omega) - d(D) = f(\Omega)$, and all the binary terms of the cut vanish.
\end{itemize}

Now let $r$ be the first period of $\Omega$ after $q$. The periods strictly between $q$ and $r$ do not belong to $\Omega$, so $\boldsymbol{y}^q$ discharges or idles in them, and its SoC does not exceed $s_{q-1} - \pi_q < \overline{s}$ on $[q, r-1]$. Consequently, it charges a positive amount in $r$, namely $c_r = \min\{\overline{c},\ \overline{s} - s_{r-1}(\boldsymbol{y}^q)\} > 0$. Let $\boldsymbol{y} := \boldsymbol{y}^q - \varepsilon(\boldsymbol{e}_{d_q} + \boldsymbol{e}_{c_r})$ with $0 < \varepsilon \leq \min\{\pi_q, c_r\}$: discharge $\varepsilon$ less in $q$ and charge $\varepsilon$ less in $r$. Its SoC increases by $\varepsilon$ on $[q, r-1]$, where it does not exceed $s_{q-1} - \pi_q + \varepsilon \leq s_{q-1} \leq \overline{s}$, and is unchanged elsewhere. The modes are unchanged, and $c(\Omega) - d(D)$ is unchanged. Hence $\boldsymbol{y} \in F_\Omega$, and (O) applied to $\boldsymbol{\delta} = -(\boldsymbol{e}_{d_q} + \boldsymbol{e}_{c_r})$ gives $\theta^d_q + \theta^c_r = 0$. Since $r \in \Omega$, Step~1 gives $\theta^c_r = \mu$, and therefore $\theta^d_q = -\mu$.

\emph{Step 3: we show that the remaining coefficients of the free periods vanish.} After Steps~0--2, the only unknown coefficients belong to periods in which the storage is \emph{not} charging for the cut: the charges, discharges, and modes of the periods of $E$ (before $l$ and in the gaps of $\Omega$), and the discharges of the periods of $\Omega \setminus D$, which only appear on the attainment sets that relax them. We show that
\begin{equation}\label{eq:step3_target}
\theta^d_t = 0 \ \ (t \in E \cup (\Omega \setminus D)), \qquad \theta^c_t = \theta^u_t = 0 \ \ (t \in E).
\end{equation}
Every point of $F_\Omega$ charges the largest possible amount in the periods of $\Omega$ that it does not relax. The coefficients in~\eqref{eq:step3_target} vanish if $F_\Omega$ contains enough schedules that operate the storage differently in the free periods while still charging that maximum, and the thresholds of Appendix~\ref{ec:thresholds} say exactly when a schedule on the face can discharge a little less, or charge a little, in a given period without losing energy in its charging periods: the extra energy must stay below the threshold. One observation turns every such pair of schedules into a statement about a single coefficient.

\emph{The charging periods cancel out.} Let $R \in \mathcal{R}_\Omega$ and let $\boldsymbol{y}, \boldsymbol{y}' \in \mathcal{A}_R$ have the same modes. Then
\begin{equation}\label{eq:bookkeeping}
\boldsymbol{\theta}^\top (\boldsymbol{y}' - \boldsymbol{y}) = \sum_{t \notin B_R} \big[ \theta^c_t\,(c'_t - c_t) + \theta^d_t\,(d'_t - d_t) \big].
\end{equation}
Indeed, the mode terms cancel, and for $q \in B_R$ the coefficients are known from Steps~1 and~2: $\theta^c_q = \mu$, and $\theta^d_q = -\mu$ if $q \in D$, whereas $d_q = d'_q = 0$ if $q \in B_R \setminus D$, because both points charge there. The terms of $B_R$ therefore add up to $\mu\,[x'(B_R) - x(B_R)] = \mu\,[f(B_R) - f(B_R)] = 0$, by Lemma~\ref{lem:face_structure}. In every pair of points used below, the two points coincide before some period and, except where stated, they discharge or idle in every free period, so~\eqref{eq:bookkeeping} and (O) reduce to a relation between a few coefficients. Accordingly, when we say below that two such points differ only in certain coordinates, we refer to the coordinates of the free periods: their charges in $B_R$ may differ as well, but, by~\eqref{eq:bookkeeping}, these do not contribute.

\emph{Step 3.1: we show that the discharge coefficients of the free periods vanish.} We show that $\theta^d_t = 0$ for $t \in E \cup (\Omega \setminus D)$, arguing by contradiction. Suppose that some of these coefficients is nonzero, and let $m$ be the \emph{latest} period of $E \cup (\Omega \setminus D)$ with $\theta^d_m \neq 0$. Then
\begin{equation}\label{eq:maximal_m}
\theta^d_j = 0 \quad \text{for every free period } j > m \text{ of every } R \in \mathcal{R}_\Omega,
\end{equation}
by the choice of $m$ for $j \in E \cup (\Omega \setminus D)$ and by Step~0 for $j > k$. The contradiction comes from a pair of points of some $\mathcal{A}_R$ in which $m$ is free, that coincide before $m$, have the same modes, and differ in $d_m$ and otherwise only in discharges of free periods after $m$: by~\eqref{eq:bookkeeping} and~\eqref{eq:maximal_m}, (O) gives $\theta^d_m\,(d'_m - d_m) = 0$, i.e., $\theta^d_m = 0$. Two constructions produce such pairs from the greedy point $\boldsymbol{z}^R$ whenever it ends period $m$ strictly below the threshold $\kappa^R_m$, that is, whenever a little extra energy in storage at the end of $m$ is harmless for the charging periods of $R$.

\begin{enumerate}[(A)]
    \item \emph{Discharge a little less in $m$.} Let $m$ be a free period of $R$ with $s^R_{m-1} > \underline{s}$ and $s^R_m < \kappa^R_m$. Then $\boldsymbol{z}^R$ discharges in $m$, $x^R_m < 0$. Let $\boldsymbol{x}'$ coincide with $\boldsymbol{x}^R$ before $m$, set $x'_m := x^R_m + \varepsilon$ with $0 < \varepsilon \leq \min\{-x^R_m,\ \kappa^R_m - s^R_m\}$, and follow the greedy rule for $B_R$ after $m$. Since $s_m(\boldsymbol{x}') = s^R_m + \varepsilon \leq \kappa^R_m$ and $\boldsymbol{x}'$ coincides with $\boldsymbol{x}^R$ on $B_R \cap [1,m]$, Corollary~\ref{cor:thresholds}(b) gives $\boldsymbol{x}' \in \Phi_R$. The greedy continuation charges in $B_R$ and discharges or idles in the free periods, so the realization $\boldsymbol{y}'$ of $\boldsymbol{x}'$ with the modes of $\boldsymbol{z}^R$ belongs to $\mathcal{A}_R$, and it differs from $\boldsymbol{z}^R$ only in $d_m$, by $-\varepsilon$, and in discharges of free periods after $m$. Hence $\theta^d_m = 0$.
    \item \emph{Carry energy through an empty storage.} Let $m$ be a free period of $R$ with $s^R_{m-1} = \underline{s} < \kappa^R_m$. Now the storage is empty at the start of $m$, so $\boldsymbol{z}^R$ does not discharge in $m$ and cannot discharge less. Instead, we take a little energy from the last period before $m$ that discharges, keep it in the otherwise empty storage up to $m$, and discharge it in $m$. Since $s^R_0 = s_0 > \underline{s}$ and the SoC only decreases in periods with negative net charging, there is a last period $q < m$ with $x^R_q < 0$; it is a free period, the SoC $\boldsymbol{s}^R$ is nondecreasing on $[q, m-1]$ and ends at $\underline{s}$, so $s^R_t = \underline{s}$ for $t \in [q, m-1]$, and the periods in $(q, m]$ are free and idle in $\boldsymbol{x}^R$, because a charging period charges a positive amount from an empty storage. First, let $\boldsymbol{x}''$ coincide with $\boldsymbol{x}^R$ on $[1,m]$ except for $x''_q := x^R_q + \varepsilon$, with $0 < \varepsilon \leq \min\{-x^R_q,\ \kappa^R_m - \underline{s}\}$, and follow the greedy rule for $B_R$ after $m$. Its SoC is $\underline{s} + \varepsilon \leq \kappa^R_m$ on $[q, m]$ and it coincides with $\boldsymbol{x}^R$ on $B_R \cap [1,m]$, so $\boldsymbol{x}'' \in \Phi_R$ by Corollary~\ref{cor:thresholds}(b), and its realization with the modes of $\boldsymbol{z}^R$ differs from $\boldsymbol{z}^R$ only in $d_q$, by $-\varepsilon$, and in discharges of free periods after $m$; hence $\theta^d_q = 0$. Second, let $\boldsymbol{x}''' := \boldsymbol{x}^R + \varepsilon(\boldsymbol{e}_q - \boldsymbol{e}_m)$: discharge $\varepsilon$ less in $q$ and $\varepsilon$ more in $m$. Its SoC is $\underline{s} + \varepsilon$ on $[q, m-1]$ and coincides with $\boldsymbol{s}^R$ elsewhere, $x'''_q \leq 0$, $x'''_m = -\varepsilon \geq -\overline{d}$, and $x'''(B_R) = f(B_R)$, so $\boldsymbol{x}''' \in \Phi_R$, and its realization with the modes of $\boldsymbol{z}^R$ differs from $\boldsymbol{z}^R$ only in $d_q$ and $d_m$. Hence $\theta^d_m = \theta^d_q = 0$.
\end{enumerate}
It remains to find, for the period $m$, a set $R \in \mathcal{R}_\Omega$ to which (A) or (B) applies.

\emph{Case 1: $m \in E$.} By (C3) at $m$, some $R \in \mathcal{R}_\Omega$ satisfies $s^R_{m-1} < \kappa^R_m$, and $m$ is a free period of $R$. If $s^R_{m-1} > \underline{s}$, then $\boldsymbol{z}^R$ discharges in $m$, so $s^R_m < s^R_{m-1} < \kappa^R_m$ and (A) applies; otherwise (B) applies.

\emph{Case 2: $m \in \Omega \setminus D$ and $\phi^-(s_{m-1}) < \kappa_m$.} Take $R := \{m\}$. Then $\boldsymbol{z}^R = \boldsymbol{z}^m$ coincides with $\boldsymbol{z}$ before $m$, $s^R_m = \phi^-(s_{m-1})$, and $\kappa^R_m = \kappa_m$, because the threshold at $m$ only depends on the periods after $m$. Hence $s^R_m < \kappa^R_m$, and (A) applies if $s_{m-1} > \underline{s}$, whereas (B) applies if $s_{m-1} = \underline{s}$; in (B), $x''_m = 0$ and $x'''_m = -\varepsilon$ respect the constraint $x_m \leq 0$ of $\Phi_{\{m\}}$.

\emph{Case 3: $m \in \Omega \setminus D$ and $\phi^-(s_{m-1}) = \kappa_m$ (the bottleneck case).} By Lemma~\ref{lem:relaxed_threshold}(b), Cases~2 and~3 exhaust the periods of $\Omega \setminus D$, and in Case~3 we have $\kappa_m = \phi^-(s_{m-1}) = \underline{s}$, $\pi_m = s_{m-1} - \underline{s} < \overline{d}$, and $s_{m-1} < \overline{\sigma}_{m-1}$; moreover, $m < k$, because $\kappa_k = \overline{s} > \underline{s}$. Here (A) and the first point of (B) are of no use in $\mathcal{A}_{\{m\}}$: both leave energy in storage at the end of $m$, and the threshold $\kappa_m = \underline{s}$ forbids it, because every point of $\mathcal{A}_{\{m\}}$ ends period $m$ with an empty storage and its discharge in $m$ equals the energy it holds at the start of $m$, $d_m = s_{m-1} - \underline{s}$. The second point of (B), which brings energy from an earlier period into $m$, remains available and gives Stage~2; Stage~3 replaces the first point by a point of another attainment set, given by (C4). We therefore move energy through $m$ rather than around it, in three stages: a period $p' < m$ gives a little energy, this energy is first discharged in $m$ (which ties $\theta^d_m$ to $\theta^d_{p'}$) and then carried beyond $m$ into the room that (C4) provides (which shows $\theta^d_{p'} = 0$).

\emph{Stage 1: we find a period $p' < m$ that discharges, with the storage not full on $[p', m-1]$.} Let $p'$ be the last period before $m$ with $x_{p'} < 0$. It exists: otherwise every $t < m$ would be a period of $\Omega$ or a period of $E$ that idles, and a period of $E$ idles only when $s_{t-1} = \underline{s}$, which is impossible because the SoC can only fall from $s_0 > \underline{s}$ to $\underline{s}$ through a discharge; so $[1, m-1] \subseteq \Omega$ and $s_{m-1} = \overline{\sigma}_{m-1}$, a contradiction. It belongs to $E$, and $s_t < \overline{s}$ for every $t \in [p', m-1]$: indeed, $s_{p'} < s_{p'-1} \leq \overline{s}$, and if $s_t = \overline{s}$ for some $t \in (p', m-1]$, the storage is full after $t$, the periods of $\Omega$ charge nothing, and a period of $E$ would discharge $\min\{\overline{d}, \overline{s} - \underline{s}\} > 0$, contradicting the choice of $p'$; so $(t, m-1] \subseteq \Omega$, the storage stays full, and $s_{m-1} = \overline{s} \geq \overline{\sigma}_{m-1}$, again a contradiction.

\emph{Stage 2: we show that $\theta^d_m = \theta^d_{p'}$.} Let $R := \{m\}$. The greedy point $\boldsymbol{z}^m$ coincides with $\boldsymbol{z}$ before $m$ and discharges $s_{m-1} - \underline{s}$ in $m$. Let $\boldsymbol{x}''' := \boldsymbol{x}^m + \varepsilon(\boldsymbol{e}_{p'} - \boldsymbol{e}_m)$ with $0 < \varepsilon \leq \min\{-x_{p'},\ \overline{d} - (s_{m-1} - \underline{s}),\ \overline{s} - \max_{t \in [p', m-1]} s_t\}$. Its SoC exceeds that of $\boldsymbol{z}^m$ by $\varepsilon$ on $[p', m-1]$, which stays below $\overline{s}$ by Stage~1, and coincides with it elsewhere; $x'''_{p'} \leq 0$; $x'''_m = -(s_{m-1} - \underline{s}) - \varepsilon \in [-\overline{d}, 0]$; and $x'''(B_R) = f(B_R)$, because $p'$ and $m$ are free periods of $R$. Hence $\boldsymbol{x}''' \in \Phi_R$, its realization with the modes of $\boldsymbol{z}^m$ differs from $\boldsymbol{z}^m$ only in $d_{p'}$ and $d_m$, and~\eqref{eq:bookkeeping} gives $-\theta^d_{p'} + \theta^d_m = 0$.

\emph{Stage 3: we show that $\theta^d_{p'} = 0$.} By (C4) at $m$, there is $R'' \in \mathcal{R}_\Omega$ with $R'' \subseteq (m, k]$ and $\kappa^{R''}_m > s_m$. The greedy point $\boldsymbol{z}^{R''}$ coincides with $\boldsymbol{z}$ on $[1,m]$. Let $\boldsymbol{x}''$ coincide with $\boldsymbol{x}$ on $[1,m]$ except for $x''_{p'} := x_{p'} + \varepsilon$, with $0 < \varepsilon \leq \min\{-x_{p'},\ \kappa^{R''}_m - s_m,\ \overline{s} - \max_{t \in [p', m-1]} s_t\}$, and follow the greedy rule for $B_{R''}$ after $m$. Its SoC exceeds that of $\boldsymbol{z}^{R''}$ by $\varepsilon$ on $[p', m]$, which stays within the bounds by Stage~1 and because $s_m + \varepsilon \leq \kappa^{R''}_m \leq \overline{s}$, and, since $\boldsymbol{x}^{R''} = \boldsymbol{x}$ on $[1,m]$ (because $R'' \subseteq (m,k]$), it coincides with $\boldsymbol{x}^{R''}$ on $[1,m]$ except in the free period $p'$, in particular on $B_{R''} \cap [1,m] = \Omega \cap [1,m]$, as Corollary~\ref{cor:thresholds}(b) requires. By Corollary~\ref{cor:thresholds}(b), $\boldsymbol{x}'' \in \Phi_{R''}$: the extra energy is below the threshold, so it is harmless. Its realization with the modes of $\boldsymbol{z}^{R''}$ differs from $\boldsymbol{z}^{R''}$ only in $d_{p'}$, by $-\varepsilon$, and in discharges of free periods after $m$, so~\eqref{eq:bookkeeping} and~\eqref{eq:maximal_m} give $\theta^d_{p'} = 0$. With Stage~2, $\theta^d_m = 0$, which contradicts the choice of $m$. This completes Step~3.1.

\emph{Step 3.2: we show that the charge and mode coefficients of the periods of $E$ vanish.} Let $t \in E$. By (C3), some $R \in \mathcal{R}_\Omega$ satisfies $s^R_{t-1} < \kappa^R_t$. Let $\boldsymbol{x}^0$ and $\boldsymbol{x}^\varepsilon$ coincide with $\boldsymbol{x}^R$ before $t$, set $x^0_t := 0$ and $x^\varepsilon_t := \varepsilon$ with $0 < \varepsilon \leq \min\{\overline{c},\ \kappa^R_t - s^R_{t-1}\}$, and follow the greedy rule for $B_R$ after $t$. Both end period $t$ at or below $\kappa^R_t$, as does $\boldsymbol{x}^R$, so both belong to $\Phi_R$ by Corollary~\ref{cor:thresholds}(b). The two realizations of $\boldsymbol{x}^0$ that differ only in $u_t$ belong to $\mathcal{A}_R$, and (O) gives $\theta^u_t = 0$. The realizations of $\boldsymbol{x}^\varepsilon$ and $\boldsymbol{x}^0$ with $u_t = 1$ have the same modes and differ only in $c_t$, by $\varepsilon$, and in discharges of free periods after $t$, whose coefficients vanish by Step~3.1 and Step~0; hence~\eqref{eq:bookkeeping} gives $\theta^c_t = 0$. This proves~\eqref{eq:step3_target}.

\emph{Step 4: we determine the mode coefficients of the periods of $\Omega$.} Switching a period $\tau \in \Omega$ from charging to discharging replaces its charge by a discharge and changes its mode. Comparing $\boldsymbol{z}$ with a point of the face that does so, all coordinates whose coefficients are already known cancel, which isolates the coefficient of $u_\tau$. Let $\tau \in \Omega$. If $\tau \in D$, let $\boldsymbol{y}^\tau$ be the point of Step~2; if $\tau \in \Omega \setminus D$, let $\boldsymbol{y}^\tau := \boldsymbol{z}^\tau$, the greedy point of the attainment set $\mathcal{A}_{\{\tau\}}$. In both cases, $\boldsymbol{y}^\tau \in F_\Omega$, its modes are those of $\boldsymbol{z}$ except for $u_\tau = 0$, and it charges $c(\Omega) = f(\Omega \setminus \{\tau\}) = f(\Omega) + \pi_\tau$ in $\Omega$, whereas $\boldsymbol{z}$ charges $f(\Omega)$. The coordinates of the periods outside $\Omega$ have zero coefficients by Steps~0 and~3. Hence, (O) gives
\[
0 = \boldsymbol{\theta}^\top (\boldsymbol{y}^\tau - \boldsymbol{z}) = \mu\,\pi_\tau - \mu\,\big[d(D)\big]_{\boldsymbol{z}}^{\boldsymbol{y}^\tau} + \theta^d_\tau\, d_\tau(\boldsymbol{y}^\tau) - \theta^u_\tau,
\]
where the term $\theta^d_\tau\, d_\tau(\boldsymbol{y}^\tau)$ is present only if $\tau \notin D$ (for $\tau \in D$, the discharge $d_\tau$ is already counted in $d(D)$), and $[\,\cdot\,]_{\boldsymbol{z}}^{\boldsymbol{y}^\tau}$ denotes the difference between the values at $\boldsymbol{y}^\tau$ and at $\boldsymbol{z}$.
\begin{itemize}
    \item If $\tau \in \Omega \setminus D$, then $d(D) = 0$ in both points and $\theta^d_\tau = 0$ by Step~3, so $\theta^u_\tau = \mu\,\pi_\tau$.
    \item If $\tau \in D$, then $d(D)$ increases by $d_\tau(\boldsymbol{y}^\tau) = \pi_\tau$ (Step~2), so $\theta^u_\tau = \mu\,\pi_\tau - \mu\,\pi_\tau = 0$.
\end{itemize}

\emph{Step 5: we determine the right-hand side and conclude.} The right-hand side follows by evaluating the equation at any point of the face. Evaluating at $\boldsymbol{z}$, where $c(\Omega) = f(\Omega)$, $d(D) = 0$, $u_\tau = 1$ for $\tau \in \Omega$, and all other coordinates have zero coefficients, gives
\[
\theta_0 = \boldsymbol{\theta}^\top \boldsymbol{z} = \mu\,f(\Omega) + \mu \sum_{\tau \in \Omega \setminus D} \pi_\tau,
\]
which is $\mu$ times the right-hand side of~\eqref{eq:VIc_D} written as $c(\Omega) - d(D) + \sum_{\tau \in \Omega \setminus D} \pi_\tau u_\tau \leq f(\Omega) + \sum_{\tau \in \Omega \setminus D} \pi_\tau$. Together with Steps~0--4, this shows that $(\boldsymbol{\theta}, \theta_0)$ is a multiple of the coefficients of the cut.
\hfill$\square$

\section{Supplementary Material for Section~\ref{sec:dominance}}\label{ec:sec_dominance}

This section complements Section~\ref{sec:dominance}. Recall that $\underline{\sigma}_t$ and $\overline{\sigma}_t$ are the reachable SoC bounds~\eqref{eq:sigma_bounds}, and that $D_\Omega$ is the set~\eqref{eq:D_Omega} of the periods with $\pi_\tau = \overline{d}_\tau$.

\subsection{Proof of Proposition~\ref{prop:reachable_soc_bounds}}\label{ec:proof_reachable}

\begin{proof}{Proof of Proposition~\ref{prop:reachable_soc_bounds}}
Let $\Omega = [1,t]$ and $D = [1,t-1]$, so that $\Omega \setminus D = \{t\}$. The greedy rule for $[1,q]$ charges as much as possible from the start, so $f([1,q]) = \overline{\sigma}_q - s_0$ for every $q$; hence $f(\Omega) = \overline{\sigma}_t - s_0$ and $\pi_t = -f(t \mid [1,t-1]) = -\big(f([1,t]) - f([1,t-1])\big) = -(\overline{\sigma}_t - \overline{\sigma}_{t-1})$. Substituting in~\eqref{eq:VIc_D},
\[
c([1,t]) - d([1,t-1]) \;\leq\; \overline{\sigma}_t - s_0 - (\overline{\sigma}_t - \overline{\sigma}_{t-1})\,(1 - u_t) \;=\; \overline{\sigma}_{t-1}\,(1 - u_t) + \overline{\sigma}_t\, u_t - s_0,
\]
which is~\eqref{eq:U_star} after replacing $c([1,t]) - d([1,t-1])$ by $s_{t-1} + c_t - s_0$. The same computation for the mirrored unit gives~\eqref{eq:L_star}. Validity thus follows from Proposition~\ref{prop:cut_strengthening}. It can also be seen directly: if $u_t = 1$, then $d_t = 0$ and $s_{t-1} + c_t = s_t \leq \overline{\sigma}_t$; if $u_t = 0$, then $c_t = 0$ and $s_{t-1} + c_t = s_{t-1} \leq \overline{\sigma}_{t-1}$.

For $u_t \in [0,1]$, the right-hand side of~\eqref{eq:U_star} is a convex combination of $\overline{\sigma}_{t-1}$ and $\overline{\sigma}_t$, and $\overline{\sigma}_{t-1} \leq \overline{\sigma}_t \leq \overline{s}$. Hence, it is at most $\overline{s}$, which gives the dominance, and it equals $\overline{s}$ for every $u_t \in [0,1]$ if and only if $\overline{\sigma}_{t-1} = \overline{s}$. The claims for~\eqref{eq:L_star} are symmetric.
\hfill$\square$
\end{proof}

\subsection{An Instance in Which the Two Ways of Crediting Are Incomparable}\label{ec:example_incomparable}

Let $T = 2$, $[\underline{s}, \overline{s}] = [0,1]$, $s_0 = 0.4$, $\overline{c} = 0.7$, $\overline{d} = 0.3$, and $\Omega = \{1,2\}$. Then $f(\Omega) = 0.6$, $f(\{2\}) = 0.7$, and $f(\{1\}) = 0.6$, so $\pi_1 = 0.1$, $\pi_2 = 0$, and $D_\Omega = \emptyset$, because $\overline{d}_1 = 0.3 \neq \pi_1$. With $D = D_\Omega$, the cut reads $c_1 + c_2 \leq 0.6 + 0.1\,(1 - u_1)$, whereas~\eqref{eq:U_star} for $t = 2$ reads $c_1 + c_2 - d_1 \leq 0.6$, since $\overline{\sigma}_1 = \overline{\sigma}_2 = 1$. The point of $\mathcal{S}^{\mathcal{R}}$ with $\boldsymbol{u} = (0.2, 1)$, $\boldsymbol{c} = (0.14, 0.7)$, and $\boldsymbol{d} = (0.24, 0)$ satisfies~\eqref{eq:U_star} ($0.84 - 0.24 \leq 0.6$) but violates the cut with $D = D_\Omega$ ($0.84 > 0.68$): the SoC bound credits the full simultaneous discharge $d_1 = 0.24$, whereas the cut only credits $\pi_1 (1 - u_1) = 0.08$. Conversely, the point with $\boldsymbol{u} = (0.2, 0.7)$, $\boldsymbol{c} = (0.14, 0.49)$, and $\boldsymbol{d} = (0, 0.09)$ satisfies the cut with $D = D_\Omega$ ($0.63 \leq 0.68$) but violates~\eqref{eq:U_star} ($0.63 > 0.6$): the cut credits the fractional mode of period~1 although nothing is discharged there. Both points exploit the fractional modes that the linear relaxation allows and the inequalities try to prevent: the first one charges and discharges in period~1 at the same time, and the second one charges in period~1 with $u_1 = 0.2$ and charges and discharges in period~2 at the same time. Both inequalities define facets: here $D^+_\Omega = \{1\}$, so $D = \emptyset$ and $D = \{1\}$, which gives~\eqref{eq:U_star}, both lie in the range~\eqref{eq:D_Omega}, and (C1)--(C4) hold for both.

\section{Supplementary Material for Section~\ref{sec:separation_algorithm}}\label{ec:proofs_separation}

This section complements Section~\ref{sec:separation_algorithm}. Recall that $\boldsymbol{y}^* = (\boldsymbol{d}^*, \boldsymbol{c}^*, \boldsymbol{u}^*)$ is the point to be separated, $\mathcal{J}(\boldsymbol{y}^*)$ the set of periods in which it charges and discharges simultaneously, and $v_{\Omega,D}(\boldsymbol{y}^*)$ the violation~\eqref{eq:violation} of the cut for $\Omega$ and $D$ at $\boldsymbol{y}^*$.

\subsection{Proof of Proposition~\ref{prop:two_passes}}\label{ec:proof_two_passes}

\begin{proof}{Proof of Proposition~\ref{prop:two_passes}}
Let $\tau \in \Omega$. We apply Lemma~\ref{lem:marginal_value} to $B = \Omega \setminus \{\tau\}$ and $t = \tau$. The greedy rules for $\Omega$ and for $\Omega \setminus \{\tau\}$ make the same decisions before $\tau$, so $s^{(B)}_{\tau-1} = s_{\tau-1}$; and the threshold at the end of $\tau$ only depends on the periods after $\tau$, so $\kappa^B_\tau = \kappa_\tau$. Since $\pi_\tau = f(\Omega \setminus \{\tau\}) - f(\Omega) = -\big[f(B \cup \{\tau\}) - f(B)\big]$ and $x_\tau = \phi^+(s_{\tau-1}) - s_{\tau-1}$, the first expression in~\eqref{eq:marginal_t} gives~\eqref{eq:pi_two_passes}. The forward pass yields $s_{\tau-1}$, $x_\tau$, and $f(\Omega) = x(\Omega)$; it also yields the SoC $s_t$ at every gap $t \in [l, k-1] \setminus \Omega$, which is all that (C1) and (C2) need, together with the suffix sums $\sum_{q \in \Omega, q > t} \overline{c}_q$. The backward pass yields $\kappa_\tau$, and $D_\Omega$ and $D^+_\Omega$ are read from~\eqref{eq:pi_two_passes} and the limits $\overline{d}_\tau$ of~\eqref{eq:singleton_cuts}.
\hfill$\square$
\end{proof}

\subsection{Proofs of Lemmas~\ref{lem:decomposition_moves} and~\ref{lem:free_removal}}\label{ec:proof_moves}

\begin{proof}{Proof of Lemma~\ref{lem:decomposition_moves}}
By Lemma~\ref{lem:decomposition} (see the proof of Proposition~\ref{prop:necessity}), in cases (a) and (b) the cut for $\Omega$ is, coefficient by coefficient, the sum of the cut for $\Omega_1 := \Omega \cap [l,t]$ (absent if $\Omega_1 = \emptyset$) and the singleton cuts of $\Omega \cap [t+1,k]$; in case (c) it is the sum of the cuts for $\Omega \cap [l,t]$ and $\Omega \cap [t+1,k]$, and in all cases the coefficients $\pi_\tau$ of the parts coincide with those of $\Omega$. The violation of a sum of inequalities is the sum of the violations, the singleton cuts are satisfied by $\boldsymbol{y}^*$, and the pointwise-optimal $D$ of each part is the restriction of that of $\Omega$ because the gains coincide. The same argument applies to the robust violation, which is also a sum of terms of the individual periods: the term of a period $q$ with $\pi_q = -\overline{c}_q$, which charges $\overline{c}_q$ in the greedy point, is $c^*_q - \overline{c}_q\,(1 - d^*_q/\overline{d}_q)$, and it is nonpositive because $c^*_q/\overline{c}_q \leq u^*_q \leq 1 - d^*_q/\overline{d}_q$ by the singleton cuts.
\hfill$\square$
\end{proof}

\begin{proof}{Proof of Lemma~\ref{lem:free_removal}}
Write $g_\tau(p) := \min\{d^*_\tau,\ p\,(1 - u^*_\tau)\}$, so that, with the pointwise-optimal $D$, $v_\Omega = \sum_{\tau \in \Omega} c^*_\tau - f(\Omega) - \sum_{\tau \in \Omega} g_\tau(\pi_\tau(\Omega))$; note that $g_\tau$ is nondecreasing. For $\tau \in \Omega \setminus \{t\}$, submodularity gives $f(\tau \mid (\Omega \setminus \{t\}) \setminus \{\tau\}) \geq f(\tau \mid \Omega \setminus \{\tau\})$, i.e., $\pi_\tau(\Omega \setminus \{t\}) \leq \pi_\tau(\Omega)$, hence $g_\tau(\pi_\tau(\Omega \setminus \{t\})) \leq g_\tau(\pi_\tau(\Omega))$. Since $f(\Omega \setminus \{t\}) = f(\Omega) + \pi_t$ and $c^*_t = 0$,
\begin{align*}
v_{\Omega \setminus \{t\}} - v_\Omega &= g_t(\pi_t) - \pi_t + \sum_{\tau \in \Omega \setminus \{t\}} \big[g_\tau(\pi_\tau(\Omega)) - g_\tau(\pi_\tau(\Omega \setminus \{t\}))\big] \geq g_t(\pi_t) - \pi_t.
\end{align*}
If $\pi_t \leq 0$, then $g_t(\pi_t) = \pi_t(1 - u^*_t)$ and $g_t(\pi_t) - \pi_t = -\pi_t u^*_t \geq 0$. If $u^*_t = 0$ and $d^*_t \geq \pi_t$, then $g_t(\pi_t) = \min\{d^*_t, \pi_t\} = \pi_t$. The robust violation is handled in the same way, with $1 - u^*_\tau$ replaced by $1 - c^*_\tau/\overline{c}_\tau$ or $d^*_\tau/\overline{d}_\tau$ according to the sign of $\pi_\tau$; both replacements are nonnegative, so $g_\tau$ remains nondecreasing.
\hfill$\square$
\end{proof}

The next lemma justifies restricting the search to the sets that meet $\mathcal{J}(\boldsymbol{y}^*)$. In a set that does not meet it, no period charges and discharges at the same time, so the modes of $\boldsymbol{y}^*$ can be rounded, within the set, to binary modes that are compatible with its power profile; the cut is valid for the rounded point, and the robust credits are at least the credits of the rounded modes.

\begin{lemma}\label{lem:robust_no_J}
Let $\boldsymbol{y}^* \in \mathcal{S}^{\mathcal{R}}_1$ and $D \subseteq \Omega \subseteq [T]$ with $\Omega \cap \mathcal{J}(\boldsymbol{y}^*) = \emptyset$. Then the robust violation~\eqref{eq:robust_D} satisfies $\bar{v}_{\Omega,D}(\boldsymbol{y}^*) \leq 0$.
\end{lemma}

\begin{proof}{Proof of Lemma~\ref{lem:robust_no_J}}
By~\eqref{eq:robust_D}, $\bar{v}_{\Omega,D}(\boldsymbol{y}^*) = c^*(\Omega) - d^*(D) - \sum_{\tau \in \Omega \setminus D} \bar{g}_\tau - f(\Omega)$, where the robust credit of $\tau \in \Omega \setminus D$ is $\bar{g}_\tau := \pi_\tau\,(1 - c^*_\tau/\overline{c}_\tau)$ if $\pi_\tau > 0$ and $\bar{g}_\tau := \pi_\tau\, d^*_\tau/\overline{d}_\tau$ if $\pi_\tau \leq 0$. Recall that $0 \leq c^*_\tau \leq \overline{c}_\tau$ and $0 \leq d^*_\tau \leq \overline{d}_\tau$ by the singleton cuts. Since $\Omega$ does not meet $\mathcal{J}(\boldsymbol{y}^*)$, every $\tau \in \Omega$ has $c^*_\tau = 0$ or $d^*_\tau = 0$.

\emph{Step 1: we round the modes.} For $\tau \in \Omega$, let $\bar{u}_\tau := 1$ if $c^*_\tau > 0$, or if $c^*_\tau = d^*_\tau = 0$ and $\pi_\tau \leq 0$, and let $\bar{u}_\tau := 0$ otherwise. Then $d^*_\tau = 0$ whenever $\bar{u}_\tau = 1$, and $c^*_\tau = 0$ whenever $\bar{u}_\tau = 0$.

\emph{Step 2: we show that $\bar{g}_\tau \geq \pi_\tau\,(1 - \bar{u}_\tau)$ for every $\tau \in \Omega \setminus D$.}
\begin{itemize}
    \item If $\bar{u}_\tau = 1$, the right-hand side is $0$. If $\pi_\tau > 0$, then $\bar{g}_\tau \geq 0$ because $c^*_\tau \leq \overline{c}_\tau$; if $\pi_\tau \leq 0$, then $d^*_\tau = 0$ and $\bar{g}_\tau = 0$.
    \item If $\bar{u}_\tau = 0$, the right-hand side is $\pi_\tau$, and $c^*_\tau = 0$. If $\pi_\tau > 0$, then $\bar{g}_\tau = \pi_\tau$; if $\pi_\tau \leq 0$, then $\bar{g}_\tau = \pi_\tau\, d^*_\tau/\overline{d}_\tau \geq \pi_\tau$ because $d^*_\tau \leq \overline{d}_\tau$.
\end{itemize}

\emph{Step 3: we apply the proof of validity to the rounded modes.} By Step~2,
\[
\bar{v}_{\Omega,D}(\boldsymbol{y}^*) \leq c^*(\Omega) - d^*(D) - \sum_{\tau \in \Omega \setminus D} \pi_\tau\,(1 - \bar{u}_\tau) - f(\Omega).
\]
Let $B := \{\tau \in \Omega : \bar{u}_\tau = 1\} \cup D$ and $\boldsymbol{x}^* := \boldsymbol{c}^* - \boldsymbol{d}^*$. By Step~1, $x^*_\tau = c^*_\tau$ if $\bar{u}_\tau = 1$, and $x^*_\tau = -d^*_\tau$ and $c^*_\tau = 0$ if $\bar{u}_\tau = 0$. Hence, as in the proof of Proposition~\ref{prop:cut_strengthening}, $x^*(B) = c^*(\Omega) - d^*(D)$. Since $\boldsymbol{x}^* \in \mathcal{X}_{[T]}(s_0)$, Proposition~\ref{prop:submodularity}(i) gives $x^*(B) \leq f(B)$, and Lemma~\ref{lem:submodular_bound}, applied to $B \subseteq \Omega$, gives $f(B) \leq f(\Omega) + \sum_{\tau \in \Omega \setminus B} \pi_\tau = f(\Omega) + \sum_{\tau \in \Omega \setminus D} \pi_\tau\,(1 - \bar{u}_\tau)$, because $\Omega \setminus B = \{\tau \in \Omega \setminus D : \bar{u}_\tau = 0\}$. Combining the three relations yields $\bar{v}_{\Omega,D}(\boldsymbol{y}^*) \leq 0$.
\hfill$\square$
\end{proof}

\subsection{Proof of Proposition~\ref{prop:separation_properties}}\label{ec:proof_separation_properties}

\begin{proof}{Proof of Proposition~\ref{prop:separation_properties}}
(i) follows from Proposition~\ref{prop:cut_strengthening} and the acceptance test. (ii) holds because Stage~1 evaluates every scanned window with the pointwise-optimal $D$ of~\eqref{eq:violation}, which maximizes the violation over all $D \subseteq \Omega$, so a window admits a violated cut only if this one is violated. (iii) follows from Lemmas~\ref{lem:decomposition_moves} and~\ref{lem:free_removal}, applied to the robust violation along the chain of moves of \textsc{Settle}: by these lemmas, the robust violation of $\Omega$ is at most the sum of those of the parts into which the moves split it. The parts that \textsc{Settle} discards have a nonpositive robust violation, because they are singletons, whose cuts are singleton cuts, sets at which (C1) fails at $l-1$ (Lemma~\ref{lem:decomposition_moves}(a)), or sets that do not meet $\mathcal{J}(\boldsymbol{y}^*)$ (Lemma~\ref{lem:robust_no_J}). Hence, the largest robust violation among the sets at which \textsc{Settle} stops is at least that of $\Omega$ divided by their number. (iv) follows from the memoization and Proposition~\ref{prop:two_passes}, and from counting the sets visited. Each move of \textsc{Settle} removes at least one period or splits the set into two disjoint parts, so \textsc{Settle} visits $\mathcal{O}(T)$ sets from each window. Stage~3 has at most $(n + |\mathcal{J}(\boldsymbol{y}^*)|)\,|\mathcal{J}(\boldsymbol{y}^*)|$ starting sets, because \textsc{Settle}$(\mathcal{E}_t)$ returns disjoint sets that each contain a period of $\mathcal{J}(\boldsymbol{y}^*)$. From each of them, the total number of periods of the sets to explore never increases, and it decreases by one with each removal, so there are at most $T$ removals, each comparing at most $T$ candidate sets.
\hfill$\square$
\end{proof}

\subsection{The Lifting Coefficient: Proofs, an Example, and Sequential Lifting}\label{ec:lifting}

\begin{proof}{Proof of Proposition~\ref{prop:lifting}}
Let $\boldsymbol{y} \in \mathcal{P}^{0\text{-}1}$ with net charging $\boldsymbol{x}$. If $u_t = 0$, then $c_t = 0$ and~\eqref{eq:lifted_cut} reduces to~\eqref{eq:VIc_D}, which is valid. Let $u_t = 1$, so that $c_t = x_t \geq 0$, and let $R := \{\tau \in \Omega \setminus D : u_\tau = 0\}$ be the set of periods that $\boldsymbol{y}$ relaxes. As in~\eqref{eq:cut_chain}, $c(\Omega) - d(D) = x(B_R)$ and $\sum_{\Omega \setminus D} \pi_\tau(1 - u_\tau) = \sum_{R} \pi_\tau$, so~\eqref{eq:lifted_cut} reads
\begin{equation}\label{eq:lifted_pattern}
x(B_R) + \alpha\,x_t \leq f(\Omega) + \sum_{\tau \in R} \pi_\tau = \xi(R) + f(B_R).
\end{equation}
The vector $\boldsymbol{x}$ belongs to $\mathcal{X}^R_t$. Conversely, every $\boldsymbol{x} \in \mathcal{X}^R_t$ is the net charging of a point of $\mathcal{P}^{0\text{-}1}$ with $u_t = 1$ that relaxes exactly $R$: take $u_\tau = 0$ on $R$, $u_\tau = 1$ on $(\Omega \setminus (D \cup R)) \cup \{t\}$, and the mode given by the sign of $x_\tau$ elsewhere. Hence~\eqref{eq:lifted_cut} is valid if and only if, for every $R \subseteq \Omega \setminus D$,
\begin{equation}\label{eq:lifting_condition}
M_R(\alpha) := \max\big\{x(B_R) + \alpha\,x_t : \boldsymbol{x} \in \mathcal{X}^R_t\big\} \leq \xi(R) + f(B_R).
\end{equation}
The maximum of a linear function with nonnegative weights over a g-polymatroid is attained by processing the elements in nonincreasing order of their weights and setting, for each $i$, the total value of the first $i$ elements of this order equal to the rank of the set they form~\citep[Chapter~14]{frank2011connections}. With weight $1$ on $B_R$ and $\alpha$ on $t$, this gives
\[
M_R(\alpha) =
\begin{cases}
b^R_t(B_R) + \alpha\,\big[b^R_t(B_R \cup \{t\}) - b^R_t(B_R)\big] & \text{if } 0 \leq \alpha \leq 1,\\[2pt]
\alpha\,b^R_t(\{t\}) + \big[b^R_t(B_R \cup \{t\}) - b^R_t(\{t\})\big] & \text{if } \alpha \geq 1.
\end{cases}
\]
Moreover, $b^R_t(B_R \cup \{t\}) = f(B_R \cup \{t\})$. Indeed, $\mathcal{X}^R_t \subseteq \mathcal{X}_{[T]}(s_0)$, so $b^R_t(A) \leq b(A; s_0) = f(A)$ for every $A \subseteq [T]$, where $b(\cdot\,; s_0)$ is the rank function~\eqref{eq:b_parameterized} of $\mathcal{X}_{[T]}(s_0)$ and the equality holds by Lemma~\ref{lem:greedy_polymatroid}(i); conversely, the greedy rule for $B_R \cup \{t\}$ charges in $B_R \cup \{t\} \supseteq (\Omega \setminus (D \cup R)) \cup \{t\}$ and discharges in $R$, so its net charging belongs to $\mathcal{X}^R_t$ and attains $f(B_R \cup \{t\})$. Substituting the definitions~\eqref{eq:lifting_quantities}, condition~\eqref{eq:lifting_condition} becomes $\alpha\,\gamma^R_t \leq \xi_t(R)$ for $\alpha \leq 1$ and $(\alpha - 1)\,\overline{c}^{\,R}_t \leq \xi_t(R) - \gamma^R_t$ for $\alpha \geq 1$. Since $M_R$ is convex and nondecreasing in $\alpha$ with $M_R(0) = b^R_t(B_R) \leq \xi(R) + f(B_R)$, the set of admissible $\alpha$ is the interval $[0, \alpha_R]$ with $\alpha_R$ as in~\eqref{eq:alpha_R}, which proves the first claim. Next, $\alpha^*_t \leq \alpha_\emptyset < \infty$. Indeed, by~\eqref{eq:alpha_R}, $\alpha_R = +\infty$ only if $\overline{c}^{\,R}_t = 0$, and $\overline{c}^{\,\emptyset}_t = b^\emptyset_t(\{t\}) \geq \min\{\overline{c}, \overline{s} - s_0\} > 0$: the sequence that idles in every period except $t$, where it charges $\min\{\overline{c}, \overline{s} - s_0\}$, belongs to $\mathcal{X}^\emptyset_t$, because idling satisfies all the sign constraints of $\mathcal{X}^\emptyset_t$, and the last inequality holds under Assumption~\ref{ass:interior_s0}. The bound on $\alpha_R$ follows from $\gamma^R_t \leq \overline{c}_t$ and $\xi_t(R) \geq \xi(R)$. Finally, $R \mapsto \xi(R)$ is supermodular, as $-f(\Omega \setminus R)$ is supermodular in $R$ and the other terms are modular, and it vanishes on $\emptyset$ and on the singletons; a supermodular function with these properties is nondecreasing.
\hfill$\square$
\end{proof}

\begin{proof}{Proof of Lemma~\ref{lem:C3_equiv}}
(a)$\Leftrightarrow$(b) is the last claim of Lemma~\ref{lem:C3_thresholds}. (a)$\Leftrightarrow$(c): for every $R \in \mathcal{R}_\Omega$, Lemma~\ref{lem:marginal_value}, applied to $B = B_R$ and to $t \notin B_R$, gives $f(B_R \cup \{t\}) > f(B_R)$ if and only if $\sigma < \kappa$, with $\sigma = s^R_{t-1}$ and $\kappa = \kappa^R_t$. (c)$\Leftrightarrow$(d): by~\eqref{eq:alpha_R}, $\alpha^*_t = 0$ if and only if some $R$ has $\xi_t(R) = 0 < \gamma^R_t$, i.e., $\xi(R) = 0$, $b^R_t(B_R) = f(B_R)$, and $f(B_R \cup \{t\}) > f(B_R)$. Given (c), $\sigma < \kappa$, and then idling in $t$ instead of discharging raises the SoC at the end of $t$ from $\phi^-(\sigma)$ to $\sigma \leq \kappa$, which by~\eqref{eq:V_difference} costs no charge of $B_R$, so $b^R_t(B_R) = f(B_R)$ and (d) holds; the converse is immediate.
\hfill$\square$
\end{proof}

\paragraph{A binding pattern outside $\mathcal{R}_\Omega$.} The binding pattern of the lifting coefficient need not lie in $\mathcal{R}_\Omega$. For $T = 24$, $[\underline{s}, \overline{s}] = [0,1]$, $s_0 = 0.5$, $\overline{c} = 4/15$, $\overline{d} = 1/2$, the discharging cut for $\Omega = \{15, 17, 18, 19\}$ reads $d_{15} + d_{17} + d_{18} + d_{19} \leq 19/15 - \tfrac{4}{15}\,u_{15} + \tfrac{7}{30}\,(u_{17} + u_{18})$, and (C3) fails at $t = 16$ for the mirrored unit. The lifting coefficient of $d_{16}$ is $\alpha^*_{16} = 7/8$, attained at $R = \{15, 18\}$, for which $\xi(R) = 7/30 > 0$: the binding schedule charges in~15 (storage full), discharges only $8/30$ in~16, discharges $1/2$ in~17, charges $4/15$ in~18, and discharges $1/2$ in~19. Any larger coefficient, e.g.\ $14/15$, is cut off by this schedule.

\paragraph{Sequential lifting.} Several periods can be lifted sequentially. After $c_{t_1}$ has been lifted with coefficient $\alpha_1$, the coefficient of a second period $t_2$ is obtained by the argument of the proof of Proposition~\ref{prop:lifting}, but not always by its closed form~\eqref{eq:alpha_R}. The lifted cut must be valid for every pattern $R$ and both modes of $t_1$, so the coefficient of $t_2$ is the minimum, over $R \subseteq \Omega \setminus D$ and $u_{t_1} \in \{0,1\}$, of the largest admissible $\alpha$ for each pair. If $u_{t_1} = 0$, then $c_{t_1} = 0$, the lifted term vanishes, and~\eqref{eq:alpha_R} applies with the sign constraint $x_{t_1} \leq 0$ added to $\mathcal{X}^R_{t_2}$. If $u_{t_1} = 1$, the maximum in~\eqref{eq:lifting_condition}, now with the term $\alpha_1 x_{t_1}$ added to the objective and $x_{t_1} \geq 0$, is attained by the greedy rule that processes the weight classes $\{1, \alpha_1, \alpha\}$ in nonincreasing order; it is a convex, piecewise linear function of $\alpha$ with up to three pieces, so the largest admissible $\alpha$ is found by checking the pieces in turn. The resulting coefficient depends on the order in which the periods are lifted, as is usual in sequential lifting.

\paragraph{The strengthening stage.} Algorithm~\ref{alg:strengthening} summarizes the stage described in Section~\ref{sec:strengthening}.

\begin{algorithm}[htpb]
\caption{Strengthening stage}
\label{alg:strengthening}
\begin{algorithmic}[1]
\Require $\boldsymbol{y}^*$ with $\mathcal{J}(\boldsymbol{y}^*) \neq \emptyset$, the pool of cuts $\mathcal{K}$ of the relaxation.
\For{each cut $(\Omega, D)$ of $\mathcal{K}$ that is tight at $\boldsymbol{y}^*$} \Comment{a discharging cut is handled as the charging cut of the mirrored unit (Lemma~\ref{lem:reflection})}
    \For{each $t \notin \Omega$ not yet lifted in the cut, with $c^*_t > 0$ (resp.\ $d^*_t > 0$ for a discharging cut)}
        \State compute $\alpha^*_t$ by~\eqref{eq:alpha_R} or, if the cut has already been lifted, by the procedure for sequential lifting, pruning with $\xi$
    \EndFor
    \State if some $\alpha^*_t > 0$, lift the period with the largest $\alpha^*_t c^*_t$, add the lifted cut, and repeat with the lifted cut as the base
\EndFor
\end{algorithmic}
\end{algorithm}

\end{APPENDICES}
\end{document}